\documentclass[11pt]{amsart}
\usepackage{amsmath,amsthm,amssymb}
\usepackage[margin=1in]{geometry}
\usepackage{amsmath, bbm}
\usepackage{mathtools}
\usepackage{amsxtra}
\usepackage{amsthm}
\usepackage{amsfonts}
\usepackage{amssymb}
\usepackage{bm}
\usepackage{cases}
\usepackage{xcolor}
\usepackage{mathrsfs}  
\usepackage{enumerate}
\usepackage{url}
\usepackage{pdfpages}
\usepackage{cancel}
\usepackage[normalem]{ulem}
\usepackage[numbers,sort]{natbib}
\usepackage[pagebackref=true, colorlinks=true, citecolor=blue]{hyperref}
\numberwithin{equation}{section}

\theoremstyle{definition}

\newtheorem{definition}{Definition}
\theoremstyle{theorem}
\newtheorem{remark}{Remark}[section]
\newtheorem{theorem}{Theorem}[section]
\newtheorem{lemma}[theorem]{Lemma}

\newtheorem{proposition}[theorem]{Proposition}
\newtheorem{corollary}[theorem]{Corollary}

\newcommand{\eps}{\epsilon}
\newcommand{\veps}{\varepsilon}

\newcommand{\E}{\mathbb{E}}
\newcommand{\F}{\mathcal{F}}
\newcommand{\R}{\mathbb{R}}
\newcommand{\Z}{\mathbb{Z}}
\renewcommand{\P}{\mathbb{P}}

\newcommand{\N}{\mathbb{N}}
\newcommand{\T}{\mathbb{T}}
\newcommand{\D}{\mathcal{D}}

\newcommand{\TP}{\widetilde{\P}}
\newcommand{\TE}{\widetilde{\E}}
\newcommand{\TOmega}{\widetilde{\Omega}}
\newcommand{\TF}{\widetilde{\F}}
\newcommand{\Tu}{\widetilde{u}}

\newcommand{\Tpi}{\widetilde{\pi}}
\newcommand{\TS}{\widetilde{\mathcal{S}}}

\renewcommand{\Pr}{\text{Pr}} 
\newcommand{\cred}{\color{red}}
\newcommand{\cblue}{\color{blue}}

\newcommand{\lp}{\left(}
\newcommand{\rp}{\right)}

\renewcommand{\Pr}{\text{Pr}}

\title[Well-posedness and invariant measures of stochastic KdV equations]{Well-posedness and existence of an invariant measure for the linearly-damped KdV equations driven by a jump noise}
\date{}
\author[K. Tawri, R. Temam, X. Yang]{Krutika Tawri$^1$
\and
 Roger Temam$^2$ \and Xinwu Yang$^2$ }
 
 \address{\newline	$^1$ Department of Applied Mathematics, University of Washington Seattle, WA, USA.
 \newline	$^2$ Department of Mathematics, Indiana University Bloomington, IN, USA.}
 
 \email{ktawri@uw.edu (Krutika Tawri), temam@iu.edu (Roger Temam), xinwyang@iu.edu (Xinwu Yang)}
\begin{document}
\maketitle
\begin{abstract}
In this paper, we investigate the linearly damped KdV equation on the one-dimensional torus $\T$, perturbed by a multiplicative L\'{e}vy noise. For any damping coefficient $\gamma > 0$, we establish the existence and uniqueness of a pathwise weak solution with values in $H^2(\T)$.

In the second part of the paper, we analyze the long-time behavior of these solutions. This study is particularly subtle as the presence of jumps in time can significantly influence the asymptotics. We show, using the techniques of Maslowski and Seidler, that, provided the frictional damping coefficient $\gamma > 0$ is sufficiently large, the system influenced by square-integrable jumps admits an invariant measure in $H^2(\T)$.

\end{abstract}

\section{Introduction}
This article studies the linearly damped Korteweg-de Veries (KdV) equations driven by a L\'evy noise. The KdV system, {describing the behavior of shallow water wave propagation on the one-dimensional periodic domain $\T$ is given by,
\begin{equation}
	    \begin{cases}  \label{eqn: skdv with K}
		    du(t) + (\partial_x^3 u(t) +  u(t) \partial_x u(t) + \gamma u(t))dt \\
		    \ \ \ \ \ \ \ \ \ \ \ \ \ \ \ \ \ = G(u(t-))  dW(t) +  \int_{E_0} K(u(t-), \xi) d\widehat{\pi}(t,\xi) + \int_{E \setminus E_0} \mathscr{K}(u(t-), \xi) d\pi(t,\xi), \\
		     u(0) = u_0,
		    \end{cases}
\end{equation}
where the unknown {$u$ describes the elongation of the wave}. 
We will assume that the damping coefficient $\gamma$ is a strictly positive constant.  The stochastic integrals on the right-hand side represent different effects of a L\'evy noise: The term $G(u)dW$ influences the system continuously in time whereas the other two terms $\int_{E_0} K(u(t-), \xi) d\widehat{\pi}(t,\xi)$ and $\int_{E \setminus E_0} \mathscr{K}(u(t-), \xi) d\pi(t,\xi)$ impact the system as discrete impulses in time.

The KdV equation was originally derived as an ideal model of propogation of unidirectional weakly nonlinear waves in a channel. In realistic environments, the waves may interact with the medium that is not homogeneous. Therefore, it is natural to consider the energy dissipation and the external excitation. In this direction, we can consider dissipative effects due to friction via the linear damping term, and
	the effects of noisy environment via stochastic external forcing, which can be interpreted as random fluctuations in the medium through which the weakly nonlinear dispersive waves propagate (see, e.g., \cite{chang1986, SRC1998, Herman1990}). Such fluctuation may not be purely Gaussian as the medium may include some intermittent or impulsive disturbances. To capture such effects, we model the external excitation by a L\'{e}vy noise, which generalizes Wiener noise by allowing for jumps and perturbations acting on the waves.

The aim of this study is two-fold:
\begin{enumerate}
	\item We prove the global-in-time existence and uniqueness of pathwise {weak} solutions to \eqref{eqn: skdv with K}, 
	\item We prove the existence of an invariant measure for the dynamics of \eqref{eqn: skdv with K}$_1$ with $\mathscr{K}=0$.
\end{enumerate}

Before we lay out our mathematical framework and state our main theorems, we will give a brief overview of the existing literature in this direction.
The study of well-posedness of the deterministic KdV equations and its variants on the entire real line or a periodic domain has been the focus of many works; see e.g. \cite{Tem69,BS75,BS76,ST76,Kato79, Bourgain93,KPV96, CKSTT03}. The long-time behavior of the damped deterministic system corresponding to \eqref{eqn: skdv with K} has also been well-studied, e.g. in \cite{Ghi88} and \cite{MR97}, {where the existence of a global attractor and its regularity are established respectively}. We also refer the reader to \cite[Chapter IV, Section 7]{Tem97} for a detailed discussion on the global attractors for the deterministic KdV equations.   
There has also been extensive work done to understand the stochastic KdV equations under the influence of Gaussian noise.
 In \cite{DeBussche98}, the undamped system (i.e. $\gamma = 0$) on the real line is considered for both additive and multiplicative Wiener noise ($K=\mathscr{K}=0$) cases: In the case of $H^{1}(\mathbb{R})$ -valued additive noise, 
global pathwise well-posedness is established in $H^{1}(\mathbb{R})$; whereas in the case of multiplicative noise, only the existence of martingale solutions in $L^{2}(\mathbb{R})$ is established. 
In \cite{dBDT99}, the undamped KdV system, driven by a localized space--time white noise on $\T$, is studied wherein the authors  prove the existence of local-in-time unique pathwise solution taking values in $H^{s}(\T)$ for $-5/8 < s \le 0$. Furthermore, in the same work, global well-posedness of the system in $L^{2}(\T)$ is proven assuming that the noise is $L^{2}(\T)$-valued. 
In \cite{dBDT04}, the KdV equation without damping is studied on a torus under the influence of an additive, white-in-time and ``almost white-in-space'' noise. 
 In \cite{GMR21}, the damped stochastic KdV equation is studied under additive Wiener noise that is white-in-time and smooth-in-space, and global pathwise well-posedness is established in $H^{m}(\T)$ for all $m \geq 2$.
  Beyond well-posedness, the long-time statistical behavior of the damped stochastic KdV, driven by spatially smooth additive Wiener noise, has also
  been investigated in several works.  The existence of invariant measures in the phase space $H^1(\R)$ for the dynamics of \eqref{eqn: skdv with K} with $K=\mathscr{K}=0$, is established in
  \cite{EKZ18}; while in the periodic setting, the existence and uniqueness of invariant measures in $H^m(\T)$ for $m \geq 2$ is given in \cite{GMR21}, where the uniqueness of the invariant measure is proved by an asymptotic coupling strategy originially developed in earlier works such as \cite{GMR17}. To the best of our knowledge, this is the first work that studies the well-posedness and existence of invariant measure for the damped KdV equation driven by a L\'evy noise.

Our main results are Theorem \ref{Thm: summary existence} and Theorem \ref{Thm: summary invariant}:
\begin{itemize} 
	\item In Theorem \ref{Thm: summary existence} we establish the existence and uniqueness of pathwise solution to the equation \eqref{eqn: skdv with K} for a general L\'evy noise. In order to prove Theorem \ref{Thm: summary existence}, we first consider a simplified system where the effects of large jumps are neglected i.e. we set $\mathscr{K} = 0$ in \eqref{eqn: skdv with K}. We prove well-posedness of this simplified system (See  Definitions \ref{definition: pathwise}), and add the noise term back using a standard ``piecing out'' argument. Motivated by the deterministic works \cite{Tem69, Tem97}, and their stochastic extensions \cite{DeBussche98, GMR21}, we apply a parabolic regularization technique and a Galerkin approximation to prove Theorem \ref{Thm: summary existence}. 

\item In Theorem \ref{Thm: summary invariant} we prove the existence of an invariant measure {in the phase space $H^2(\mathbb{T})$} to the system \eqref{eqn: skdv with K} for appropriately large damping parameter $\gamma>0$ and under the assumption that the noise is square-integrable (i.e. $\mathscr{K}=0$). 
Our proof of existence of an invariant measure is based on classical techniques of Maslowski and Seidler \cite{MaSe99}; which have also been employed previously in several works (See e.g. \cite{BMO17, RM23} and the references therein).
\end{itemize}
The paper is organized as follows: In Section \ref{section: setup}, we layout the mathematical setup and the assumptions on the stochastic terms appearing on the right side of the equation \eqref{eqn: skdv}. In Section \ref{section: results}, we state the main results in this article. In Section \ref{section: galerkin}, we formulate our Galerkin approximation for a regularized system given in terms of the regularization coefficient $\epsilon>0$. Moreover, we apply a standard cutoff function to truncate all the nonlinear terms. We obtain a priori estimates and tightness of the laws of the solutions to the approximation equations. Then we apply the Skorohod Representation Theorem to obtain a sequence with the same distribution, that converges almost surely on (potentially) a new probability space, and pass to the limit to obtain a global martingale solution to the regularized equation \eqref{eqn: regularized}. The limit passages are as follows: In Section \ref{Ntoinfty}, we pass the Galerkin and the truncation parameter $N\to\infty$ {and the regularization coefficient $\eps \to 0$}. In Section \ref{section: pathwise uniqueness}, pathwise uniqueness of the solution is proven. Thanks to a well-known argument of Gy\"{o}ngy and Krylov \cite{GK96}, that generalizes to infinite dimensions the classical Yamada–Watanabe theorem, we conclude the existence of a unique pathwise solution to \eqref{eqn: skdv}. In Section \ref{section: invariant measure}, we consider a square-integrable L\'evy noise and take the damping coefficient $\gamma > 0$ to be large enough and prove the existence of an invariant measure using the Maslowski-Seidler Theorem \cite{MaSe99}.

\section{Mathematical Framework} \label{section: setup}
In this section we introduce the necessary notation and definitions.
We start by defining the appropriate functional spaces that are required for our analysis.
We will denote the one-dimensional torus by $\T$. Then, for any  $s \in \R$, we define,
\begin{equation}
    \begin{aligned}\label{Hs}
	    H^s(\T) & := \bigg\{ v\in L^2(\T):  \sum_{n \in \Z} (1+|n|^2)^s |\widehat{v}(n)|^2 < \infty   \bigg\},
	    \end{aligned}
\end{equation}
where $\widehat{v}(n)$ is the Fourier transform of $v$ defined by
\begin{equation} \label{fourier transform}
	\widehat{v}(n) := \frac{1}{\sqrt{2\pi}} \int_\T e^{-inx} v(x) dx.
\end{equation}
The Hilbert space $H^s(\T)$ defined in \eqref{Hs} is endowed with inner product 
\begin{equation*}
	\langle v_1, v_2  \rangle_{H^s(\T)} := \sum_{n \in \Z} (1+|n|^2)^s \widehat{v_1}(n) \overline{\widehat{v_2}(n)}, \ \ v_1, v_2 \in H^s(\T).
\end{equation*}
For any $m \in \N$ and $v \in H^m(\T)$, it follows from the definition of Fourier transform \eqref{fourier transform} that 
\begin{equation}\label{fourier transform of derivative}
	\widehat{\partial_x^k v}(n) = (in)^k \widehat{v}(n), \ \ k \leq m.
\end{equation}
Then by \eqref{fourier transform of derivative}, the Parseval identity, and the fact that there exists a constant $c = c(m) > 0$ such that
\begin{equation*}
	1+|n|^{2m} \leq (1+|n|^2)^m \leq c  (1+|n|^{2m}),
	\end{equation*}
we note the following equivalence of norms:
	\begin{equation} \label{Hs norm equivalent}
		|v|_{H^m(\T)} \sim |v|_{L^2(\T)} + |\partial_x^m v|_{L^2(\T)}, \ \ v \in H^m(\T).
	\end{equation}
Occasionally, we will abbreviate the norm notation $|\cdot|_{H^s(\T)}$ by excluding $\T$.

Given a Hilbert space $V$, we denote by $\mathcal{D}([0,T]; V)$, the collection of c\`adl\`ag functions i.e. right-continuous functions with left limits, from $[0,T]$ to $V$. The Skorohod space $\mathcal{D}([0,T]; V)$ is endowed with the so-called Skorohod topology, which is separable and metrizable; see \cite{billingsley2013convergence} for more details.

Next, we will describe the stochastic terms appearing on the right-hand side of \eqref{eqn: skdv with K} and explain the motivation behind our assumptions in the remark that follows (see Remark \ref{rem:noise}).  {Consider a separable real Hilbert space $U$. Let $E$ be a separable, complete metric space with Borel $\sigma$-field $\mathcal{E}$ and a measure $\nu$ that is finite on bounded Borel subsets of $E$.}
\begin{definition}\label{basis}
 A {\bf stochastic basis} is a multiplet $(\Omega, \mathcal{F}, (\mathcal{F}_t)_{t\geq 0}, \P, W, \pi)$ consisting of 
a filtered probability space $(\Omega, \mathcal{F}, (\mathcal{F}_t)_{t\geq 0},\P)$, a Wiener process $W$ and Poisson measure $\pi$ with the following properties. The filtration $(\mathcal{F}_t)_{t\geq 0}$ satisfies the ``usual" conditions; that is,  $\mathcal{F}_0$ is complete and that for any $t\geq 0$, $\mathcal{F}_t = \bigcap_{s > t} \mathcal{F}_s$. The processes $W$ and $\pi$ are independent and satisfy the following conditions: 
\begin{itemize}
    \item $W$ is a $U$-valued Wiener process with respect to the filtration $(\mathcal{F}_t)_{t \geq 0}$. 
  \item  $\pi$ is a Poisson random measure on $[0, \infty) \times E$ with intensity measure $dt \otimes d\nu$, that arises from a stationary $(\mathcal{F}_t)_{t\geq 0}$-Poisson point process $\Xi$. 
 We denote by $\widehat{\pi}$, the corresponding compensated Poisson random measure which is formally given by $\widehat{\pi}=\pi-\nu$. 
\end{itemize}
\end{definition}
\begin{remark}\label{rem:noise}
The assumptions on the noise are based on the following observation: let $L$ be an $(\mathcal{F}_t)_{t\geq 0}$-Lévy process, not necessarily square-integrable, taking values in a real, separable Hilbert space $U$. The Lévy–Khintchin decomposition \cite[Page 53]{peszat2007stochastic} says that formally we can write 
$L(t) = at + W(t) + \mathcal{L}(t) + \mathcal{P}(t)$,
where $a \in U$, $W$ is a Wiener process, $\mathcal{L}$ is a sum of independent compensated compound Poisson processes, $\mathcal{P}$ is a compound Poisson process and, the processes $W, \mathcal{L}$ and $\mathcal{P}$ are independent. Hence, it tells us that stochastic integration with respect to $L$ can be decomposed informally as
$dL = a dt + dW + d\widehat{\pi} + d\pi$,
where $\pi := \sum_{t>0} \delta_{(t, \Delta L(t))}$ is the jump measure of $L$, {which is a Poisson random measure on $[0,\infty)\times E$ for $E := U \setminus \{0\}$}, and $\delta_{(t, \Delta L(t))}$ is the Dirac measure at $(t, \Delta L(t))$; see \cite{peszat2007stochastic, CTT18compare} for more details. The jump measure $\pi$ is a Poisson random measure that satisfies the conditions listed above. The corresponding stationary $(\mathcal{F}_t)_{t\geq 0}$-Poisson point process $\Xi$ is induced by the jumps of $L$. The intensity measure of $\pi$ is $dt \otimes d\nu$, where $\nu$ is the Lévy measure of $L$. 
We note that $\nu(E)$ might not be finite, but $\nu$ is finite on bounded subset of $E$. In fact, $(E, \mathcal{E}, \nu)$ is a $\sigma$-finite measure space (see \eqref{appendix:sigma finite}). We define the set $E_0 :=  \{x \in U: 0 < |x|_U < 1 \}$. It follows that $\nu(E \setminus E_0) < \infty$. The set $E_0\in\mathcal{E}$ appearing in \eqref{eqn: skdv with K} is chosen such that $\nu(E \setminus E_0) < \infty$. As an example, we can set $E_0 :=  \{\xi \in U: 0 < |\xi|_U < 1 \}$.

\end{remark}

Next, let $\mathcal{Q}$ be the covariance operator of $W$, and define $U_0 := \mathcal{Q}^{1/2}(U)$.
Then, for the noise coefficients $G$ and $K$,  we impose the following conditions: For each $s = 0,1, 2$ we assume that
\begin{align*}
    G: H^s(\T) \to HS(U_0, H^s(\T)),\quad
    K : H^s(\T) \times E_0 \to H^s(\T), 
\end{align*}
where $HS(X, Y)$ stands for the space of Hilbert-Schmidt operators from a Hilbert space $X$ to a Hilbert space $Y$. Furthermore, we assume the following linear growth conditions and Lipschitz conditions on the noise coefficients in the spaces $L^2(\T), H^1(\T), H^2(\T)$: There exist constants $\kappa_1, \kappa_2, C > 0$ such that 
\begin{align} 
    \tag{A1} \label{assumption: growth of G}  \|G(v)\|^2_{HS(U_0, H^s)} & \leq \kappa_1 (1+|v|^2_{H^s}), \ v \in H^s(\T), \ s = 0, 1, 2 \\ 
    \tag{A2} \label{assumption: growth of K}  \int_{E_0} |K(v,\xi)|_{H^s}^2 d\nu(\xi) & \leq \kappa_2 (1+|v|^2_{H^s}), \ v \in H^s(\T), \ s = 0, 1, 2 \\
    \tag{A3} \label{assumption: further growth of K} \int_{E_0} |K(v, \xi)|^8_{L^2} d \nu(\xi) & \leq \kappa_2 (1+|v|^8_{L^2}), \ v \in L^2(\T),
\end{align}
and, for $s = 0, 1, 2$ we assume that
\begin{equation}\label{assumption: lipschitz}
	\tag{A4}
	\begin{aligned} 
      \|G(v) - G(w)\|^2_{HS(U_0, H^s)} + \int_{E_0} |K(v, \xi) - K(w,\xi)|^2_{H^s} d \nu(\xi) \leq C |v - w|^2_{H^s}, \ v, w \in H^s(\T). 
    \end{aligned}
\end{equation}
We note, by H\"{o}lder's inequality, that for $2 \leq p \leq 8$, there exists some constant $C = C(\kappa_2, p)$ such that
\begin{equation} \label{assumption:interpolate K}
    \int_{E_0} |K(v, \xi)|^p_{L^2} d \nu(\xi) \leq C(1+|v|^p_{L^2}),\qquad \ v \in L^2(\T).
\end{equation}
We assume that $\mathscr{K}: L^2(\T) \times E \to L^2(\T)$ is a measurable function. In contrast to $K$, we do not impose any growth or Lipschitz assumptions on $\mathscr{K}$.  The last term on the right-hand side of \eqref{eqn: skdv with K} represents the influence of large jumps from the L\'{e}vy noise. If the L\'{e}vy noise is not square-integrable, then the influence of large jumps cannot be represented without this term; see \cite{peszat2007stochastic,CTT18compare}.

\section{Main Results} \label{section: results}
{ 
 In this Section, we will introduce the necessary definitions and state the main results of this manuscript: the existence of a unique solution and that of an invariant measure. It is standard to prove the existence of appropriate solutions to stochastic PDEs driven by a L\'{e}vy noise by first considering a square integrable jump noise and then augmenting it by the effects of large jumps using the so-called `piecing out' argument. Moreover, to prove our second main result which gives the existence of an invariant measure, we discount the effects of the large jumps. Henceforth, we will set $\mathscr{K}=0$ in \eqref{eqn: skdv with K} and consider the following equations:
 \begin{equation}
 	\begin{cases}  \label{eqn: skdv}
 		du(t) + (\partial_x^3 u(t) +  u (t)\partial_x u(t) + \gamma u(t))dt = G(u(t-))  dW(t) +  \int_{E_0} K(u(t-), \xi) d\widehat{\pi}(t,\xi), \\
 		u(0) = u_0,
 	\end{cases}
 \end{equation}

We will begin by defining the two relevant notions of solutions to the damped stochastic KdV equation \eqref{eqn: skdv}.
 A martingale solution is a weak solution (in the probabilistic sense) where the stochastic basis, which includes the underlying probability space, the Wiener process and the Poisson random measure, is itself constructed as a part of the solution. Since the probability space is itself one of the unknowns
in a martingale solution, it is not possible to a priori specify the initial condition as a random variable and instead it is prescribed only in law by a given Borel probability measure. 
\begin{definition}[Martingale solutions to \eqref{eqn: skdv}] \label{definition: martingale} Assume that the given probability measure $\mu_0$ satisfies 
		\begin{align}\label{assm:mu0}
		\int_{L^2}|v|^8_{L^2}d\mu_0(v) + \int_{H^2}|v|^2_{H^2}d\mu_0(v)<\infty.
	\end{align} 
    A pair $(\TS, \Tu)$ is called a martingale solution to \eqref{eqn: skdv} if 
    \begin{itemize} 
    	\item  $\TS$ is a stochastic basis  $(\TOmega, \TF, (\TF_t)_{t \geq 0}, \TP, \widetilde{W}, \widetilde{\pi})$ as defined in Definition \ref{basis},
    	\item the process $\Tu$ satisfying 
    \begin{align*}
    & \Tu \in L^8(\TOmega, L^\infty(0,T; L^2(\T))) \cap L^2(\TOmega, L^\infty(0,T;H^2(\T)))
     \end{align*}
    is $(\TF_t)_{t\geq 0}$-adapted with c\`{a}dl\`{a}g paths in $H^{-2}(\T)$, $\TP$-a.s., 
    \item  $\Tu(0)$ is $\TF_0$-measurable with law $\mu_0$, and for every $t \in [0,T]$ the following equation,
    \begin{equation}
        \begin{aligned}
        \Tu(t) + \int_0^t (\partial_x^3 \Tu(s) + \Tu(s) \partial_x \Tu(s) + \gamma \Tu(s)) ds & = \Tu(0) + \int_0^t G(\Tu(s-)) d\widetilde{W}(s)\\
         & \quad + \int_{(0,t]}\int_{E_0} K(\Tu(s-),\xi) d\widehat{\Tpi}(s,\xi) 
        \end{aligned}
    \end{equation}
    holds (in the sense of distributions) $\TP$-a.s.
    \end{itemize}
\end{definition}
Next, we will give the definition of a pathwise solution to \eqref{eqn: skdv}. In this case, we construct an appropriate stochastic process solving \eqref{eqn: skdv} with respect to a pre-ordained stochastic basis.
\begin{definition}[Pathwise solutions to \eqref{eqn: skdv}]
     \label{definition: pathwise}
    Suppose that we are given a stochastic basis $\mathcal{S} = (\Omega, \F, (\F_t)_{t \geq 0}, \P, W, \pi)$ as in Definition \ref{basis}, and a random variable $u_0 \in L^8(\Omega, \F_0, \P; L^2(\T)) \cap L^2(\Omega, \F_0, \P; H^2(\T))$ with law $\mu_0$ satisfying \eqref{assm:mu0}. We say that an $(\F_t)_{t\geq 0}$-adapted process $u$ is a (global) pathwise solution to \eqref{eqn: skdv} if $u$ satisfying
    \begin{align*}
    u \in L^8(\Omega, L^\infty(0,T; L^2(\T))) \cap L^2(\Omega, L^\infty(0,T;H^2(\T)))
    \end{align*}
     takes c\`{a}dl\`ag  paths in $H^{-2}(\T)$, $\P$-a.s., and satisfies the following equation for every $t \in [0,T]$,
  \begin{equation}
	\begin{aligned}
		u(t) + \int_0^t (\partial_x^3 u(s) + u(s) \partial_x u(s) + \gamma u(s)) ds  & = u(0) + \int_0^t G(u(s-)) d W(s)\\
		&\quad + \int_{(0,t]}\int_{E_0} K(u(s-),\xi) d\widehat{\pi}(s,\xi) 
	\end{aligned}
\end{equation}
    $\P$-a.s.  (in the sense of distributions).
\end{definition}

We will now state one of the main results of this manuscript, namely the existence of a pathwise solution to \eqref{eqn: skdv}. 
\begin{theorem} \label{Thm: summary existence}
	Assume that we are given an $\F_0$-measurable random variable $u_0: \Omega \to H^2(\T)$ such that $u_0 \in L^8(\Omega, \F_0, \P; L^2(\T)) \cap L^2(\Omega, \F_0, \P; H^2(\T))$.
	We also assume that $G$ and $K$ satisfy the assumptions \eqref{assumption: growth of G}-\eqref{assumption: lipschitz}
	. Then there exists a unique pathwise solution $u$ to \eqref{eqn: skdv} in the sense of Definition \ref{definition: pathwise}, with initial condition $u_0$.
\end{theorem}

Next, we will state our second main result of the manuscript.
For any $u_0 \in H^2(\T)$, we denote by $u(t;u_0)$, the unique solution to $\eqref{eqn: skdv}$ with the initial condition $u_0$ given by the Theorem \ref{Thm: summary existence}. We define the transition probabilities as follows,
\begin{equation} 
	P_t(u_0, A) := \P( u(t;u_0) \in A ), \ t \geq 0, \ u_0 \in H^2(\T), \ A \in \mathcal{B}(H^2(\T)),
\end{equation}
where $\mathcal{B}(H^2(\T))$ is the collection of Borel measurable sets on $H^2(\T)$. The associated transition semigroup $P_t$ acting on $C_b(H^2(\T))$, the set of continuous, bounded and measurable functions from $H^2(\T)$ to $\R$, is then given by: 
\begin{equation} \label{Pt}
	P_t\phi(u_0) := \E[\phi(u(t; u_0))] = \int_{H^2} \phi(v) P_t(u_0, dv), \ t \geq 0, \ \phi \in C_b(H^2(\T)).
\end{equation}
 For any Borel probability measure $\mu \in \Pr(H^2(\T))$, we then define
\begin{equation} \label{transition prob with random initial}
	\mu P_t(A) := \int_{H^2} P_t(u_0, A) \mu(du_0), \ A \in \mathcal{B}(H^2(\T)),
\end{equation}
and say that $\mu \in \Pr(H^2(\T))$ is an \emph{invariant measure} for $\{P_t\}_{t \geq 0}$ if $\mu P_t = \mu$ for every $t \geq 0$. 

Now we are in position to precisely state our next main result.
\begin{theorem} \label{Thm: summary invariant}
 Assume that the noise coefficients $G$ and $K$ satisfy the assumptions \eqref{assumption: growth of G}-\eqref{assumption: lipschitz}.  Then for a sufficiently large damping coeffcient $\gamma > 0$ there exists an invariant measure $\mu \in \text{Pr}(H^2(\mathbb{T}))$ for  $\{P_t\}_{t\geq 0}$ describing the dynamics of the damped KdV equation  \eqref{eqn: skdv}$_1$ driven by a square-integrable L\'evy noise.
\end{theorem}

In what follows we will give proofs for Theorems \ref{Thm: summary existence} and \ref{Thm: summary invariant}. 

\section{Existence of pathwise solution}\label{sec:existence}
In this Section we will present a proof of our first main result Theorem \ref{Thm: summary existence}. The first step in the proof is to prove the existence of a martingale solution to \eqref{eqn: skdv} in the sense of Definition \ref{definition: martingale}. Then we will establish pathwise uniqueness of such solutions in an appropriate class which will imply the existence of a unique pathwise solution due to the classical theorem of Gyongy-Watanabe.

The proof of existence of a solution to the system  \eqref{eqn: skdv}  is based on three layers of approximations based on regularization, finite-dimensional approximation and truncation of nonlinearities. We first regularize the system \eqref{eqn: skdv} by augmenting it with artificial viscosity.  To prove well-posedness of this we employ a finite-dimensional Galerkin approximation. Moreover, this approximation system consists of truncated nonlinearities obtained by applying a cutoff function to the two nonlinear terms appearing on the left-hand side of  \eqref{eqn: skdv}. {The truncation gives us globally Lipschitz drift terms which ensures global-in-time well-posedness of the Galerkin scheme.} 
\subsection{Galerkin Scheme} \label{section: galerkin}
In this section we will construct solutions to approximations of the system \eqref{eqn: skdv}. 
Following \cite{Tem69, Tem97, DeBussche98,GMR21}, we first regularize the equation \eqref{eqn: skdv} by augmenting it with a regularization term $\eps \partial_x^4 u$, $0<\eps \leq 1$. 
\begin{equation}  \label{eqn: regularized}
\begin{cases}
	\begin{split}
    du_\eps(t) + \big(\partial_x^3 u_\eps(t) +  u_\eps(t) \partial_x u_\eps(t))& + \gamma u_\eps(t) + \eps \partial_x^4 u_\eps (t)\big) dt \\
    &= G(u_\eps(t-)) dW(t) +  \int_{E_0} K(u_\eps(t-), \xi) d\widehat{\pi}(t,\xi) ,  \\
    u_\eps(0) &= u_{ { 0} }  .
    \end{split}
    \end{cases}
\end{equation}
To establish the existence of solutions to \eqref{eqn: regularized} will next introduce a Galerkin approximation scheme.
For that purpose, we define $H_N$ as follows,
\begin{equation}
	H_N := \text{span} \bigl\{ e^{ikx}: |k| \leq N \bigr\}.
\end{equation}
We project \eqref{eqn: regularized} onto its Fourier modes $|k| \leq N$ and thus consider following $N$-th Galerkin scheme for any $N\in\mathbb{N}$:
\begin{equation}  \label{eqn: galerkin relabel}
	\begin{cases}
		du^N_\eps(t) + \bigg( \theta_N(|u^N_\eps(t)|_{H^1}) ( \partial_x^3 u^N_\eps (t)+ P_N( u^N_\eps (t)\partial_x u^N_\eps(t)) ) + \gamma u_\eps^N(t) + \eps \partial_x^4 u^N_\eps(t) \bigg) dt \\ 
		\hspace{5cm} = P_N G(u^N_\eps(t-)) dW(t) +  \int_{E_0} P_N K(u^N_\eps(t-), \xi) d\widehat{\pi}(t, \xi)  ,\\
		u^N_\eps(0, x) = { P_N u_{0} (x) } .
	\end{cases}
\end{equation}
Here, to deal with the stochasticity, we additionally truncated the nonlinear terms via a cutoff function $\theta_N$ defined below. Let $\theta(t):[0,\infty) \to \R$ be the smooth function such that 
	\begin{equation*} \label{truncation base function}
		\theta(x) := \begin{cases}
			1, \ x < \frac{1}{2}; \\
			0, \ x \geq 1.
		\end{cases}
	\end{equation*} 
	Define $\theta_N(x) := \theta(\frac{x}{N})$. Then we observe that for any $N>0$, $\theta_N: [0,\infty) \to [0,1]$ is a smooth function satisfying
\begin{equation} \label{truncation definition}
    \theta_N(x) = \begin{cases}
        1, \ x < \frac{N}{2}; \\
        0, \ x \geq N.
    \end{cases}
\end{equation}
 Moreover, we note that the cutoff function $\theta_N$ constructed in \eqref{truncation definition} is a Lipschitz function with the Lipschitz coefficient independent of $N$:
 \begin{equation} \label{truncation lipschitz}
 	|\theta_N'(x)| = \frac{1}{N} \biggl|\theta'\biggl(\frac{x}{N}\biggr) \biggr| \leq |\theta'|_{L^\infty}, \quad \forall N \in \N.
 \end{equation}

Note that, since, for any fixed $N \in\mathbb{N}$, the maps $v \mapsto \theta_N(|v|_{H^1}) v\partial_x v$, $v \mapsto \eps \partial_x^3 v$ and $v \mapsto \partial_x^4 v$ are all globally Lipschitz in $H_N$, and $G$ and $K$ satisfy the Lipschitz conditions \eqref{assumption: growth of G}-\eqref{assumption: growth of K}, there exists a unique pathwise solution $u_\eps^N$ to the SDE \eqref{eqn: galerkin relabel} thanks to \cite[Theorem 3.3]{CNTT20}.

\subsection{Moment estimates} \label{section: moment}
In this section we will obtain estimates for the Galerkin approximates $u^N_\eps$ independently of $N$ and $\eps$. Hereon, the notation $C$ will be used to denote a generic constant which, unless specified otherwise, will be independent of the parameters $N$ and $\eps$. 
To obtain the desired moment estimates, we begin by defining the following functionals
\begin{align*}
	\mathcal{I}_0(v) & := \int_\T (v(x))^2 dx, \\
	\mathcal{I}_1(v) & := \int_\T \frac{1}{2} \left( (\partial_x v(x))^2 - \frac{1}{6} (v(x))^3\right)  dx, \\
	  \mathcal{I}_2(v) & := \int_\T \left( (\partial_x^2 v(x))^2 - \frac{5}{3} v(\partial_x v(x))^2 + \frac{5}{36}(v(x))^4\right)  dx.
\end{align*}
These functionals correspond to appropriate conserved quantities for the deterministic KdV equation.
 In fact the deterministic KdV equation admits an infinite family of conserved quantities $\mathcal{I}_m, m \geq 0$, which are polynomials of $(u, \partial_x u, \cdots, \partial_x^m u)$, {with the leading order term given by the $L^2(\T)$-norm of $\partial_x^m u$}.  For precise definitions and detailed discussion of $\mathcal{I}_m$, the reader may refer to \cite[Chapter IV, Section 7]{Tem97}. We will use these functionals as we apply the Ito formula to the approximation system \eqref{eqn: galerkin relabel} to derive appropriate bounds.

Our major goal is to obtain bounds for the approximate solutions $u_\epsilon^N$ independently of $N$ in $L^2_\P L^\infty_t H^2_x$. Since our spatial domain is $\T$, we first seek to obtain uniform $L^8_\P L^\infty_t L^2_x$ estimates for $u_\eps^N$ which is done in the following theorem. 
\begin{theorem}[$L^8_\P L^\infty_t L^2_x$ regularity for $u_\eps^N$] \label{ape:thm1}

	Assume that $u_{0} \in L^8(\Omega, \F_0, \P; L^2(\T))$ and that the noise coefficients $G$ and $K$ satisfy the assumptions stated in \eqref{assumption: growth of G}-\eqref{assumption: lipschitz} Then for any $p\in[2,8]$, there exists a constant $ C = C(T, p, \kappa_1, \kappa_2) > 0$, independent of {$N$ and $\epsilon$}, such that 
\begin{equation} \label{ape: u_eps^N}
    \E\bigg[\sup_{t \in [0,T]} |u^{N}_\eps(t)|^p_{L^2} \bigg] + \eps p \E \int_0^T |u_\eps^N(s)|_{L^2}^{p-2}|\partial_x^2 u^N_\eps(s)|^2_{L^2}ds \leq  Ce^{CT} (\E|u_{0}|_{L^2}^p + 1).
\end{equation}
\end{theorem}
\begin{proof}
We begin introduce the following functional which is a slight modification of $\mathcal{I}_0:L^2(\T)\to\R$:
\begin{equation*}
	\tilde{\mathcal{I}}_0(v) := |v|_{L^2}^p.
\end{equation*} 
The purpose of this modified functional is to obtain the higher order moment bounds claimed in \eqref{ape: u_eps^N}.
Observe that the first and second Fr\'echet derivatives ${\tilde {\mathcal{I}}}_0': L^2(\T) \to L^{2}(\T)$ and $\tilde{\mathcal{I}}_0'': L^2(\T) \to L(L^2(\T), L^{2}(\T))$ of $\tilde{\mathcal{I}}_0$ are given by, 
    \begin{align*}
    \tilde{\mathcal{I}}'_0(v) = p|v|_{L^2}^{p-2}\langle v, \cdot \rangle_{L^2}, \qquad
   \tilde{\mathcal{I}}''_0(v) = p(p-2)|v|_{L^2}^{p-4}\langle v, \cdot \rangle_{L^2} v + p|v|_{L^2}^{p-2} I.
    \end{align*}
We now apply It\^o's formula with the functional $\tilde{\mathcal{I}}_0(u_\eps^N)$ to our approximation system \eqref{eqn: galerkin relabel}. This gives us,
\begin{equation} \label{ape: u_eps^N LHS 1}
    \begin{aligned}
        &  |u^N_\eps(t)|_{L^2}^p + p \int_0^t |u^N_\eps(s)|_{L^2}^{p-2} (u_\eps^N(s),  u_\eps^N(s) \partial_x u_\eps^N(s))_{L^2} \theta_N(|u_\eps^N(s)|_{H^1})ds + \gamma p \int_0^t |u_\eps^N(s)|_{L^2}^p ds  \\
        & + p \int_0^t |u^N_\eps(s)|_{L^2(s)}^{p-2} \left( u^N_\eps(s),  \partial_x^3 u_\eps^N(s)\right) _{L^2} \theta_N(|u_\eps^N(s)|_{H^1})ds   + \eps p \int_0^t |u_\eps^N(s)|_{L^2}^{p-2} (u_\eps^N(s), \partial_x^4 u_\eps^N(s))_{L^2}ds \\  
     & = |u^N_\eps(0)|_{L^2}^p + p\int_0^t |u^N_\eps(s)|_{L^2}^{p-2}(u_\eps^N(s), P_NG(u^N_\eps(s))dW(s))_{L^2} \\
     & + p\int_0^t \int_{E_0} |u^N_\eps(s)|_{L^2}^{p-2}(u^N_\eps(s), P_N K(u^N_\eps(s), \xi))_{L^2} d\widehat{\pi}(\xi,s)  \\
        & + \frac{p}{2} \int_0^t |u^N_\eps(s)|_{L^2}^{p-2} \|G(u^N_\eps(s))\|_{HS(U_0, L^2)}^2  ds + \frac{p(p-2)}{2}\int_0^t |u_\eps^N(s)|_{L^2}^{p-4}\|G^*(u^N_\eps(s)) u^N_\eps(s)\|^2_{U_0} ds \\
        & + \int_0^t \int_{E_0} |u_\eps^N(s) + P_N K(u_\eps^N(s), \xi)|_{L^2}^p - |u_\eps^N(s)|_{L^2}^p  - p|u_\eps^N(s)|_{L^2}^{p-2}(u_\eps^N(s), P_N K(u_\eps^N(s), \xi))_{L^2} d\pi(\xi,s) \\
        & =: |u^N_\eps(0)|_{L^2}^p + I_1 + I_2 + I_3 + I_4 + I_5,
    \end{aligned}
\end{equation}
where $G^*$ is the adjoint operator of $G$.
Before estimating the terms on the right hand side, we make some important observations about the terms on the left-hand side of \eqref{ape: u_eps^N LHS 1}. Note that the second and third term on the left-hand side of \eqref{ape: u_eps^N LHS 1} vanish after we integrate by parts as follows: 
\begin{align*}
	\theta_N(|u_\eps^N(s)|_{H^1}) \lp u_\eps^N(s),  u_\eps^N(s) \partial_x u_\eps^N(s) \rp_{L^2}  &  =  \theta_N(|u_\eps^N(s)|_{H^1}) \int_{\T} \frac{1}{3} \partial_x (u_\eps^N(s, x))^3 dx \\
	& =  \theta_N(|u_\eps^N(s)|_{H^1}) \frac{1}{3} (u_\eps^N(s, x))^3\bigg|^{x=1}_{x=0} = 0,
\end{align*}
and, 
\begin{align*}
	\theta_N(|u_\eps^N(s)|_{H^1}) (u^N_\eps(s), \partial_x^3 u_\eps^N(s))_{L^2}  & = \theta_N(|u_\eps^N(s)|_{H^1}) \bigg[ u^N_\eps(s,x) \partial_x^2 u_\eps^N(s,x)\bigg|^{x=1}_{x=0} - \int_0^1 \partial_x u^N_\eps \partial_x^2 u_\eps^N dx \bigg] \\
	& = -\theta_N(|u_\eps^N(s)|_{H^1}) \int_0^1 \frac{1}{2} \partial_x [(\partial_x u_\eps^N(s,x))^2] dx \\
	& = -\frac{1}{2} \theta_N(|u_\eps^N(s)|_{H^1}) (\partial_x u_\eps^N(s,x))^2 \bigg|_{x=0}^{x=1} = 0.
\end{align*}
Moreover, {after integrating by parts twice}, we can rewrite the last term on the left-hand side of \eqref{ape: u_eps^N LHS 1}, as follows
\begin{align*}
	\eps p \int_0^t |u_\eps^N(s)|^{p-2} (u_\eps^N(s), \partial_x^4 u_\eps^N(s))_{L^2}  ds =  \eps p \int_0^t |u_\eps^N(s)|^{p-2} |\partial_x^2 u_\eps^N(s)|^2 ds.
\end{align*}
Hence, after applying the three equalities above, the left-side of the equation \eqref{ape: u_eps^N LHS 1} simplifies down to the following expression: 
\begin{equation} \label{ape: u_eps^N LHS 2}
    \begin{aligned}
    \text{LHS of \eqref{ape: u_eps^N LHS 1}}   & = |u^N_\eps(t)|_{L^2}^p + \gamma p \int_0^t |u^N_\eps(s)|_{L^2}^p ds +  \eps p \int_0^t |u_\eps^N(s)|_{L^2}^{p-2} |\partial_x^2 u_\eps^N(s)|_{L^2}^2 ds.
    \end{aligned}
\end{equation}
Next, we will estimate the right-hand side terms of \eqref{ape: u_eps^N LHS 1} and treat each term $I_j;j=1,2,..,5$ individually. We first apply $\sup_{t \in [0,t']}$ for any $t'\geq 0$ on both sides of the equation \eqref{ape: u_eps^N LHS 1} and then take expectation. Using the Burkholder-Davis-Gundy (BDG) inequality and the growth assumption \eqref{assumption: growth of G} we obtain for $I_1$ that, 
\begin{equation} \label{eqn: ape wiener term}
\begin{aligned}
    \E \sup_{t\in [0,t']} |I_1(t)| 
    & \leq C \E \lp \int_0^{t'} |u^{N}_\eps(s)|_{L^2}^{2(p-1)}\|G(u^N_\eps(s))\|_{HS(U_0, L^2)}^2 ds \rp^{1/2} \\
     & \leq C \E \lp \int_0^{t'} |u^{N}_\eps(s)|_{L^2}^{2(p-1)}(1+|u_\eps^N(s)|_{L^2}^2) ds \rp^{1/2}  \\
& \leq C \E \lp\sup_{t \in [0,t']}|u^{N}_\eps(t)|_{L^2}^p \int_0^{t'} |u^{N}_\eps(s)|_{L^2}^{p-2}(1+|u_\eps^N(s)|_{L^2}^2) ds \rp^{1/2}  \\
& \leq \frac{1}{4} \E\sup_{t \in [0,t']}|u^N_\eps(t)|_{L^2}^p + C \E\int_0^{t'} |u_{\eps}^N(s)|^{p-2}_{L^2} ds + C \E\int_0^{t'} |u_{\eps}^N(s)|^{p}_{L^2} ds \\
& \leq \frac{1}{4} \E\sup_{t \in [0,t']}|u^N_\eps(t)|_{L^2}^p +  C \lp 1 + \E\int_0^{t'} |u_{\eps}^N(s)|^{p}_{L^2} ds \rp,
\end{aligned}
\end{equation}
where the constant $C = C(p,T)$ depends only on $p$ and $T$. A similar argument, applied to the term $I_2$ gives us, for some constant $C>0$, that
\begin{equation} \label{eqn: ape CPRM term}
\begin{aligned} 
    \E \sup_{t\in [0,t']} |I_2(t)|
    & \leq C \lp \E \int_0^{t'} \int_{E_0} |u^N_\eps(s)|_{L^2}^{2p-2}|P_N K(u^N_\eps(s), \xi))|^2_{L^2} d\pi(\xi,s) \rp^{1/2} \\
    & \leq C \E \lp \int_0^{t'} |u^{N}_\eps(s)|_{L^2}^{2(p-1)}(1+|u_\eps^N(s)|_{L^2}^2) ds \rp^{1/2}  \\
& \leq C \E \lp\sup_{t \in [0,t']}|u^{N}_\eps(t)|_{L^2}^p \int_0^{t'} |u^{N}_\eps(s)|_{L^2}^{p-2}(1+|u_\eps^N(s)|_{L^2}^2) ds \rp^{1/2}  \\
& \leq \frac{1}{4} \E\sup_{t \in [0,t']}|u^N_\eps(t)|_{L^2}^p + C \E\int_0^{t'} |u_{\eps}^N(s)|^{p-2}_{L^2} ds + C \E\int_0^{t'} |u_{\eps}^N(s)|^{p}_{L^2} ds \\
& \leq \frac{1}{4} \E\sup_{t \in [0,t']}|u^N_\eps(t)|_{L^2}^p +  C \lp 1 + \E\int_0^{t'} |u_{\eps}^N(s)|^{p}_{L^2} ds \rp,
\end{aligned}
\end{equation}
Next we consider the term $I_3$. Recall the growth condition on $G$ stated in \eqref{assumption: growth of G}.  Under this growth assumption, we obtain
\begin{equation*}
\begin{aligned}
    \E \sup_{t\in [0,t']} |I_3(t)| & \leq C \E \int_0^{t'} |u_\eps^N(s)|^{p-2}_{L^2} (1 + |u_\eps^N(s)|^{2}_{L^2}) ds
    \leq C\lp 1 + \E \int_0^{t'} |u_\eps^N(s)|^{p}_{L^2}ds  \rp.
\end{aligned}
\end{equation*}
To estimate $I_4$ term, we note that \eqref{assumption: growth of G} gives us,
\begin{equation*}
\begin{aligned}
    \|G(u^N_\eps(s))^* u^N_\eps(s)\|_{U_0}^2 &\leq \|G(u^N_\eps(s))^*\|_{HS( U_0,L^2)}^2 |u^N_\eps(s)|^2_{L^2} \\
    & = \|G(u^N_\eps(s))\|_{HS(U_0, L^2)}^2 |u^N_\eps(s)|^2_{L^2} \leq C(1+|u^N_\eps(s)|_{L^2}^2)|u^N_\eps(s)|^2_{L^2}.
\end{aligned}
\end{equation*}
Hence, for $I_4$ we obtain the following bounds,
\begin{equation}
\begin{aligned}
    \E \sup_{t\in [0,t']} |I_4(t)| & \leq C \E \int_0^{t'} |u_\eps^N(s)|^{p-2}_{L^2} \left( 1 + |u_\eps^N(s)|^{2}_{L^2}\right) ds  
     \leq C\lp 1 + \E \int_0^{t'} |u_\eps^N(s)|^{p}_{L^2}ds  \rp.
\end{aligned}
\end{equation}
To estimate $I_5$, we begin by recalling the following standard inequality given for any Hilbert space $H$ and $y, k \in H$ (see e.g. \cite{CNTT20}):
\begin{equation}\label{ineq}
    \bigg| |y+k|_H^p - |y|_H^p - p|y|_H^{p-2}\langle y,k \rangle_H \bigg| \leq (p^2-p)2^{p-3}\bigg(|y|_{H}^{p-2}|k|_H^2 + |k|_H^p\bigg),
\end{equation}
We take ${H := L^2(\T)}, y := u_\eps^N(s)$ and $k := P_NK(u_\eps^N(s), \xi)$ in \eqref{ineq} and invoke the growth condition \eqref{assumption: further growth of K} to get
\begin{equation}
    \begin{aligned}
        \E \sup_{t\in [0,t']} |I_5| & \leq C \E \int_0^{t'}\int_{E_0} |u_\eps^N(s)|_{L^2}^{p-2} |K(u_\eps^N(s), \xi)|_{L^2}^2 + |K(u_\eps^N(s), \xi)|_{L^2}^p d\nu(\xi) ds \\
        & \leq C \E \int_0^{t'}|u_\eps^N(s)|_{L^2}^{p-2} (1+|u_\eps^N(s)|_{L^2}^2) + (1+|u_\eps^N(s)|_{L^2}^p) ds\leq C \lp 1 + \E \int_0^{t'} |u_\eps^N(s)|^{p}_{L^2}  ds \rp.
    \end{aligned}
\end{equation}
Therefore, collecting all the terms $I_j; 1\leq j\leq 5$ and using \eqref{ape: u_eps^N LHS 2} in \eqref{ape: u_eps^N LHS 1}, we obtain
\begin{equation*}
    \begin{aligned}
        \E\sup_{t \in [0,t']} \bigg[|u^N_\eps(t)|_{L^2}^p + \gamma p \int_0^{t'} |u_\eps^N(s)|^{p}_{L^2} ds  +  \eps p \int_0^{t'} |u_\eps^N(s)|_{L^2}^{p-2} |\partial_x^2 u_\eps^N(s)|_{L^2}^2 ds  \bigg] \\
        \leq \frac{1}{2} \E \sup_{t \in [0,t']} |u_{\eps}^N(t)|_{L^2}^p + C \lp 1 + \E \int_0^{t'} |u_{\eps}^N(s)|^p_{L^2} ds \rp.
    \end{aligned}
\end{equation*}
Rearranging the terms gives us,
\begin{equation*}
    \begin{aligned}
        \E\sup_{t \in [0,t']} \bigg[ |u^N_\eps(t)|_{L^2}^p \bigg] + 2\gamma p \int_0^{T} |u_\eps^N(s)|^{p}_{L^2}  ds + 2 \eps p \E \int_0^{t'} |u_\eps^N(s)|_{L^2}^{p-2} |\partial_x^2 u_\eps^N(s)|_{L^2}^2 ds   \\
        \leq C \lp \E|u_{0}|_{L^2}^p + 1 +  \int_0^{t'} \E\sup_{t'' \in [0,s]}|u_{\eps}^N(t'')|^p_{L^2} ds \rp.
    \end{aligned}
\end{equation*}
Finally applying the Gronwall inequality, and setting $t' := T$, we arrive at 
\begin{equation*}
    \E\sup_{t \in [0,T]} \bigg[|u^N_\eps(t)|_{L^2}^p\bigg] + \eps p \E \int_0^T |u_\eps^N(s)|^{p-2}_{L^2} |\partial_x^2 u_\eps^N(s)|_{L^2}^2 ds   \leq C e^{CT} (\E|u_{0}|^p + 1),
\end{equation*}
where $C = C(p, T, \kappa_1, \kappa_2)$ is independent of $\eps, N$. 
\end{proof}
Next, we will obtain uniform estimates for $u_\eps^N$ in more regular spaces; specifically 
 in $L^2_\P L^\infty_t H^2_x$  using the invariant $\mathcal{I}_2$.
To prove these bounds, we shall use several classical inequalities, such as Agmon's inequality in one dimension and the Sobolev interpolation inequality, that we recollect below. We first recall {\bf Agmon's inequality} (see, e.g., \cite[Chapter II, Section 1.3]{Tem97}): For $v \in H^s(\T)$, $s \geq 1$, 
\begin{equation} \label{agmon}
	|v|_{L^\infty(\T)} \leq C |v|_{L^2(\T)}^{1/2} \left( |\partial_x v|_{L^2(\T)} + |v|_{L^2(\T)}\right) ^{1/2}.
\end{equation}
Next, we state {\bf Sobolev interpolation inequality}: Let $0 \leq s_1 < s < s_2$ and $\theta \in (0,1)$ such that $s = (1-\theta)s_1 + \theta s_2$. Suppose $v \in H^{s_2}(\T)$. Then 
\begin{equation} \label{interpolation}
	|v|_{H^s(\T)} \leq C |v|_{H^{s_1}(\T)}^{1-\theta} |v|_{H^{s_2}(\T)}^\theta,
\end{equation} 
where $C$ is a generic constant depending only on $s_1, s_2, \theta$. We will abbreviate this interpolation inequality as $H^s(\T) = [H^{s_2}(\T), H^{s_1}(\T)]_\theta$.
We now turn to proving the following higher regularity estimates. 
\begin{theorem}[$L^2_\P L^\infty_t H^2_x$ estimates for $u_\eps^N$] \label{ape:thm3}
    Let $u_{0} \in L^8(\Omega, \F_0, \P; L^2(\T)) \cap L^2(\Omega, \F_0, \P; H^2(\T))$ and assume that the noise coefficients $G$ and $K$ satisfy \eqref{assumption: growth of G}-\eqref{assumption: lipschitz} Then there exists a constant $C = C(\kappa_1, \kappa_2, T) > 0$, independent of $N$ and $ \eps$, such that
        \begin{equation}\label{ape: partial^2 u_eps^N}
            \E\biggl[ \sup_{t \in [0,T]} |\partial^2_x u^{N}_\eps(t)|_{L^2}^2 \biggr] + \eps \E \int_0^T |\partial_x^4 u^N_\eps(s)|_{L^2}^2\,ds \leq Ce^{CT}\lp \E|u_0|_{L^2}^8 +  \E|\partial_x^2 u_0|_{L^2}^2 + 1  \rp.
        \end{equation}
Consequently, we have, for some $C>0$ independent of $N$ and $\eps$, that
    \begin{equation}\label{ape: h2}
    	\E\biggl[ \sup_{t \in [0,T]} |u^{N}_\eps(t)|_{H^2}^2 \biggr] + \eps \E \int_0^T |u^N_\eps(s)|_{H^4}^2\,ds \leq Ce^{CT}\lp \E|u_0|_{L^2}^8 +  \E|\partial_x^2 u_0|_{L^2}^2 + 1  \rp.
    \end{equation}
    \end{theorem}
\begin{proof}
Consider the functional  $\mathcal{I}_2: H^2(\T) \to \R$ defined by
\begin{equation}\label{I2}
    \mathcal{I}_2(v) := \int_\T \bigg( (\partial_x^2 v(x))^2 - \frac{5}{3} v(x)(\partial_x v(x))^2 + \frac{5}{36}(v(x))^4 \bigg) dx. 
\end{equation}
First we note the two Fr\'echet derivatives of this functional $\mathcal{I}': H^2(\T) \to H^{-2}(\T)$ and $\mathcal{I}'': H^2(\T) \to L(H^2(\T), H^{-2}(\T))$:
\begin{equation}\label{derivativeI2}
\begin{aligned}
\langle \mathcal{I}_2'(v), w \rangle & = \int_\T \bigg( 2\partial_x^2v(x)\partial_x^2w(x) - \frac{5}{3}(\partial_x v(x))^2 w(x)- \frac{5}{3} \partial_x (v(x)^2)\partial_x w(x) + \frac{5}{9}(v(x))^3w(x) \bigg) dx, \\
  \langle \mathcal{I}_2''(v)h, w \rangle &=  \int_\T \biggl( \partial_x^2 h(x) \partial_x^2 w(x) + \frac{5}{3} v(x)^2 h(x) w(x) + \frac{10}{3} \partial_x v(x) \partial_x h(x) w(x) + \frac{10}{3}  v(x)\partial_x^2h(x) w(x)\\
  & \hspace{2cm} + \frac{10}{3} h(x)\partial_x^2 v(x) w(x) \biggr) dx.
\end{aligned}
\end{equation}
Applying It\^{o}'s formula to \eqref{eqn: galerkin relabel} with the functional $\mathcal{I}_2(u_\eps^N)$, we obtain
\begin{equation}
    \begin{aligned}\label{ito3}
    	& \mathcal{I}_2(u_\eps^N(t)) + \int_0^t {( \mathcal{I}_2'(u_\eps^N(s)), \theta_N(|u_\eps^N(s)|_{H^1}) (\partial_x^3 u_\eps^N(s) + u_\eps^N(s) \partial_x u_\eps^N(s)) )_{L^2}} ds \\
    	& + \int_0^t (\mathcal{I}_2'(u_\eps^N(s)), \gamma u_\eps^N(s) + \eps \partial_x^4 u_\eps^N(s))_{L^2} ds\\
        & = |\partial_x^2 u_\eps^N(t)|_{L^2}^2 - \frac{5}{3} \int_\T u_\eps^N(t, x) (\partial_x u_\eps^N(t,x))^2 dx  +  \frac{5}{36}  |u_\eps^N(t)|_{L^4}^4 \\
        & + 2 \gamma \int_0^t |\partial_x^2 u_\eps^N(s)|_{L^2}^2 ds - 5\gamma \int_0^t \int_\T u_\eps^N(s,x) (\partial_x u_\eps^N(s,x))^2 dx ds  + \frac{5}{9}\gamma \int_0^t |u_\eps^N(s)|_{L^4}^4 ds \\
        & + 2\eps \int_0^t |\partial_x^4 u_\eps^N(s)|_{L^2}^2 ds + \eps \bigg[ \frac{5}{3} \int_0^t  \int_\T (\partial_x u_\eps^N(s,x))^2 \partial_x^4 u_\eps^N(s,x) dx ds + \frac{5}{9} \int_0^t ((u_\eps^N(s))^3, \partial_x^4 u_\eps^N(s))_{L^2} ds \bigg] \\
& = \mathcal{I}_2(u_\eps^N(0)) + \int_0^t (\mathcal{I}_2'(u_\eps^N(s)), P_NG(u_\eps^N(s))dW(s))_{L^2} + \int_0^t\int_{E_0} (L'(u_\eps^N(s)), P_N K(u_\eps^N(s), \xi))_{L^2} d\widehat{\pi}(\xi, s) \\
        & + \frac{1}{2} \int_0^t \text{Tr}(\mathcal{I}_2''(u_\eps^N(s)) P_N G(u_\eps^N(s)) P_N G(u_\eps^N(s))^*) ds \\
        & + \int_0^t\int_{E_0} \left( \mathcal{I}_2(u_\eps^N(s) + P_N K(u_\eps^N(s), \xi)) - \mathcal{I}_2(u_\eps^N(s)) - \left( \mathcal{I}_2'(u_\eps^N(s)), P_N K((u_\eps^N(s), \xi))\right) _{L^2}\right)  d\pi (\xi, s) \\
        & =: \mathcal{I}_2(u_\eps^N(0)) + R_1(t) + R_2 (t)+ R_3(t) + R_4(t).
    \end{aligned}
\end{equation}
Here we note that we additionally used the following identity in simplifying the left hand side, which follows from integration by parts:
$$\int_0^t {\biggl( \mathcal{I}_2'(u_\eps^N(s)), \theta_N(|u_\eps^N(s)|_{H^1}) (\partial_x^3 u_\eps^N(s) + u_\eps^N(s) \partial_x u_\eps^N(s)) \biggr)_{L^2}} ds =0.$$
We collect the rest of terms on the left side as denote them as,
\begin{equation}
	\begin{aligned}\label{Lterms}
		L_1 (t)+ L_2(t) + L_3(t) + L_4 (t)& := - \frac{5}{3} \int_\T u_\eps^N(t,x) (\partial_x u_\eps^N(t, x))^2 dx - 5\gamma \int_0^t \int_\T u_\eps^N(s,x) (\partial_x u_\eps^N(s,x))^2 dx ds \\ & +  \frac{5\eps}{3}\int_0^t \int_\T (\partial_x u_\eps^N(s,x))^2 \partial_x^4 u_\eps^N(s,x) dxds + \frac{5\eps}{9} \int_0^t ((u_\eps^N(s))^3, \partial_x^4 u_\eps^N(s))_{L^2}ds.
	\end{aligned}
\end{equation}
Our aim for the rest of this proof will be to obtain appropriate estimates for the term $R_j, L_j$, $j=1,2,3,4$. 
We begin with the right side terms. Using the formula for $\mathcal{I}'_2$ given in \eqref{derivativeI2} we rewrite $R_1$ as follows,
\begin{equation*}
    \begin{aligned}
        R_1(t) & = 2 \int_0^t (\partial_x^2 u_\eps^N(s), \partial_x^2 P_N G(u_\eps^N(s)) dW(s)) - \frac{5}{3}\int_0^t ((\partial_x u_\eps^N(s))^2, P_N G(u_\eps^N(s)) dW(s)) \\
        &  -\frac{5}{3}\int_0^t (\partial_x(u_\eps^N(s))^2, \partial_x P_N G(u_\eps^N(s))dW(s))_{L^2} + \frac{5}{9} \int_0^t ((u_\eps^N(s))^3, \partial_x P_N G(u_\eps^N(s))dW(s))_{L^2} \\
        & =: R_{1,1} (t)+ R_{1,2} (t)+ R_{1,3} (t)+ R_{1,4}(t).
    \end{aligned}
\end{equation*}
Next we apply $\E \sup_{t \in [0,t']}$ on both sides of the equation for $R_1$ given above. Then, an application of the BDG inequality, the growth assumption \eqref{assumption: growth of G} along with the bound obtained in Theorem \ref{ape:thm1} for $p = 2$ give us
\begin{equation*}
    \begin{aligned}
        \E\sup_{t \in [0,t']}|R_{1,1}(t)| 
        & \leq C \E \lp \int_0^{t'}|\partial_x^2 u_\eps^N(s)|^2_{L^2} \|\partial_x^2 P_N G(u_\eps^N(s))\|_{HS(U_0, L^2)}^2 ds \rp^{1/2} \\
        & \leq C \E \bigg[ \sup_{t \in [0,t']}|\partial_x^2 u_\eps^N(t)|_{L^2} \lp \int_0^{t'} \|\partial_x^2 P_N G(u_\eps^N(s))\|_{HS(U_0, L^2)}^2 ds \rp^{1/2} \bigg]\\
        & \leq \frac{1}{100} \E\sup_{t \in [0,t']}|\partial_x^2 u_\eps^N(t)|_{L^2}^2 + C \E \int_0^{t'}\| G(u_\eps^N(s))\|_{HS(U_0, H^2)}^2 ds \\
        & \leq \frac{1}{100} \E\sup_{t \in [0,t']}|\partial_x^2 u_\eps^N(t)|_{L^2}^2 + C \E \int_0^{t'}\left(1 + |u_\eps^N(s)|_{L^2}^2 + |\partial_x^2 u_\eps^N(s)|_{L^2}^2\right) ds \\
        & \leq \frac{1}{100} \E\sup_{t \in [0,t']}|\partial_x^2 u_\eps^N(t)|_{L^2}^2 + C \biggl( 1 + \E|u_0|^2_{L^2} + \E \int_0^{t'} |\partial_x^2 u_\eps^N(s)|_{L^2}^2ds \biggr).
    \end{aligned}
\end{equation*}
Similarly, we also find for $R_{1,2}$ that,
\begin{equation}
    \begin{aligned}\label{r12}
        \E\sup_{t \in [0,t']}|R_{1,2}(t)| & = C \E \sup_{t \in [0,t']} \bigg| \int_0^{t'} ( (\partial_x u_\eps^N(s))^2, P_N G(u_\eps^N(s)) dW(s) )_{L^2} \bigg| \\
        & \leq \E \lp \int_0^{t'}|\partial_x u_\eps^N(s)|^4_{L^4} \|G(u_\eps^N(s))\|_{HS(U_0, L^2)}^2 ds \rp^{1/2} \\
        & \leq \E \lp \int_0^{t'}|\partial_x u_\eps^N(s)|^4_{L^4} \left( 1 + |u_\eps^N(s)|_{L^2}^2  \right) ds  \rp^{1/2}.
    \end{aligned}
\end{equation}
To bound the right hand side of \eqref{r12} we note that the Sobolev embedding $H^{1/4}(\T) \hookrightarrow L^4(\T)$  in 1D and the Sobolev interpolation inequality $H^{5/4}(\T) = [H^2(\T), L^2(\T)]_{5/8}$ give us that,
\[|\partial_x u_\eps^N(s)|_{L^4} \leq C|u_\eps^N(s)|_{H^{5/4}} \leq C |u_\eps^N(s)|^{3/8}_{L^2}\left(|u_\eps^N(s)|_{L^2}^{5/8} + |\partial^2_x u_\eps^N(s)|_{L^2}^{5/8} \right).\]
Hence by Theorem \ref{ape:thm1}, for $p = 4$, for some constant $C$ independent of $N$ and $\eps$ we have,
\begin{equation*}
    \begin{aligned}
        \E \lp \int_0^{t'}|\partial_x u_\eps^N(s)|^4_{L^4} ds  \rp^{1/2} & \leq C \E \lp \int_0^{t'} \left( |u_\eps^N(s)|^4_{L^2}  + |u_\eps^N(s)|_{L^2}^{3/2}|\partial_x^2(s) u_\eps^N(s)|_{L^2}^{5/2} \right) ds \rp^{1/2}  \\
        & \leq C\E\sup_{t\in [0,t']}|u_\eps^N(t)|_{L^2}^2  + C\E \biggl[  \sup_{t\in[0,t']}  |u_\eps^N(t)|_{L^2}^{3/4}  |\partial_x^2 u_\eps^N(t)|_{L^2}^{5/4}  \biggr] \\
        & \leq C + \frac{1}{100} \E \sup_{t \in [0,t']} |\partial_x^2 u_\eps^N(t)|_{L^2}^2 + C \E\sup_{t \in [0,t']} |u_\eps^N(t)|_{L^2}^4 \\
        & \leq C (1+\E|u_0|_{L^2}^4) + \frac{1}{100} \E \sup_{t \in [0,t']} |\partial_x^2 u_\eps^N(t)|_{L^2}^2 .
    \end{aligned}
\end{equation*}
Furthermore, note that by Theorem \ref{ape:thm1} for $p = 6$,
\begin{equation*}
    \begin{aligned}
        \E \lp \int_0^{t'}|u_\eps^N(s)|_{L^2}^2 |\partial_x u_\eps^N(s)|^4_{L^4} ds  \rp^{1/2} & \leq C \E \lp \int_0^{t'}|u_\eps^N(s)|^6_{L^2} + |u_\eps^N(s)|_{L^2}^{7/2}|\partial_x^2 u_\eps^N(s)|_{L^2}^{5/2}  \rp^{1/2}  \\
        & \leq C\E\sup_{t\in [0,t']}|u_\eps^N(t)|_{L^2}^3  + C\E \biggl[  \sup_{t\in[0,t']} |u_\eps^N(t)|_{L^2}^{7/4} \sup_{t \in [0,t']} |\partial_x^2 u_\eps^N(t)|_{L^2}^{5/4} \biggr] \\
        & \leq C + \frac{1}{100} \E \sup_{t \in [0,t']} |\partial_x^2 u_\eps^N(t)|_{L^2}^2 + C \E\sup_{t \in [0,T]} |u_\eps^N(t)|_{L^2}^6 \\
        & \leq C (1+\E|u_0|_{L^2}^6) + \frac{1}{100} \E \sup_{t \in [0,t']} |\partial_x^2 u_\eps^N(t)|_{L^2}^2.
    \end{aligned}
\end{equation*}
Combining these two estimates with \eqref{r12}, we derive the following bounds for $R_{1,2}$:
\begin{equation*}
    \E \sup_{t \in [0,t']} |R_{1,2}(t)| \leq C(1+\E|u_0|_{L^2}^6) + \frac{2}{100} \E\sup_{t \in [0,t']}|\partial_x^2 u_\eps^N(t)|_{L^2}^2.
\end{equation*}
Next, we will find bounds for the term $R_{1,3}$. For that purpose, we apply the BDG inequality, use the continuous embedding $H^{1/4}(\T) \hookrightarrow L^4(\T)$ together with the interpolation inequality $H^{1/4} (\T)= [H^2(\T), L^2(\T)]_{1/8}$ and {bounds from Theorem \ref{ape:thm1} with $p = 6$ }to conclude that,
\begin{equation*}
    \begin{aligned}
        \E\sup_{t \in [0,t']} |R_{1,3}(t)| & \leq C \E \sup_{t \in [0,t']} \bigg| \int_0^t \left(\partial_x (u_\eps^N(s))^2, \partial_x P_N G(u_\eps^N(s))dW(s) \right) \bigg| \\
        & \leq C \E \lp \int_0^{t'} |u_\eps^N(s)|_{L^4}^4 \|G(u_\eps^N(s))\|_{HS(U_0, H^2)}^2 ds \rp^{1/2} \\
        & \leq C\E \lp \int_0^{t'} \lp |u_\eps^N(s)|_{L^2}^4 + |u_\eps^N(s)|_{L^2}^{7/2} |\partial_x^2 u_\eps^N(s)|_{L^2}^{1/2}\rp  \lp 1 + |u_\eps^N(s)|_{L^2}^2 + |\partial_x^2 u_\eps^N(s)|_{L^2}^2\rp ds \rp^{1/2} \\
        & \leq C\lp 1 +  \E \sup_{t\in [0,T]} |u_\eps^N(t)|_{L^2}^6 \rp + \frac{1}{100} \E \sup_{t \in [0,t']} |\partial_x^2 u_\eps^N(t)|_{L^2}^2 \\
        & \leq C (1+\E|u_0|_{L^2}^6)+ \frac{1}{100} \E \sup_{t \in [0,t']} |\partial_x^2 u_\eps^N(t)|_{L^2}^2.
    \end{aligned}
\end{equation*}
Finally, we will treat the last term $R_{1,4}$ in the expansion of $R_1$. Note that, BDG inequality, the Sobolev inequality $H^{1/3}(\T) = [H^2(\T), L^2(\T)]_{1/6} \hookrightarrow  L^6(\T)$, and the bounds obtained in Theorem \ref{ape:thm1} with $p = 8$ together give us
\begin{equation}
    \begin{aligned}
        \E\sup_{t \in [0,t']} |R_{1,4}(t)| & = C \E \sup_{t \in [0,t']} \bigg| \int_0^t \left( (u_\eps^N(s))^3,  P_N G(u_\eps^N(s))dW(s)\right) \bigg| \\
        & \leq C \E \lp \int_0^{t'} |u_\eps^N(s)|_{L^6}^6 (1+|u_\eps^N(s)|_{L^2}^2) ds \rp^{1/2} \\
        & \leq C\E \lp \int_0^{t'} \lp |u_\eps^N(s)|_{L^2}^6 + |u_\eps^N(s)|_{L^2}^{5} |\partial_x^2 u_\eps^N(s)|_{L^2}\rp  \lp 1 + |u_\eps^N(s)|_{L^2}^2 \rp ds \rp^{1/2} \\
        & \leq C \biggl( 1 + \E\sup_{t \in [0,T]}|u_\eps^N(t)|_{L^2}^8 \biggr) + \frac{1}{100} \E\sup_{t\in [0,t']}|\partial_x^2 u_\eps^N(t)|_{L^2}^2 \\
        & \leq C (1+\E|u_0|_{L^2}^8) + \frac{1}{100} \E\sup_{t\in [0,t']}|\partial_x^2 u_\eps^N(t)|_{L^2}^2.
    \end{aligned}
\end{equation}
Collecting the estimates for $R_{1,j}$, $j=1,2,3,4$ derived above, we obtain for $R_1$ that
\begin{equation*}
    \E \sup_{t \in [0,t']}|R_1(t)| \leq C(1+\E|u_0|_{L^2}^8) + \frac{5}{100} \E \sup_{t \in [0,t']} |\partial_x^2 u_\eps^N(t)| + C \E \int_0^{t'} |\partial_x^2 u_\eps^N(s)|_{L^2}^2 ds.
\end{equation*}
Estimates for the term $R_2$ are derived identically to those for the term $R_1$ and hence we skip these calculations here and state the bounds we obtain ultimately,
\begin{equation*}
    \E \sup_{t \in [0,t']}|R_2(t)| \leq C(1+\E|u_0|_{L^2}^8) + \frac{5}{100} \E \sup_{t \in [0,t']} |\partial_x^2 u_\eps^N(t)| + C \E \int_0^{t'} |\partial_x^2 u_\eps^N(s)|_{L^2}^2 ds.
\end{equation*}
Now we turn to the term $R_3$. We recall that $\{e_n\}_{n=0}^\infty$ is an orthonormal basis for $U_0$.  We further recall that the second Frechet derivative of $\mathcal{I}_2$ is given in \eqref{derivativeI2}.
Then the integrand of $R_3$ can be written as, 
\begin{equation} 
    \begin{aligned}\label{r3}
     \frac{1}{2}&\sum_n \langle  (P_NG(u_\eps^N(s)))^* \mathcal{I}_2''(u_\eps^N(s)) P_NG(u_\eps^N(s)) e_n, e_n \rangle_{U_0} \\
    & = \frac{1}{2}\sum_n \langle   \mathcal{I}_2''(u_\eps^N(s))P_NG(u_\eps^N(s)) e_n, P_NG(u_\eps^N(s)) e_n \rangle \\
    & =  \sum_n \bigg[ |\partial_x^2 P_N G (u_\eps^N(s))e_n|_{L^2}^2 + \frac{5}{6} ((u_\eps^N(s))^2 P_N G (u_\eps^N(s))e_n, P_N G (u_\eps^N(s))e_n)_{L^2} \\
    &  + \frac{5}{6} \int_\T u_\eps^N(s,x)  \partial_x (P_N G(u_\eps^N(s, x)) e_n)^2 dx + \frac{5}{3} ( u_\eps^N(s)\partial_x^2 P_N G(u_\eps^N(s)) e_n, P_N G(u_\eps^N(s))e_n)_{L^2} \\
    &  + \frac{5}{3} (P_N G (u_\eps^N(s))e_n \partial_x^2 u(s), P_N G(u_\eps^N(s)) e_n   )_{L^2} \bigg] \\
    & := R_{3,1}(s) + \cdots + R_{3,5}(s).
    \end{aligned}
\end{equation}
We will find bounds for each term $R_{3,j}$, $j=1,2,3,4,5$ individually.  For $R_{3,1}$, the growth assumption \eqref{assumption: growth of G} immediately gives us,
\begin{equation} \label{temporary:R31}
    \begin{aligned}
        |R_{3,1}(s)| \leq \sum_n |P_N G (u_\eps^N(s))e_n|^2_{H^2} \leq \|G(u_\eps^N(s))\|_{HS(U_0, H^2)}^2 \leq C (1 + |u_\eps^N(s)|^2_{L^2} + |\partial_x^2 u_\eps^N(s) |_{L^2}^2).
    \end{aligned}
\end{equation}
The next term $R_{3,2}$ is treated similarly by using \eqref{assumption: growth of G},
\begin{equation*}
    \begin{aligned}
        |R_{3,2}(s)| & = |\frac{5}{6} \sum_n  ((u_\eps^N(s))^2 P_N G(u_\eps^N(s)) e_n, P_N G(u_\eps^N(s)) e_n)| \leq \frac{5}{6} |u_\eps^N(s)|_{L^\infty}^2 \|G(u_\eps^N(s))\|^2_{HS(U_0, L^2)} \\
        & \leq C  \lp |u_\eps^N(s)|_{L^2}^2 + |u_\eps^N(s)|_{L^2} |\partial_x u_\eps^N(s)|_{L^2} \rp \lp 1 + |u_\eps^N(s)|_{L^2}^2 \rp \\
        & \leq C(|\partial_x^2 u_\eps^N (s)|_{L^2}^2 + |u_\eps^N(s)|_{L^2}^6 + 1).
    \end{aligned}
\end{equation*}
By using a similar argument along with the interpolation inequality $[H^2(\T), L^2(\T)]_{1/8} = H^{1/4}(\T) \hookrightarrow L^4(\T)$ and the H\"{o}lder inequality, we obtain for $R_{3,3}$ that
\begin{equation*}
    \begin{aligned}
        |R_{3,3}(s)|& = |\frac{5}{6} \sum_n \bigg( \partial_x^2 u_\eps^N(s), (P_N G(u_\eps^N(s))e_n)^2  \bigg)_{L^2} |\leq \frac{5}{6} |\partial_x^2 u_\eps^N(s)|_{L^2} \sum_n |P_N G(u_\eps^N(s))e_n|_{L^4}^2 \\
        & \leq \frac{5}{6} |\partial_x^2 u_\eps^N(s)|_{L^2} \sum_n |P_N G(u_\eps^N(s))e_n|_{L^4}^2 \\
        & \leq C |\partial_x^2 u_\eps^N(s)|_{L^2} \left( \sum_n |P_N G(u_\eps^N(s))e_n|_{H^2}^2 \right)^{1/8}  \left( \sum_n |P_N G(u_\eps^N(s))e_n|_{L^2}^2 \right)^{7/8} \\
        & \leq  C |\partial_x^2 u_\eps^N(s)|_{L^2} \left( 1 + |u_\eps^N(s)|_{L^2}^{1/4} + |\partial_x^2 u_\eps^N(s)|_{L^2}^{1/4}\right)  \left( 1 + |u_\eps^N(s)|_{L^2}^{7/4} \right) \\
        & \leq C(|\partial_x^2 u_\eps^N(s)|^2_{L^2} + |u_\eps^N(s)_{L^2}^6 + 1).
    \end{aligned}
\end{equation*}
We next turn to the term $R_{3,4}$. After integrating by parts twice, we write 
\begin{equation*}
	\begin{aligned}
		R_{3,4}(s) & = -\frac{10}{3} \sum_n \bigl( \partial_x P_N G(u_\eps^N(s)) e_n, (\partial_x u_\eps^N(s))(P_N G(u_\eps^N(s))e_n) \bigr)_{L^2} - \frac{10}{3} \sum_n \bigl( (\partial_x P_N G(u_\eps^N(s))e_n)^2, u_\eps^N(s) \bigr)_{L^2} \\
		& = \frac{5}{3} \sum_n \bigl( \partial_x^2 u_\eps^N(s), (P_N G(u_\eps^N(s))e_n)^2  \bigr)_{L^2} - \frac{10}{3} \sum_n \bigl( (\partial_x P_N G(u_\eps^N(s))e_n)^2, u_\eps^N(s) \bigr)_{L^2} .
	\end{aligned}
\end{equation*}
Both terms on the right hand side of the equation above can be estimated similar to $R_{3,3}$. For the second term, we additionally use the interpolation identity $[H^2(\T), L^2(\T)]_{5/8} = H^{5/4}(\T)$, and conclude that
\begin{equation*}
	\begin{aligned}
		|R_{3,4}(s)|
		& \leq C(|\partial_x^2 u_\eps^N(s)|^2_{L^2} + |u_\eps^N(s)|_{L^2}^6 + 1).
	\end{aligned}
\end{equation*}
The final term $R_{3,5}$ is treated identically:
\begin{equation*}
    \begin{aligned}
        |R_{3,5}(s)| \leq C(|\partial_x^2 u_\eps^N(s)|^2_{L^2} + |u_\eps^N(s)|_{L^2}^6 + 1).
    \end{aligned}
\end{equation*}
Combining the five estimates for the right side of \eqref{r3}, applying on $\E\sup_{t \in [0,t']}$ both sides and recalling the estimates in Theorem \ref{ape:thm1} with $p = 6$, we conclude that 
\begin{equation*}
    \begin{aligned}
        \E \sup_{t \in [0,t']} |R_3(t)| \leq C \lp 1+   \E|u_0|_{L^2}^6 +  \E \int_0^{t'} |\partial_x^2 u_\eps^N(s)|_{L^2}^2ds \rp .
          \end{aligned}
\end{equation*}
Next, we turn our attention to the final term $R_4$ appearing in the equation \eqref{ito3}. Using the definitions \eqref{I2} and \eqref{derivativeI2} we write 
\begin{equation}
    \begin{aligned} \label{R4}
        R_4&(t)  = \int_0^t \int_{E_0} \int_\T \bigg( (\partial_x^2 P_N K(u_\eps^N(s), \xi))^2 + \frac{5}{6} (u_\eps^N(s))^2 (P_N K(u_\eps^N(s), \xi))^2 \\
        &  + \frac{5}{9} u_\eps^N (P_N K(u_\eps^N(s), \xi))^3  + \frac{5}{9}(P_N K(u_\eps^N(s), \xi))^4  - \frac{5}{3} \big( u_\eps^N(s)(\partial_x P_N K(u_\eps^N(s), \xi))^2  \\
        &  + P_N K(u_\eps^N(s), \xi) \big[ (\partial_x P_N K(u_\eps^N(s), \xi))^2 + 2(\partial_x u_\eps^N(s))(\partial_x P_N K(u_\eps^N(s), \xi))\big] \big) \bigg) dx d\nu(\xi)ds\\
        & =: (R_{4,1} + \cdots + R_{4,7})(t).
    \end{aligned}
\end{equation}
Next, we will find bounds for the terms $R_{4,j}$, $j=1,2,..,7$ individually. First, observe that the growth assumption \eqref{assumption: growth of K} for the noise coefficient $K$ gives us that
\begin{equation*}
    \begin{aligned}
        |R_{4,1}(t)| & \leq \int_0^t \int_{E_0} |P_N \partial_x^2 K(u_\eps^N(s), \xi)|_{L^2}^2 d\nu(\xi)ds  \leq \int_0^t \int_{E_0} |K(u_\eps^N(s), \xi)|_{H^2}^2 d\nu(\xi)ds \\
        &  \leq C  \int_0^t ( 1 + |u_\eps^N(s)|_{L^2}^2 +  |\partial_x^2 u_\eps^N(s)|_{L^2}^2 ) ds.
    \end{aligned}
\end{equation*}
Next, observe that the growth assumption \eqref{assumption: growth of K} along with the interpolation $  H^{1} (\T)= [H^2(\T), L^2(\T)]_{1/2} \hookrightarrow L^\infty(\T)$ gives us that
\begin{equation*}
    \begin{aligned}
        |R_{4,2} (t)|& \leq \frac{5}{6}\int_0^t |u_\eps^N(s)|_{L^\infty}^2 \biggl(\int_{E_0} |P_N K(u_\eps^N(s), \xi)|_{L^2}^2 d\nu(\xi) \biggr) ds \\
         & \leq C \int_0^t (|u_\eps^N(s)|_{L^2}^2 + |u_\eps^N(s)|_{L^2} |\partial_x^2 u_\eps^N(s)|_{L^2})(1+|u_\eps^N(s)|_{L^2}^2)ds\\
        & = C \int_0^t \biggl[ |u_\eps^N(s)|_{L^2}^2 + |u_\eps^N(s)|_{L^2}^4 + |u_\eps^N(s)|_{L^2}|\partial_x^2 u_\eps^N(s)|_{L^2} + |u_\eps^N(s)|_{L^2}^{3}|\partial_x^2 u_\eps^N(s)|_{L^2} \biggr] ds\\
        & \leq C \left( 1 + \int_0^t \bigl[ |u_\eps^N(s)|_{L^2}^6 +  |\partial_x^2 u_\eps^N(s)|_{L^2}^2  \bigr] ds\right)  .
    \end{aligned}
\end{equation*}
Next, we use the interpolation identity $H^{1/6} (\T)= [H^2(\T), L^2(\T)]_{1/12} \hookrightarrow L^3(\T)$, the Sobolev embedding $H^2(\T) \hookrightarrow L^\infty(\T)$ and the assumption \eqref{assumption: growth of K} to obtain, for $R_{4,3}$, the following estimate: 
\begin{multline*}
    \begin{aligned}
      |  R_{4,3}(t)| & \leq \frac{5}{9} \int_0^t |u_\eps^N(s)|_{L^\infty} \biggl(  \int_{E_0} |P_N K(u_\eps^N(s), \xi)|_{L^3}^3 d \nu(\xi) \biggr) ds \\
        & \leq C  \int_0^t |u_\eps^N(s)|_{L^\infty} \biggl( \int_{E_0} |K(u_\eps^N(s), \xi)|_{L^2}^{11/4} \bigl(|K(u_\eps^N(s), \xi)|_{L^2}^{1/4} + |\partial_x^2 K(u_\eps^N(s), \xi)|_{L^2}^{1/4} \bigr) d\nu(\xi) \biggr) ds \\
        & \leq  C \int_0^t |u_\eps^N(s)|_{L^\infty} \biggl(\int_{E_0} \bigl( |K(u_\eps^N(s), \xi)|_{L^2}^{3} + |K(u_\eps^N(s), \xi)|_{L^2}^{11/4}|\partial_x^2 K(u_\eps^N(s), \xi)|_{L^2}^{1/4} \bigr) d\nu(\xi)\bigg) ds.
  \end{aligned}
\end{multline*}
Next, we apply the H\"{o}lder inequality to the second term of the inequality above, to obtain
\begin{multline*}
	\begin{aligned}
  |  &R_{4,3}(t)|          \leq  C \int_0^t |u_\eps^N(s)|_{L^\infty} \lp \int_{E_0}|K(u_\eps^N(s), \xi)|_{L^2}^{22/7} d\nu(\xi) \rp^{7/8}\lp \int_{E_0}|\partial_x^2 K(u_\eps^N(s), \xi)|_{L^2}^2 d\nu(\xi)\rp^{1/8} ds  \\
        & \quad + C \int_0^t |u_\eps^N(s)|_{L^\infty} (1 + |u_\eps^N(s)|_{L^2}^3) ds \\
        & \leq C \int_0^t (|u_\eps^N(s)|_{L^2} + |\partial_x^2 u_\eps^N(s)|_{L^2}) \lp 1 + |u_\eps^N(s)|_{L^2}^3 + |\partial_x^2 u_\eps^N(s) |_{L^2}^{1/4} + |u_\eps^N(s)|_{L^2}^{11/4}|\partial_x^2 u_\eps^N(s) |_{L^2}^{1/4} \rp ds \\
        & \leq C \int_0^t \bigl( 1 + |u_\eps^N(s)|_{L^2}^6 + |\partial_x^2 u_\eps^N(s)|_{L^2}^2 \bigr) ds.
    \end{aligned}
\end{multline*}
A similar argument along with the continuous embedding $H^{1/4}(\T) = [H^2(\T), L^2(\T)]_{1/8} \hookrightarrow L^4(\T)$, the growth assumption \eqref{assumption: growth of K} and \eqref{assumption:interpolate K}, gives us
\begin{equation*}
    \begin{aligned}
        |R_{4,4} (t)| & \leq C \int_0^t \int_{E_0} |P_N K(u_\eps^N(s), \xi)|_{L^4}^4 d\nu(\xi) ds \\
        & \leq  C\int_0^t \int_{E_0} \biggl( |K(u_\eps^N(s), \xi)|_{L^2}^4 + |K(u_\eps^N(s), \xi)|_{L^2}^{7/2}|\partial_x^2 K(u_\eps^N(s), \xi)|_{L^2}^{1/2} \biggr) d\nu(\xi) ds\\
        & \leq C\int_0^t \biggl[ (1+|u_\eps^N(s)|_{L^2}^4) + \lp \int_{E_0} |K(u_\eps^N(s), \xi)|_{L^2}^{14/3} d\nu(\xi) \rp^{3/4}\lp \int_{E_0}|\partial_x^2 K(u_\eps^N(s), \xi)|_{L^2}^{2} d\nu(\xi) \rp^{1/4} \biggr] ds \\
        & \le C \int_0^t \bigl((1+|u_\eps^N(s)|_{L^2}^4) + (1+|u_\eps^N(s)|_{L^2}^{7/2}) (1+|u_\eps(s)^N|_{L^2}^{1/2} + |\partial_x^2 u_\eps^N(s)|_{L^2}^{1/2} \bigr)ds \\
        & \le C \int_0^t \bigl( 1 + |u_\eps^N(s)|_{L^2}^6 + |\partial_x^2 u_\eps^N(s)|_{L^2}^2 \bigr) ds.
    \end{aligned}
\end{equation*}
The bounds for $R_{4,5}$ are straightforward:
\begin{equation*}
    \begin{aligned}
        |R_{4,5}(t)|  &\leq C \int_0^t \int_{E_0} (u_\eps^N(s), (\partial_x P_N K(u_\eps^N(s), \xi))^2)_{L^2} d\nu(\xi) ds \\
         &\leq C \int_0^t |u_\eps^N(s)|_{L^2} \biggl( \int_{E_0} |\partial_x P_N  K(u_\eps^N(s), \xi)|_{L^4}^2 d\nu(\xi) \biggr) ds .
    \end{aligned}
\end{equation*}
Next, by using the Sobolev interpolation $H^{5/4} (\T)= [H^2(\T), L^2(\T)]_{5/8}$, the continuous embedding $H^{1/4}(\T) \hookrightarrow L^4(\T)$, the growth assumption \eqref{assumption: growth of K} and \eqref{assumption:interpolate K} we obtain that
\begin{equation*}
    \begin{aligned}
    \int_{E_0}&|\partial_x P_N K(u_\eps^N(s), \xi)|^2_{L^4} d\nu(\xi)\leq C  \int_{E_0} \bigl( |K(u_\eps^N(s), \xi)|_{L^2}^{3/4}|\partial_x^2 K(u_\eps^N(s), \xi)|_{L^2}^{5/4}  +  |K(u_\eps^N(s), \xi)|_{L^2}^2 \bigr) d\nu(\xi)  \\
    & \le C \lp \int_{E_0}|K(u_\eps^N(s), \xi)|_{L^2}^2 d\nu(\xi) \rp^{3/8} \lp \int_{E_0}|\partial_x^2 K(u_\eps^N(s), \xi)|_{L^2}^2 d\nu(\xi)\rp^{5/8} \\
     & \quad + C\int_{E_0}  |K(u_\eps^N(s), \xi)|_{L^2}^2 d\nu(\xi)  \\
    & \le C\biggl( \lp 1 + |u_\eps^N(s)|_{L^2}^2 \rp^{3/8} \lp 1 + |u_\eps^N(s)|_{L^2}^2 + |\partial_x^2 u_\eps^N(s)|_{L^2}^2 \rp^{5/8} + \lp 1 + |u_\eps^N(s)|_{L^2}^2 \rp\biggr)  \\
    & \le C \bigl(1 + |u_\eps^N(s)|_{L^2}^2 + |u_\eps^N(s)|_{L^2}^4 + |u_\eps^N(s)|_{L^2}^{3/4} |\partial_x^2 u_\eps^N(s)|_{L^2}^{5/4} \bigr).
    \end{aligned}
\end{equation*}
Thus it follows that
\begin{equation*}
    \begin{aligned}
        |R_{4,5}(t)| & \leq C\int_0^t \bigl( |u_\eps^N(s)|_{L^2}^3 + |u_\eps^N(s)|_{L^2}^5 + |u_\eps^N(s)|_{L^2}^{7/4} |\partial_x^2 u_\eps^N(s)|_{L^2}^{5/4} \bigr)ds  \\
        &\leq C \int_0^t \bigl( 1 + |u_\eps^N(s)|_{L^2}^6 + |\partial_x^2 u_\eps^N(s)|_{L^2}^{2} \bigr) ds .
    \end{aligned}
\end{equation*}
Next, integrating by parts, we rewrite $R_{4,6}$ as follows:
\begin{equation*}
    \begin{aligned}
        R_{4,6}(t) & = \int_0^t \int_{E_0} \bigl(P_N K(u_\eps^N(s), \xi), (\partial_x P_N K(u_\eps^N(s), \xi))^2 \bigr)_{L^2} d\nu(\xi) ds \\
        & = \frac{1}{2} \int_0^t \int_{E_0} \bigl(\partial_x (P_N K(u_\eps^N(s), \xi))^2, \partial_x(P_N K(u_\eps^N(s), \xi)) \bigr)_{L^2} d\nu(\xi) ds \\
        & = -\frac{1}{2}\int_0^t \int_{E_0} \bigl((P_N K(u_\eps^N(s), \xi))^2, \partial_x^2(P_N K(u_\eps^N(s), \xi)) \bigr)_{L^2} d\nu(\xi) ds.
    \end{aligned}
\end{equation*}
We estimate the right hand side above by using the Sobolev interpolation inequality and the continuous embedding $H^{1/4}(\T) = [H^2(\T), L^2(\T)]_{1/8} \hookrightarrow L^4(\T)$, as follows,
\begin{equation*}
	\begin{aligned}
        R_{4,6}(t) &  \leq C \int_{0}^t \int_{E_0} |P_N K(u_\eps^N(s), \xi)|_{L^4}^2 |\partial_x^2 P_N K(u_\eps^N(s), \xi)|_{L^2} d\nu(\xi) ds \\
        & \le C \int_{0}^t \int_{E_0} |K(u_\eps^N(s), \xi)|_{L^2}^{7/4}(|K(u_\eps^N(s), \xi)|_{L^2}^{1/4} + |\partial_x^2 K(u_\eps^N(s), \xi)|_{L^2}^{1/4}) |\partial_x^2 K(u_\eps^N(s), \xi)|_{L^2}d\nu(\xi) ds \\
        & \leq C \int_0^t \int_{E_0} \bigl( |K(u_\eps^N(s), \xi)|^2_{L^2} |\partial_x^2  K(u_\eps^N(s), \xi)|_{L^2} + | K(u_\eps^N(s), \xi)|^{7/4}_{L^2} |\partial_x^2 K(u_\eps^N(s), \xi)|^{5/4}_{L^2} \bigr) d\nu(\xi) ds.
    \end{aligned}
\end{equation*}
We will next find bounds for the two terms on the right hand side of the inequality above. For the first term, by using the growth assumption \eqref{assumption: growth of K} and \eqref{assumption:interpolate K} we arrive at, 
\begin{equation*}
    \begin{aligned}
    & \int_0^t \int_{E_0} | K(u_\eps^N(s), \xi)|^2_{L^2} |\partial_x^2 K(u_\eps^N(s), \xi)|_{L^2} d\nu(\xi) ds \\
    &  \leq \int_0^t \int_{E_0} \bigl( |K(u_\eps^N(s), \xi)|^4_{L^2} + |\partial_x^2  K(u_\eps^N(s), \xi)|_{L^2}^2  \bigr) d\nu(\xi) ds  \\
    & \le C \int_0^t \bigl( 1+|u_\eps^N(s)|^4_{L^2} + |u_\eps^N(s)|_{L^2}^2 + |\partial_x^2 u_\eps^N(s)|_{L^2}^2 \bigr) ds \le C \int_0^t (1 + |u_\eps^N(s)|_{L^2}^6 + |\partial_x^2 u_\eps^N(s)|_{L^2}^2) ds.
    \end{aligned}
\end{equation*}
For the remaining term, we find that
\begin{equation*}
    \begin{aligned}
    \int_0^t \int_{E_0}& |P_N K(u_\eps^N(s), \xi)|^{7/4}_{L^2} |\partial_x^2 P_N K(u_\eps^N(s), \xi)|^{5/4}_{L^2} d\nu(\xi) ds \\
    & \leq \int_0^t \lp \int_{E_0} |K(u_\eps^N(s), \xi)|^{14/3}_{L^2} d\nu(\xi)\rp^{3/8}\lp \int_{E_0} |\partial_x^2 K(u_\eps^N(s), \xi)|_{L^2}^2 d\nu(\xi)\rp^{5/8} ds \\
    & \le C \int_0^t (1 + |u_\eps^N(s)|_{L^2}^{14/3})^{3/8} (1+|u_\eps^N(s)|_{L^2}^2 + |\partial_x^2 u_\eps^N(s)|_{L^2}^2)^{5/8} ds \\
    & \le C \int_0^t (1 + |u_\eps^N(s)|_{L^2}^6 + |\partial_x^2 u_\eps^N(s)|_{L^2}^2) ds.
    \end{aligned}
\end{equation*}
The two estimates above imply for $R_{4,6}$ that
\begin{equation*}
    \begin{aligned}
        |R_{4,6}(t)| \le C \int_0^t (1 + |u_\eps^N(s)|_{L^2}^6 + |\partial_x^2 u_\eps^N(s)|_{L^2}^2) ds.
    \end{aligned}
\end{equation*}
The final term $R_{4,7}$ in \eqref{R4} is treated identically as the term $R_{4,6}$ and we obtain,
\begin{equation*}
    \begin{aligned}
      |  R_{4,7}(t)| \le C \int_0^t (1 + |u_\eps^N(s)|_{L^2}^6 + |\partial_x^2 u_\eps^N(s)|_{L^2}^2) ds.
    \end{aligned}
\end{equation*}
Now, we collect all the estimates for the terms $R_{4,j}; j=1,..,7$ appearing in \eqref{R4} and take $\E\sup_{t \in [0,t']}$ both sides. We summarize our findings for the final right hand side term $R_4$ of \eqref{ito3}:
\begin{equation*}
    \E \sup_{t\in [0,t']} |R_4(t)| \le C \E \int_0^{t'} (1 + |u_\eps^N(s)|_{L^2}^6 + |\partial_x^2 u_\eps^N(s)|_{L^2}^2) ds.
\end{equation*}
Furthermore, thanks to the bounds obtained in Theorem \ref{ape:thm1} by setting $p = 6$, we conclude that
\begin{equation*}
	\E \sup_{t\in [0,t']} |R_4(t)| \le C \biggl(  1 + \E|u_0|_{L^2}^6 +  \E \int_0^{t'} |\partial_x^2 u_\eps^N(s)|_{L^2}^2 ds\biggr).
\end{equation*}
Having found appropriate bounds for the right hand side terms of \eqref{ito3}, we now turn to the left hand side terms of \eqref{ito3} defined in \eqref{Lterms}. We start with $L_1$. Integrating by parts and by using the Sobolev embedding $ H^{1/4}(\T) = [H^2(\T), L^2(\T)]_{1/8} \hookrightarrow L^4(\T)$ we obtain that,
\begin{equation*}
    \begin{aligned}
        L_1 (t) & = -\frac{5}{3} \int_\T u_\eps^N(t,x) (\partial_x u_\eps^N(t,x))^2 dx =  \frac{5}{6}\int_\T (u_\eps^N(t,x))^2 \partial_x^2 u_\eps^N (t,x) dx\\
        & \leq \frac{5}{6} |u_\eps^N(t)|_{L^4}^2 |\partial_x^2 u_\eps^N(t)|_{L^2} \\
        & \leq C (|u_\eps^N(t)|_{L^2}^2 + |u_\eps^N(t)|_{L^2}^{7/4}|\partial_x^2 u_\eps^N(t)|_{L^2}^{1/4})|\partial_x^2 u_\eps^N(t)|_{L^2}\\
        & \leq \frac{1}{8} |\partial_x^2 u_\eps^N(t)|_{L^2}^2 + C \bigl(|u_\eps^N(t)|_{L^2}^8 + 1\bigr) .
    \end{aligned}
\end{equation*}
The term $L_2$ is handled similarly,
\begin{equation*}
    \begin{aligned}
        |L_2 (t)| & \leq \frac{1}{8} \int_0^t |\partial_x^2 u_\eps^N(s)|_{L^2}^2 ds + C \int_0^t (1+|u_\eps^N(s)|_{L^2}^8 ) ds.
    \end{aligned}
\end{equation*}
For $L_3$, we use the embedding $H^{1/4}(\T) \hookrightarrow L^4(\T)$ and the Sobolev interpolation inequality $ H^{5/4} (\T)= [H^4(\T), L^2(\T)]_{5/16}$ to obtain,
\begin{equation*}
    \begin{aligned}
        |L_3(t)| & = C \eps \int_0^t \int_\T  (\partial_x u_\eps^N(s,x))^2(\partial_x^4 u_\eps^N(s,x)) dx ds  \leq C \eps \int_0^t |\partial_x u_\eps^N(s)|_{L^4}^2 |\partial_x^4 u_\eps^N(s)|_{L^2} ds \\
        & \le C \eps\int_0^t \biggl( |u_\eps^N(s)|_{L^2}^2|\partial_x^4 u_\eps^N(s)| + |u_\eps^N(s)|_{L^2}^{11/8}|\partial_x^4 u_\eps^N(s)|_{L^2}^{13/8}  \biggr) ds \\
        & \leq \frac{\eps}{50}\int_0^t |\partial_x^4 u_\eps^N(s)|_{L^2}^2 ds + C \eps \int_0^t (1 + |u_\eps^N(s)|_{L^2}^8) ds.
    \end{aligned}
\end{equation*}
Similarly, the embedding $ H^{1/3}(\T) = [H^4(\T), L^2(\T)]_{1/12} \hookrightarrow L^6(\T)$ gives us for the final term $L_4$ that,
\begin{equation*}
    \begin{aligned}
        |L_4(t)| & = C \eps \int_0^t \int_\T  (u_\eps^N(s))^3(\partial_x^4 u_\eps^N(s)) ds 
         \leq C \eps \int_0^t |u_\eps^N(s)|_{L^6}^3 |\partial_x^4 u_\eps^N(s)|_{L^2} ds  \\
        & \le C \int_0^t \biggl( |u_\eps^N(s)|_{L^2}^3|\partial_x^4 u_\eps^N(s)| + |u_\eps^N(s)|_{L^2}^{11/4}|\partial_x^4 u_\eps^N(s)|_{L^2}^{5/4}  \biggr) ds\\
        & \leq \frac{\eps}{50}\int_0^t |\partial_x^4 u_\eps^N(s)|_{L^2}^2 ds + C \eps \int_0^t (1 + |u_\eps^N(s)|_{L^2}^8) ds.
    \end{aligned}
\end{equation*}
Therefore, combining above estimates for $L_j$, $j = 1,2,3,4$, taking $\E\sup_{t \in [0,t']}$ and applying Theorem \ref{ape:thm1} for $p = 8$, we obtain for the left hand side terms of \eqref{ito3} collectively that,
\begin{equation*}
	\E\biggl[ \sup_{t \in[0,t']}|(L_1+L_2+L_3+L_4)(t)| \biggr] \leq \frac{1}{4}\E\sup_{t \in[0,t']}|\partial_x^2 u_\eps^N(t)|_{L^2}^2 + \frac{\eps}{25} \int_0^t |\partial_x^4u_\eps^N(s)|_{L^2}^2ds + C(1+\E|u_0|_{L^2}^8). 
\end{equation*}
Hence, by collecting all the estimates for the terms $R_i, L_i; i=1,..,4$ appearing in \eqref{ito3}, we conclude that
\begin{equation*}
	\E\sup_{t\in [0,t']}|\partial_x^2 u_\eps^N(t)|_{L^2}^2 + \eps \E \int_0^{t'}|\partial_x^4 u_\eps^N(s)|_{L^2}^2 ds \leq C\biggl( 1 +  \E|u_0|_{L^2}^8 + \E|\partial_x^2 u_0|_{L^2}^2 + \E\int_0^{t'}|\partial_x^2 u_\eps^N(s)|_{L^2} ds \biggr).
\end{equation*}
We then apply Gronwall's Lemma to arrive at our desired result \eqref{ape: partial^2 u_eps^N}, i.e., 
\[
\E\sup_{t \in [0,T]} \biggl[ |\partial^2_x u^{N}_\eps(t)|_{L^2}^2 \biggr] + \eps \E \int_0^T |\partial_x^4 u^N_\eps(s)|_{L^2}^2 ds \leq Ce^{CT}\lp \E|u_0|_{L^2}^8 + \E|\partial_x^2 u_{0}|_{L^2}^2 + 1  \rp.
\]
In particular, using the equivalence of norms given in \eqref{Hs norm equivalent} for $m=2$, and the lower order bounds obtained in Theorem \ref{ape:thm1}, we obtain \eqref{ape: h2} from \eqref{ape: partial^2 u_eps^N}.
This concludes the proof of Theorem \ref{ape:thm3}.
\end{proof}
\subsection{Tightness} \label{section:tightness}
Thanks to the bounds obtained in Theorems \ref{ape:thm1} - \ref{ape:thm3}, we obtain convergence of the approximate solutions $u^N_\eps$, up to a subsequence, in the weak-* topology of the respective spaces. However, to pass  $N\to \infty$ in the non-linear term in the equation \eqref{eqn: galerkin relabel}, these weak-* convergence results will not be sufficient. Hence, to upgrade this mode of convergence, we will derive (compactness) results involving tightness of the laws of the aforementioned approximate solutions. Our compactness argument is based on the following version of the Aubin-Lions theorem (see e.g. \cite[Lemma 4.3]{NTT21}):
Since $H^2(\T) \subset\subset H^1(\T)$, the following embedding is compact; 
\begin{equation}\label{compact}
	L^\infty(0,T;H^2(\T)) \cap W^{1/4, 2}(0,T; H^{-2}(\T)) \subset \subset L^p(0,T; H^1(\T))\qquad\forall p>1.
\end{equation}
 Thus, to apply \eqref{compact}, we derive temporal bounds for the approximations $u_\epsilon^N$ in fractional Sobolev spaces. Our first observation involves the following lemma which provides the desired fractional derivative bounds in time, of order strictly less than $\frac12,$ for the stochastic integrals on the right side of \eqref{eqn: galerkin relabel}.
\begin{lemma}\label{lem:fracstochint}
For every $\alpha \in (0,1/2)$, there exists a constant $C = C(\alpha, T) > 0$ such that:
\begin{align*}
     &\E \bigg\| \int_0^\cdot G(u^N_\eps(s)) dW(s)  \bigg\|_{H^\alpha(0,T;L^2(\T))}^2 \leq C \E \int_0^T \|G(u^N_\eps(s))\|_{HS(U_0, L^2(\T))}^2 ds \leq C ,\\
     &\E \bigg\| \int_0^\cdot \int_{E_0} K(u^N_\eps(s),\xi) d \widehat{\pi}(s, \xi) \bigg\|_{H^\alpha(0,T;L^2(\T))}^2 \leq C \E \int_0^T \int_{E} |K(u^N_\eps(s),\xi)|^2_{L^2(\T)}d\nu(\xi) ds \leq C .
\end{align*}
\end{lemma}
\begin{proof}
	The proof of this theorem is standard. The interested reader is refered to \cite[Lemma 2.1]{flandoli1995martingale} and \cite[Lemma 3.4]{CTT18euler} for a proof.
\end{proof}
\noindent Next, we find bounds for the fractional time derivatives of the remaining (deterministic) drift terms in \eqref{eqn: galerkin relabel}. Then, using Lemma \ref{lem:fracstochint} we will arrive at the desired bounds for the approximations $u^N_\eps$.
\begin{theorem}[Fractional Sobolev estimates] \label{ape:thm_fractional}
    For any fixed $\alpha \in (0, 1/2)$, there exists a constant $C> 0$, depending only on $T, \alpha$ and the initial data $u_0$, such that 
    \begin{equation}
        \E\| u^N_\eps \|^2_{H^{\alpha}(0,T; H^{-2}(\T))} \leq C.
    \end{equation}
\end{theorem}
    \begin{proof}
In the previous result, Lemma \ref{lem:fracstochint}, we found bounds for fractional derivative in time of the stochastic integrals appearing on the right hand side of the equation \eqref{eqn: galerkin relabel}. Hence, now we consider the remaining drift terms in \eqref{eqn: galerkin relabel}:
\begin{equation}
    \begin{aligned}
     & \bigg\| u^N_{0} - \int_0^{\cdot} \partial_x^3 u^N_\eps(s) ds - \int_0^{\cdot} P_N(\theta_N(|u^N_\eps(s)|_{L^2}) u^N_\eps(s) \partial_x u^N_\eps(s)) ds - \int_0^{\cdot} \eps \partial_x^4 u^N_\eps (s) ds\bigg\|_{H^\alpha(0,T; H^{-2})}^2    \\
     & \leq C \bigg\| u^N_{0} - \int_0^{\cdot} \partial_x^3 u^N_\eps(s) ds - \int_0^{\cdot} P_N(\theta_N(|u^N_\eps(s)|_{L^2}) u^N_\eps(s) \partial_x u^N_\eps(s)) ds - \int_0^{\cdot} \eps \partial_x^4 u^N_\eps (s) ds \bigg\|_{H^1(0,T; H^{-2})}^2 \\
     & \leq C \bigg\| u^N_{0} - \int_0^{\cdot} \partial_x^3 u^N_\eps(s) ds - \int_0^{\cdot} P_N(\theta_N(|u^N_\eps(s)|_{L^2}) u^N_\eps(s) \partial_x u^N_\eps(s)) ds - \int_0^{\cdot} \eps \partial_x^4 u^N_\eps (s) ds \bigg\|_{L^2(0,T; H^{-2})}^2 \\
     & + C \biggl( \int_0^T |\partial_x^3 u^N_\eps(s)|_{H^{-2}}^2 ds + \int_0^T |u^N_\eps(s) \partial_x u^N_\eps(s)|_{H^{-2}}^2 ds + \eps^2 \int_0^T|\partial_x^4 u^N_\eps(s)|_{H^{-2}}^2 ds \biggr) \\
     & \leq C \biggl( |u_{0}^N|_{L^2}^2 +  \int_0^T |\partial_x^3 u^N_\eps(s)|_{H^{-2}}^2 ds + \int_0^T \theta_N(|u_\eps^N(s)|_{H^1}) |u^N_\eps(s) \partial_x u^N_\eps(s)|_{H^{-2}}^2 ds + \eps^2 \int_0^T|\partial_x^4 u^N_\eps(s)|_{H^{-2}}^2 ds  \biggr) \\
     & =: C\bigl( |u_{0}^N|_{H^{-2}}^2 + I_1 + I_2 + I_3 \bigr).
    \end{aligned}
\end{equation}
Next, observe that, for any $v \in H^2(\T)$, upon integrating by parts we obtain
\begin{equation*}
	|(\partial_x^3 u_\eps^N, v)_{L^2}| = |(\partial_x u_\eps^N, \partial_x^2 v)_{L^2}| \leq |\partial_x u_\eps^N|_{L^2} |v|_{H^2}.
\end{equation*}
Hence, $ |\partial_x^3 u_\eps^N|_{H^{-2}} \leq |\partial_x u_{\eps}^N|_{L^2}$. 
Similarly, we have that $ |\partial_x^4 u_\eps^N|_{H^{-2}} \leq |\partial^2_x u_{\eps}^N|_{L^2}.$

Hence, by a straightforward application of the bounds obtained in Theorem \ref{ape:thm3}, we conclude that
\begin{equation*}
	\E[I_1 + I_3] \leq C + C\eps^2 \leq C. 
\end{equation*}
For $I_2$, we note by Sobolev embedding $H^2(\T) \hookrightarrow L^4(\T)$ and Theorem \ref{ape:thm3}, that 
\begin{equation*}
    \begin{aligned}
        \E \int_0^T \theta_N(|u_\eps^N(s)|_{H^1})| u^N_\eps(s) \partial_x u_\eps^N(s)|_{H^{-2}}^2 ds \leq C \E \int_0^T  |u_\eps^N(s)|_{L^4}^2 ds  \leq C \E\sup_{t \in [0,T]} |u_\eps^N(t)|_{H^2}^2  \leq C.
    \end{aligned}
\end{equation*}
Combining the estimates for $I_1, I_2, I_3$ we end the proof of Theorem \ref{ape:thm_fractional}.
\end{proof}
Now using this temporal regularity result we prove the following tightness result.
\begin{proposition}\label{tightness:L^2(0,T; L^2)}
The sequence of laws of the solutions $\lp u_\eps^N \rp_{N=1}^\infty$ to \eqref{eqn: galerkin relabel}, is tight in $L^2(0,T; H^1(\T))$.
\begin{proof}
	The proof of this Proposition relies on the compact embedding \eqref{compact}. For that purpose,
we begin by defining the following bounded ball, 
\begin{equation}
    \begin{aligned}
        B_r := \big \{v \in L^2(0, T; H^2(\T)) \cap W^{1/4, 2}(0,T; H^{-2}(\T)): 
        \|v\|_{L^2(0, T; H^2(\T))} + \|v\|_{W^{1/4, 2}(0,T; H^{-2}(\T))} \leq r \big \}.
    \end{aligned}
\end{equation}
 The compact embedding \eqref{compact} then implies that $B_r$ is compact in $L^2(0,T; H^1(\T))$. Thanks to Theorems \ref{ape:thm3} and \ref{ape:thm_fractional} along with an application of the Chebyshev inequality, we obtain
\begin{equation}
    \begin{aligned}
         \P \big[  u_\eps^N \notin B_r \big] & \leq \P\lp \|u_\eps^N \|_{L^2(0,T;H^2(\T))} > \frac{r}{2} \rp + \P\lp \|u_\eps^N \|_{W^{1/4, 2}(0,T; H^{-2}(\T))} > \frac{r}{2} \rp \\
        & \leq \frac{4}{r^2} \lp \E\|u_\eps^N \|_{L^2(0,T;H^2(\T))} + \E \|u_\eps^N \|_{W^{1/4, 2}(0,T; H^{-2}(\T))} \rp \\
        & \leq \frac{C}{r^2},
    \end{aligned}
\end{equation}
where the constant $C$ is independent of $N$ and $\eps$. This shows that the laws of $\lp u_\eps^N \rp_{N=1}^\infty$ form a tight sequence of measures in $L^2(0,T; H^1(\T))$.
\end{proof}
\end{proposition}
Recall that in Definition \ref{definition: martingale} we require the martingale solution to be a process that has c\`adl\`ag sample paths in the space $H^{-2}(\T)$. This criterion gives us the desired measurability of the noise coefficients required for the stochastic integrals to be well-defined. Hence, to obtain a  c\`adl\`ag candidate solution we must establish tightness of the laws of the approximation solutions $u^N_\eps$ in the space of c\`adl\`ag-in-time functions endowed with the Skorohod topology for which this space of c\`adl\`ag functions is separable and metrizable; see below and \cite{billingsley2013convergence}. The Skorohod topology is complicated to analyze. However, a classical result of Aldous \cite{Aldous78} gives us an equivalent condition which we restate in the following theorem.
\begin{theorem}[Aldous tightness criterion \cite{Aldous78}] \label{thm:aldous}
	Let $(S,d)$ be a separable complete metric space and let $X_N = (X_N(t))_{t\in[0,T]}$ be a sequence of $S$-valued stochastic processes defined on a filtered probability space $(\Omega, \F, \P, (\F_t)_{t \geq 0})$ such that each $X_N$ is $(\F_t)_{t\in[0,T]}$-adapted and has c\`{a}dl\`{a}g sample paths in $S$.
	Assume that the following conditions hold:
	\begin{enumerate}[(i)]
		\item \label{aldous:i} There exists a dense subset $D\subset[0,T]$ such that for every $t\in D$, the sequence of laws of $X_N(t)$ is tight in $S$.
		\item \label{aldous:ii} For every $\eta>0$ and $\varepsilon>0$, there exists $\delta>0$ such that for every sequence $(\tau_N)_{N=1}^\infty$ of $(\mathcal F_t)_{t \geq 0}$-stopping times bounded by $T$,
		\begin{equation*}
			\sup_{N\geq 1}\ \sup_{0\leq t\leq \delta} \P\bigl( d(X_N((\tau_N + t) \wedge T), X_N(\tau_N)) > \eta  \bigr) \leq \varepsilon.
		\end{equation*}
	\end{enumerate}
	Then the sequence $(X_N)_{n\ge1}$ is tight in the Skorokhod space $\mathcal{D}([0,T];S)$.
\end{theorem}
By using the Aldous tightness criteria Theorem \ref{thm:aldous}, we will establish the following tightness result for the laws of our approximation solutions.
\begin{proposition}\label{tightness:Dspace}
The sequence of laws of
the approximations $\lp u_\eps^N \rp_{N=1}^\infty$ is tight in the space $\mathcal{D}([0,T]; H^{-2}(\T))$ endowed with the Skorohod topology.
\begin{proof} 
	We will verify the two conditions in Theorem \ref{thm:aldous} for $S := H^{-2}(\T)$ and $X_N := u_\eps^N$. \\
{\bf Step 1:}	First, we will show that the approximation solutions $u^N_\eps$ satisfy the assumption (i) in Theorem \ref{thm:aldous}. In other words, we will prove that
	for each $t \in [0, T]$, the sequence of the laws of the $L^2(\T)$-valued random variables $\lp u_\eps^N (t)\rp_{N=1}^\infty$ is tight in $H^{-2}(\T)$.
	 	
	 	Let ${ B_r := \{v \in L^2(\T): |v|_{{L^2}} \leq r \}}$. We know that by Sobolev embedding theorem, $B_r$ is compactly embedded in {$H^{-2}(\T)$}. We now apply Chebyshev's inequality and Theorem \ref{ape:thm1}, and obtain, for any $t \in [0,T]$, that
	 \begin{equation}
	 	\begin{aligned}
	 		\P( u_\eps^N(t) \notin B_r ) = \P( |u_{\eps}^N(t)|_{L^2} > r ) \leq \frac{
	 		\E\sup_{t\in [0,T]}|u_\eps^N(t)|_{L^2}^2 }{r^2} \leq \frac{C}{r^2},
	 	\end{aligned}
	 \end{equation}
	 where $C$ is independent of $N$ and $t$. This shows that the laws of $\lp u_\eps^N(t) \rp_{N=1}^\infty$ are tight in ${ H^{-2}(\T)}$.

\noindent {\bf Step 2:} Next, we will show that the approximation solutions $u^N_\eps$ satisfy the assumption (ii) in Theorem \ref{thm:aldous}.
In other words, we will  prove that for each sequence $(\tau_N)_{N = 1}^\infty$ of $(\F_t)_{t\geq 0}$-stopping times which are bounded by $T$, the following condition is satisfied $\P$-a.s.
\begin{equation} \label{aldous condition in V_2'}
    \lim_{t' \to 0^+} \sup_{N\in \N}  \sup_{t \in [0,t']}\E\big|u_\eps^N((\tau_N+t) \wedge T) - u_\eps^N(\tau_N)\big|^2_{H^{-2}} = 0.
\end{equation}
We consider the SPDE \eqref{eqn: galerkin relabel} satisfied by the functions $u_\eps^N$ and observe for fixed $t$ by applying the H\"{o}lder inequality, that,
\begin{equation}
    \begin{aligned}
       & \E\big|u_\eps^N((\tau_N+t)\wedge T) - u_\eps^N(\tau_N)\big|_{H^{-2}}^2\\
         & \leq  t \cdot \E\int_{\tau_n}^{(\tau_n+t) \wedge T}|\partial_x^3 u_\eps^N(s)|_{H^{-2}}^2 ds + t \cdot \E \eps^2 \int_{\tau_n}^{(\tau_n+t) \wedge T}|\partial_x^4 u_\eps^N(s)|_{H^{-2}}^2 ds\\
        & + t \cdot \E \int_{\tau_n}^{(\tau_n+t) \wedge T}|\theta_N(|u_\eps^N(s)|_{H^1})u_\eps^N(s) \partial_x u_\eps^N(s)|_{H^{-2}}^2 ds  \\
        & + \E \bigg|\int_{\tau_n}^{(\tau_n+t) \wedge T}P_N G(u_\eps^N(s)) dW(s) \bigg|_{H^{-2}}^2  + \E\bigg| \int_{\tau_n}^{(\tau_n+t) \wedge T}\int_{E_0}P_N K(u_\eps^N(s), \xi) d\widehat{\pi}(\xi, s)  \bigg|_{H^{-2}}^2,  \\
        & =: J_1(t) + J_2(t) + J_3 (t)+ J_4(t) + J_5(t).
    \end{aligned}
\end{equation}
Our aim thus is to show that $\lim_{t' \to 0^+}{\sup_{N\in \N} \sup_{t\in[0,t']}}J_i(t) = 0$ for $i = 1,2,3,4, 5$. Observe that by Theorem \ref{ape:thm3}, there exists some $C$ independent of $N$ and $\eps$ and $R$, that
    \begin{align*}
     J_1 (t)&\leq t \cdot \E \int_{\tau_n}^{(\tau_n+t) \wedge T}|\partial_x^3 u_\eps^N(s)|_{H^{-2}}^2 ds  \leq t \cdot T \E \sup_{t \in [0,T]}|\partial_x^2 u_\eps^N|_{L^2}^2 \leq Ct, \\
  J_2(t) &\leq t\cdot \eps^2 \E \int_{\tau_n}^{(\tau_n+t) \wedge T}|\partial^4_x u_\eps^N(s)|_{L^2}^2 ds \leq Ct, \\
   J_3(t) & \leq C t \cdot \E \int_{\tau_n}^{(\tau_n+t) \wedge T} |u_\eps^N(s)|_{H^2}^2 ds \leq Ct.
    \end{align*}
Thanks to the It\^o isometry, the growth conditions of $G$ and $K$ \eqref{assumption: growth of G}-\eqref{assumption: growth of K} and Theorem \ref{ape:thm1}, the two terms $J_4$ and $J_5$ can be treated similarly:
\begin{align*}
  (J_4(t)+ J_5(t) )&\leq Ct \cdot \E \int_{\tau_n}^{(\tau_n+t) \wedge T}\bigl( 1 + |u_\eps^N(s)|_{L^2}^2 \bigr) ds \leq C t \lp  1 + \E \sup_{t\in [0,T]} |u_\eps^N(t)|_{L^2}^2 \rp \leq Ct.
\end{align*}
We emphasize that this constant $C$ is independent of $t', N$ and $\eps$. It then follows that 
\[
    \sup_{N\in \N} \sup_{t\in [0,t']}\E\big|u_\eps^N((\tau_n+t) \wedge T) - u_\eps^N(\tau_n)\big|_{H^{-2}}^2 \leq \sup_{N\in \N}  \sup_{t \in [0,t']} (J_1(t) + J_2(t) + J_3(t) + J_4 (t)+ J_5(t)) \leq Ct'.
\]
Passing $t' \to 0$ yields the desired result \eqref{aldous condition in V_2'}.
\end{proof}
\end{proposition}
{ 
	Next, we leverage the arguments in the proof of Proposition \ref{tightness:L^2(0,T; L^2)}, Proposition \ref{tightness:Dspace} and \cite[Lemma 2]{motyl2013} to establish the tightness of $\lp u_\eps^N \rp_{N=1}^\infty$ 
in the space
	\begin{equation} \label{chi space}
		\mathcal{X}_u := L^2(0,T; H^1(\T)) \cap \mathcal{D}([0,T]; H^{-2}(\T)) \cap \mathcal{D}([0,T]; H^{2}_w(\T)),
	\end{equation} 
equipped with the product topology.	Here, $H^2_w(\T)$ denotes the space $H^2(\T)$ equipped with the weak topology. 
\begin{proposition} \label{tightness:L2H2weak}
	The sequence of laws of
	the approximations $\lp u_\eps^N \rp_{N=1}^\infty$ is tight in the space $\mathcal{X}_u$.
\end{proposition}
	\begin{proof}
		Fix $\varepsilon > 0$. Following a similar argument to Proposition \ref{tightness:L^2(0,T; L^2)}, we can define $A_1 \subset L^2(0,T;H^1(\T))$ by
	\begin{equation*}
		A_1 := \bigg\{ v \in L^{ \infty}(0,T;H^2(\T)) \cap W^{1/4,2}(0,T;H^{-2}(\T): \sup_{t \in [0,T]} |v(t)|_{H^2} +  \|v \|_{W^{1/4, 2}(0,T;H^{-2})} < R \bigg\},
	\end{equation*}
	which is compact in $L^2(0,T; H^1(\T))$. By Chebyshev's inequality, we may choose $R$ sufficiently large, so that 
	\begin{equation*}
		\P(u_n \notin A_1)  < \frac{\varepsilon}{2}.
	\end{equation*}
	Next, note that by following the proof of Proposition \ref{tightness:Dspace}, we have that the laws of $\lp u_\eps^N \rp_{N=1}^\infty$ are tight in the space $\mathcal{D}([0,T]; H^{-3}(\T))$. In other words, there exists a compact subset $A_2 \subset \mathcal{D}(0,T;H^{-3}(\T))$ such that
	\begin{equation*}
			\P(u_n \notin A_2)  < \frac{\varepsilon}{2}.
	\end{equation*}
	By \cite[Lemma 2]{motyl2013} (see also \cite[Corollary 1]{motyl2013}), we infer that $A_1 \cap A_2$ is a compact subset in $\mathcal{D}([0,T];H_w^2(\T))$ and 
		\begin{equation*}
		 \P(u_n \notin A_1 \cap A_2) \leq \P(u_n \notin A_1) + \P(u_n \notin A_2) < \varepsilon. 
	\end{equation*}
	Since $\varepsilon > 0$ is arbitrary, we conclude that the laws of $\lp u_\eps^N \rp_{N=1}^\infty$ are tight in $\mathcal{D}([0,T];H_w^2(\T))$.
	\end{proof}}
\subsection{Passage to the limit $N \to \infty$}\label{Ntoinfty}

Let $\mathcal{N}_{[0,T)\times E}^{\#}$ be the space of counting measures on  $[0,T) \times E$ that are finite on bounded subsets. We endow $\mathcal{N}_{[0,T)\times E}^{\#}$ with the topology induced by weak-\# convergence. We say that $\mu_n\to\mu$ weakly$-\#$ if, 
	\begin{equation} \label{weak hash conv}
		\int_{[0,T)\times E} f d\mu_n \to \int_{[0,T)\times E} fd\mu
	\end{equation}
for every real-valued $f$ that is bounded, continuous and has bounded support in $[0,T)\times E$.
Thanks to Proposition 9.1.IV in \cite{daley2008introduction}, we know that this space is separable and metrizable by a complete metric. 

	We note that $\lp u^{N}_{0}, u^N_\eps, W, \pi \rp_{N \in \N}$ is tight in the space
\begin{equation}
    \mathcal{X} := L^2(\T) \times \mathcal{X}_u \times C([0,T], U_0) \times \mathcal{N}_{[0,T)\times E}^{\#}
\end{equation}
endowed with the product topology. We denote by $\mu^N_\eps$ the law of $\lp u^{N}_{0}, u^N_\eps, W, \pi \rp$ in $\mathcal{X}$.  By Prohorov's Theorem, there exists a subsequence of $\{\mu^N_\eps\}_{N \in \Lambda}$ for some countable set $\Lambda \subset \N$, such that
\begin{equation}\label{weakconv}
    \mu^N_\eps \text{ converges weakly to } \mu_\eps \text{ as $N\to \infty$ along $\Lambda$ on the space $\mathcal{X}$.}
\end{equation}
This weak convergence of measures will be upgraded to almost sure convergence by an application of the version the Skorohod Representation theorem given by Jakubowski in \cite{jakubowski1998almost}. 
\begin{theorem}\label{thm:skorohod}
    Suppose that  ${u_{0}} \in L^8(\Omega, \F_0, \P; L^2(\T)) \cap L^2(\Omega, \F_0, \P;H^2(\T))$ has law ${ \mu_0}$. 
    Then there exist a probability space $(\widetilde{\Omega}, \widetilde{\F}, \widetilde{\P})$\footnote{$(\widetilde{\Omega}, \widetilde{\F}, \widetilde{\P}) =([0, 1), \text{Borel sets of }[0,1), \text{Lebesgue measure}) $} and 
    $\mathcal{X}$-valued random variables $\lp \widetilde{u}_{0}, \widetilde{u_\eps}, \widetilde{W}, \widetilde{\pi} \rp$ and $\lp \widetilde{u}^N_{0}, \widetilde{u_\eps}^N, \widetilde{W}_{N}, \widetilde{\pi}_{N} \rp_{N\in \Lambda}$ such that 
    \begin{enumerate}[(i)]
        \item  For each $N \in \Lambda$, $\lp \widetilde{u}_{0}^N, \widetilde{u}_\eps^N, \widetilde{W}_N, \widetilde{\pi}_N \rp$ has law $\mu^N_\eps$ , $\lp \widetilde{u}_{0}, \widetilde{u}_\eps, \widetilde{W}, \widetilde{\pi} \rp$ has law $\mu_\eps$ and $\widetilde{u}_{0}$ has law $\mu_{0}$,
     \item The following convergence holds  $\widetilde{\P}$-almost surely,  \begin{equation} \label{as conv: galerkin}
            \lim_{N \to \infty}\lp \widetilde{u}_{0}^N, \widetilde{u}_\eps^N, \widetilde{W}_N, \widetilde{\pi}_N \rp = \lp \widetilde{u}_{ 0}, \widetilde{u_\eps}, \widetilde{W}, \widetilde{\pi} \rp, 
        \end{equation}
        \item For some filtration $( \widetilde{\F}_t)_{t\in[0,T]}$, that satisfies the usual condition, $\tilde W_N$ is an $( \widetilde{\F}_t)_{t\in[0,T]}$-Wiener process, and $\tilde \pi_N$ is a time homogeneous Poisson random measures with intensity measure $dt\otimes d\nu$.
        \item For every $N \in \Lambda$, $\widetilde{u}^N_\eps$ is a solution to the Galerkin scheme \eqref{eqn: galerkin relabel} with respect to the stochastic basis $\lp \widetilde{\Omega}, \widetilde{\F}, \widetilde{\P}, (\TF^N_t)_{t \in [0,T]}, \widetilde{W}_N, \widetilde{\pi}_N \rp$ and $\widetilde{u}_\eps^N(0) = { \widetilde{u}_{0}^N }$, $\widetilde{\P}$-a.s. 
     
    \end{enumerate}
\end{theorem}

The filtration $(\TF^N_t)_{t \in [0,T]}$  is constructed as follows:
Let ${\mathcal{F}^N}'_t$ be the $\sigma$-field generated by the random variables $ \Tu_\eps^N(s), \widetilde{W}_N(s), \widetilde{\pi}_N((0,s] \times \Gamma)$ for all $s \leq t$ and all Borel subsets $\Gamma$ of $E_0$. Then we define
\begin{equation} \label{filtration}
    \begin{aligned}
    \mathcal{N} := \left\{ A \in \TF \mid \TP(A) = 0 \right\}, \quad
    \bar{\mathcal{F}}^N_t:= \sigma(\mathcal{F}'_t \cup \mathcal{N}),\quad 
    \widetilde{\mathcal{F}}^N_t := \bigcap_{s \geq t} \bar{\mathcal{F}}^N_s.
    \end{aligned}
\end{equation}
This gives a complete, right-continuous filtration $(\TF_t^N)_{t \in [0,T]}$, dependent on $N,\eps$ and the cutoff parameter $R$, defined on the new probability space $(\TOmega, \TF, \TP)$, to which the noise processes and the approximate solutions are adapted. The construction of a filtration $(\TF_t^\epsilon)_{t\in[0,T]}$ to which the limiting processes are adapted is identical.

The proof of this theorem is standard and thus we shall skip it.
Parts (i) -(iii) follow straight from the Skorohod Representation Theorem \cite{jakubowski1998almost}. Whereas, Part (iv) is proven in the appendix of \cite{CTT18euler} by using the equivalence of laws of the old and new solutions.
An immediate consequence of Part (iv) is that $\Tu_\eps^N$ inherits all the estimates derived for $u_\eps^N$ in Theorems \ref{ape:thm1}-\ref{ape:thm3}.
\begin{corollary}\label{cor:skorohod estimates}
    Let $\Tu_\eps^N$ be new random variable defined on $\lp \TOmega, \TF, \TP,  (\TF_t)_{t \in [0,T]}, \widetilde{W}, \widetilde{\pi} \rp$ obtained in Theorem \ref{thm:skorohod}. Then $\widetilde{u}_\eps^N$ satisfies the following bounds that are independent of $N$ and $\epsilon$,
    \begin{equation} 
    	\E\bigg[\sup_{t \in [0,T]} |\Tu_\eps^N(t)|^p_{L^2} \bigg] + \eps p \E \int_0^T |\Tu_\eps^N(s)|_{L^2}^{p-2}|\partial_x^2 \Tu_\eps^N(s)|^2_{L^2}ds \leq  Ce^{CT} (\E|u_{0}|_{L^2}^p + 1).
    \end{equation}
 and,
   \begin{equation}
 	\E\biggl[ \sup_{t \in [0,T]} |\Tu_\eps^N(t)|_{H^2}^2 \biggr] + \eps \E \int_0^T |\Tu_\eps^N(s)|_{H^4}^2\,ds \leq Ce^{CT}\lp \E|u_0|_{L^2}^8 +  \E|\partial_x^2 u_0|_{L^2}^2 + 1  \rp.
 \end{equation}   
\end{corollary}
We now recall  Vitali's convergence lemma which we will use several times in the rest of this section to upgrade the almost sure convergence \eqref{as conv: galerkin} to convergence in $L^p$ spaces.
\begin{lemma}[Vitali's Lemma]\label{lem:vitali}
    Let $({\Omega}, {\F}, {\P})$ be a probability space and $X$ be a Banach space. For some $p \in [1, \infty)$ assume that we have a sequence of functions $f, f_1, f_2, \cdots \in L^p({\Omega}, {\F},  {\P}; X)$. If $\{f_n\}_{n=1}^\infty$ satisfies
    \begin{enumerate}[(i)]
        \item $|f_n - f|_X \to 0$ in probability as $n \to \infty$;
        \item $\sup_{n \geq 1} \E|f_n|_X^q < \infty$ for some $q \in (p, \infty)$,
    \end{enumerate}
    then $f_n \to f$ in $L^p({\Omega}, {\F},  {\P}; X)$.
\end{lemma}
First, we will summarize the immediate consequences of Corollary \ref{cor:skorohod estimates} and Lemma \ref{lem:vitali} in the following lemma.
\begin{lemma} \label{thm:convergence}
The limiting stochastic process $\Tu_{\eps} $ obtained in Theorem \ref{thm:skorohod} belongs to the space $ L^8(\TOmega; L^\infty(0,T; L^2(\T)))  \cap L^2(\TOmega; L^2(0,T; H^4(\T)))$, 
    and satisfies the following convergence results:
    \begin{enumerate}[(i)]
        \label{strong conv: galerkin} \item $\Tu_\eps^N \to \Tu_\eps$ strongly in the space $L^1(\TOmega; L^2(0,T; H^1(\T)) ) \cap L^2(\TOmega; L^2(0,T; L^2(\T)))$  as $N\to\infty$ along $\Lambda$, 
        \item $\Tu_\eps^N \rightharpoonup \Tu_\eps$ weakly in the space $L^2(\TOmega; L^2(0,T; H^4(\T)))$ as $N\to\infty$ along $\Lambda$,
        \item $\Tu_\eps^N \stackrel{*}{\rightharpoonup} \Tu_\eps$ in weak* in the space $L^8(\TOmega; L^\infty(0,T; L^2(\T)))$ as $N\to\infty$ along $\Lambda$.
    \end{enumerate}    
    \end{lemma}
    \begin{proof}
        \begin{enumerate}[(i)]
            \item Due to Theorem \ref{thm:skorohod}, we know that $\Tu_\eps^N \to \Tu_\eps$ in the space $L^2(0,T;H^1(\T))$, $\widetilde{\P}$-a.s. This implies that $\Tu_\eps^N \to \Tu_\eps$ in  $L^2(0,T;H^1(\T))$ and $L^2(0,T;L^2(\T))$ $\TP$-a.s. (and thus in probability). Moreover, we have that 
            \[
                \begin{aligned}
                \sup_{N \in \Lambda} \TE \int_0^T |\Tu_\eps^N(s)|_{H^1}^2 ds& \leq \sup_{N \in \Lambda} \lp \TE \sup_{t \in [0,T]} |\Tu_\eps^N(t)|_{H^1}^2 \rp < \infty, \\
                \sup_{N \in \Lambda} \TE \int_0^T |\Tu_\eps^N(s)|_{L^2}^8ds & \leq \sup_{N \in \Lambda} \lp \TE \sup_{t \in [0,T]} |\Tu_\eps^N(t)|_{L^2}^8 \rp < \infty.
                \end{aligned}
            \]
            By using Lemma \ref{lem:vitali}, we infer that the convergence holds in $L^1(\TOmega; L^2(0,T; H^1(\T)) ) \cap L^2(\TOmega; L^2(0,T;L^2(\T)))$. 
            \item According to Corollary \ref{cor:skorohod estimates} and Theorem \ref{ape:thm3}, $\lp \Tu_\eps^N \rp_{N\in \Lambda}$ is bounded in $L^2(\TOmega; L^2(0,T;H^4(\T)))$. By Alaoglu's theorem, every subsequence further admits a subsequence that converges weakly in $L^2(\TOmega; L^2(0,T;H^4(\T)))$. 
            \item According to Corollary \ref{cor:skorohod estimates} and Theorem \ref{ape:thm1}, we have that $\lp \Tu_\eps^N \rp_{N\in \Lambda}$ is bounded in $L^8(\TOmega; L^\infty(0,T;L^2(\T)))$. 
             Hence, it follows that $\Tu_\eps^N \stackrel{*}{\rightharpoonup} \Tu_\eps$ weakly in $L^8(\TOmega; L^\infty(0,T;L^2(\T)))$ as $N\to \infty$ along $\Lambda$.
        \end{enumerate}
    \end{proof}
Next, we will gather convergence results for the stochastic integrals beginning with that of the Wiener noise. The proofs of the following result is identical to the proof of Lemma 5.6 - Corollary 5.11 presented in \cite{CTT18euler} (see also {Lemma 4.14, 4.15} in \cite{NTT21}) in conjunction with Lemma 2.1 in \cite{DGT11}.
\begin{lemma}\label{cor:L^1 convergence of G}
	The noise terms on the right hand side of \eqref{eqn: regularized} satisfy the following convergences:  
\begin{enumerate} 
\item	    $P_N G(\Tu_\eps^N) \to G(\Tu_\eps)$ as $N \to \infty$ along $\Lambda$, $\TP$-a.s  in the space $L^2(0, T; HS(U_0, H^1(\T)))$.

\item The processes $ \lp\int_0^t P_N G(\Tu_\eps^N (s-))d\tilde W_N(s)\rp_{t \in [0,T]}$ converge in probability as $L^2(0,T; H^1(\T))$-valued random variables to $\lp \int_0^t G(\widetilde{u}_\eps(s-)) dW(s)\rp_{t \in [0,T]}$ as $N \to \infty$. 

\item The processes $ \lp\int_0^t P_N G(\Tu_\eps^N (s-))d\tilde W_N(s)\rp_{t \in [0,T]}$ converge in $L^1(\TOmega; L^1(0,T;H^1(\T)))$ 
to the process $\lp \int_0^t G(\widetilde{u}_\eps(s-)) dW(s)\rp_{t \in [0,T]}$ as $N \to \infty$.
\end{enumerate}
\end{lemma}
Similarly, we show the following convergence results for the jump integrals. The proof of the following Lemma is given in Appendix \ref{section:appendix}.
\begin{lemma}\label{cor:L^1 convergence of K}
	The noise terms on the right hand side of \eqref{eqn: regularized} satisfy the following convergences:  
\begin{enumerate}
\item $ P_N K(\Tu_\eps^N(s-), \xi) \to K(\Tu_\eps(s-), \xi)$ as $N \to \infty$ along $\Lambda$, $\TP$-a.s in the space $ L^2([0, T] \times E_0, dt \otimes d\nu; H^1(\T))$.

\item The processes $ \left(\int_{(0,t]} \int_{E_0} P_N K(\widetilde{u}_\eps^N(s-),\xi) d\widetilde{\widehat{\pi}}_N(s,\xi)\right)_{t \in [0,T]} $ converge in probability, as random variables with values in $ L^2(0, T; H^1(\T)) $, to the process $ \left(\int_{(0,t]} \int_{E_0} K(\widetilde{u}(s-),\xi) d\widetilde{\pi}(s,\xi)\right)_{t \in [0,T]} $.

\item The processes $ \left(\int_{(0,t]} \int_{E_0} P_N K(\widetilde{u}_\eps^N(s-),\xi) d\widetilde{\widehat{\pi}}_N(s,\xi)\right)_{t \in [0,T]} $ converge in $ L^1(\widetilde{\Omega}; L^1(0, T; H^1(\T))) $ to the process $ \left(\int_{(0,t]} \int_{E_0} K(\widetilde{u}(s-),\xi) d\widetilde{\widehat{\pi}}(s,\xi)\right)_{t \in [0,T]} $. 
\end{enumerate}
\end{lemma}

With the convergence results Theorem \ref{thm:skorohod}, Lemma \ref{thm:convergence} and Lemmas \ref{cor:L^1 convergence of G}-\ref{cor:L^1 convergence of K}, we are now in position to pass $N\to\infty$ in the approximation system \eqref{eqn: galerkin relabel}. This brings us to the following existence result for the regularized system \eqref{eqn: regularized}.
\begin{theorem} \label{thm:solution to galerkin}
The stochastic basis $\lp \widetilde{\Omega}, \widetilde{\F}, \widetilde{\P}, ({\TF^\eps_t})_{t \geq 0}, \widetilde{W}, \widetilde{\pi} \rp$  and the limiting random variable $\Tu_\eps$ obtained in Theorem \ref{thm:skorohod} is a global martingale solution to \eqref{eqn: regularized} satisfying $\Tu_\eps(0) = \Tu_{0}$, where $\Tu_{0}$ has law $\mu_{0}$.
\end{theorem}
    \begin{proof}
The aim of this proof is to show that the process $\Tu_{\eps} $ obtained after applying the Skorohod Representation theorem in Theorem \ref{thm:skorohod} which belongs to the space $ L^8(\TOmega; L^\infty(0,T; L^2(\T)))  \cap L^2(\TOmega; L^2(0,T; H^4(\T)))$ and takes values in $\Tu_\eps \in \D([0,T]; H^{-2}(\T))$, $\TP$-a.s., solves \eqref{eqn: regularized}. For that purpose we will pass $N\to\infty$ in \eqref{eqn: galerkin relabel} by using the convergence obtained in Lemma \ref{thm:convergence}. First, recall that Lemmas \ref{cor:L^1 convergence of G}- \ref{cor:L^1 convergence of K} provides us with the desired convergence of the stochastic integrals in the right side of \eqref{eqn: galerkin relabel}. Hence, in the rest of this proof we will briefly explain the passage of $N\to\infty$ in the left-hand side terms of \eqref{eqn: galerkin relabel}.
        
        Fix a vector $\phi \in H^3(\T)$ and a set $\Gamma \in \TF \otimes \mathcal{B}([0,T])$. Since, due to Theorem \ref{thm:skorohod}, we know that $\Tu_\eps^N \to \Tu_\eps$ in $L^1(\TOmega; L^1(0,T;H^1(\T)))$ as $N \to \infty$ along $\Lambda$, we readily find for the linear terms that
            \begin{align}
            \label {passage to limit: u} \TE \int_0^T \bigg|\chi_\Gamma (\phi, \Tu_\eps^N(s) - \Tu_\eps(s))_{L^2} ds \bigg| & \leq CT |\phi|_{H^3} \TE \int_0^T |\Tu_\eps^N(s) -\Tu_\eps(s)|_{L^2} ds \to 0, \\
            \label {passage to limit: gamma u}  \TE \int_0^T \bigg| \int_0^t \chi_\Gamma (\phi, \gamma\Tu_\eps^N(s) - \gamma\Tu_\eps(s))_{L^2} ds \bigg| & \leq CT |\phi|_{H^3} \TE \int_0^T |\Tu_\eps^N(s) -\Tu_\eps(s)|_{L^2} ds \to 0, \\
            \eps \TE \int_0^T \bigg|\int_0^t   \chi_\Gamma \lp \phi,  \partial_x^4 \Tu_\eps^N(s) - \partial_x^4 \Tu_\eps(s) \rp_{L^2} ds \bigg| & \leq C T \eps |\phi|_{H^3} \TE \int_0^T |\partial_x \Tu_\eps^N (s)- \partial_x \Tu_\eps(s)|_{L^2} ds  \to 0,
            \end{align}
        as $N \to \infty$ along $\Lambda$.
         It remains to check the convergence for the two truncated nonlinear terms.  For that purpose we define, 
        \begin{equation*}
            \begin{aligned}
                \alpha_N (t)& := \bigg| \int_0^t \chi_\Gamma \bigl(\phi, \theta_N (|\Tu_\eps^N(s)|_{H^1}) P_N \Tu^N_\eps \partial_x(s) \Tu^N_\eps(s) - \Tu_\eps \partial_x \Tu_\eps(s)\bigr)_{L^2} ds \bigg| \\
                & \leq \bigg| \int_0^t \theta_N (|\Tu_\eps^N(s)|_{H^1}) \chi_\Gamma \bigl(\phi, P_N(\Tu_\eps^N(s)\partial_x (s)\Tu_\eps^N(s) - \Tu_\eps \partial_x \Tu_\eps(s)) \bigr)_{L^2} ds   \bigg|  \\
                & + \bigg| \int_0^t \theta_N (|\Tu_\eps^N(s)|_{H^1}) \chi_\Gamma \bigl( (I-P_N) \phi, \Tu_\eps(s) \partial_x \Tu_\eps(s) \bigr)_{L^2}  ds  \bigg| \\
                & + \bigg| \int_0^t (\theta_N (|\Tu_\eps^N(s)|_{H^1}) - 1)\chi_\Gamma \bigl( \phi, \Tu_\eps(s) \partial_x \Tu_\eps(s) \bigr)_{L^2}ds   \bigg| \\
                & =: \alpha_{N,1}(t) + \alpha_{N,2}(t) + \alpha_{N,3}(t).
            \end{aligned}
        \end{equation*}
        We will show that $\alpha_N \to 0$ in $L^1(\TOmega \times [0,T])$ by showing $\alpha_{N,j} \to 0$ in $L^1(\TOmega \times [0,T])$, $j = 1,2,3$. We start with the term $\alpha_{N,1}$. By an application of integration by parts formula, along with Corollary \ref{cor:skorohod estimates} and the strong convergence \eqref{strong conv: galerkin}, we obtain
        \begin{equation} \label{convergence: alpha_{N,1}}
            \begin{aligned}
               & { \TE\int_0^T} \alpha_{N,1}(t) dt  = \TE \int_0^T \bigg| \int_0^t\theta_N (|\Tu_\eps^N(s)|_{H^1}) \chi_\Gamma \left(\phi, P_N(\Tu_\eps^N(s) \partial_x \Tu_\eps^N(s)- \Tu_\eps(s)\partial_x \Tu_\eps(s)) \right) ds   \bigg| \\
                & \leq |\phi|_{L^\infty} T \TE \int_0^T \bigg|(\Tu_\eps^N(s) - \Tu_\eps(s))\partial_x \Tu_\eps^N(s) + \Tu_\eps(s)(\partial_x \Tu^N_\eps(s) - \partial_x \Tu_\eps(s))\bigg|_{L^1} ds \\
                & \leq  C|\phi|_{L^\infty}\lp \TE\int_0^T|\Tu_\eps^N(s) - \Tu_\eps(s)|_{L^2}^2 ds \rp^{1/2} \bigg[ \lp \TE\sup_{t\in[0,T]}|\Tu_\eps^N(t)|_{H^1}^2 \rp^{1/2}+ \lp \TE \sup_{t\in[0,T]} |\Tu_\eps(t)|_{H^1}^2 \rp^{1/2} \bigg] \\
                & \leq  C|\phi|_{L^\infty} \lp \TE\int_0^T|\Tu_\eps^N(s) - \Tu_\eps(s)|_{L^2}^2 ds \rp^{1/2} \to 0,
            \end{aligned}
        \end{equation}
        For $\alpha_{N,2}$, thanks to the Sobolev embedding $H^1 (\T) \hookrightarrow L^\infty(\T)$, we obtain
        \begin{equation*}  
            \begin{aligned}
                \alpha_{N,2}(t) & = \bigg| \int_0^t \theta_N (|\Tu_\eps^N(s)|_{H^1}) \chi_\Gamma \left( (I-P_N)\phi, \partial_x \Tu_\eps(s) \right)  ds  \bigg| 
                 \leq |(I-P_N)\phi|_{L^2}  \int_0^T  |\Tu_\eps(s) \partial_x \Tu_\eps(s)|_{L^2} ds  \\
                 & \leq C|(I-P_N)\phi|_{L^2} \sup_{t \in [0,T]}|\Tu_\eps(t)|^2_{H^1}
            \end{aligned}
        \end{equation*}
        It follows that $\alpha_{N,2} \to 0$, $d\TP \otimes dt$-a.e. Moreover, by noting that, 
        \begin{equation*} 
            \begin{aligned}
                \alpha_{N,2} & \leq 2 C |\phi|_{L^2} \sup_{t \in [0,T]}|\Tu_\eps(t)|^2_{H^1} \in L^1(\TOmega \times [0,T]),
            \end{aligned}
        \end{equation*}
        we can conclude that $\alpha_{N,2} \to 0$ in $L^1(\TOmega \times [0,T])$ due to the Dominated Convergence Theorem. Finally, for $\alpha_{N,3}$, we write
        \begin{equation}   \label{alphaN3}
            \begin{aligned}
                \alpha_{N,3}(t) & = \bigg| \int_0^t (\theta_N (|\Tu_\eps^N(s)|_{H^1}) -1) \chi_\Gamma \left( \phi, \Tu_\eps(s) \partial_x \Tu_\eps(s) \right)_{L^2} ds  \bigg|  \\
                & \leq T |\phi|_{L^\infty} \sup_{t\in [0,T]} |\Tu_\eps(t)|_{H^1}^2 \int_0^T |\theta_N(|\Tu_\eps^N(s)|_{H^1}) - 1| ds  \\
            \end{aligned}
        \end{equation}
        { Next we will show that $\alpha_{N,3} \to 0$, $d\TP \otimes dt$-a.e:
                First, we note that due to Corollary \ref{cor:skorohod estimates} { and \eqref{as conv: galerkin}}, we have
                \begin{equation} \label{u_eps H2 a.s. bound} 
                  \sup_{t\in [0,T]}|\Tu_\eps(t)|_{H^2} < \infty, \quad \TP\text{-a.s.}
                \end{equation}
                Therefore, to show that $\alpha_{N,3}\to 0$, $\TP$-a.s., it is sufficient to argue that the integral on the right side of \eqref{alphaN3} converges to $0$, $\TP$-a.s.
 From \eqref{u_eps H2 a.s. bound}, we infer that, for almost surely $\omega \in \TOmega$, we might pick $N_0$ large, depending on $\omega$, such that $\theta_N(|\Tu_\eps(s)|_{H^1}) = 1$ for all $s \in [0,T]$ and $N \geq N_0$. Noting that $\theta_N$ is a Lipschitz function satisfying \eqref{truncation lipschitz} for every $N\geq N_0$ that,
                \begin{equation}\label{thetaconv} 
                	\begin{aligned}
                		\int_0^T \bigl| \theta_N(|\Tu_\eps^N(s)|_{H^1}) - 1 \bigr| ds = \int_0^T \bigl| \theta_N(|\Tu_\eps^N(s)|_{H^1}) - \theta_N(|\Tu_\eps|_{H^1}) \bigr|  ds \leq C\int_0^T \bigl| |\Tu_\eps^N(s)|_{H^1} - |\Tu_\eps|_{H^1} \bigr| ds.
                	\end{aligned}
        	        \end{equation}  
        	        Hence, thanks to the convergence \eqref{as conv: galerkin}, we obtain that  $\alpha_{N,3}(t) \to 0$, $d\TP \otimes dt$-a.e. 
        	    
        	}
        Moreover, as noted above in \eqref{alphaN3}, $\alpha_{N,3}$ is dominated by a function in $L^1(\TOmega \times [0,T]).$       Hence, again by applying the Dominated Convergence Theorem we conclude that $\alpha_{N,3} \to 0$ in $L^1(\TOmega \times [0,T])$. 
        Collecting the convergence results for $\alpha_{N,j}; j=1,2,3$,  we have proved that 
                \begin{equation*}
            \alpha_N \to 0 \text{ in } L^1(\TOmega \times [0,T]).
        \end{equation*}
        Next we consider the following truncated term,
        \begin{equation*}
            \begin{aligned}
                \beta_N(t)& := \bigg|\int_0^t (\theta_N (|\Tu_\eps^N(s)|_{H^1}) \partial_x^3 \Tu_\eps^N(s) -   \partial_x^3 \Tu_\eps(s), \phi)_{L^2}  ds \bigg| \\
                & \leq \bigg|\int_0^t \theta_N (|\Tu_\eps^N(s)|_{H^1}) ( (\partial_x^3 \Tu_\eps^N(s) -  \partial_x^3 \Tu_\eps(s)), \phi)_{L^2}ds \bigg| +  \bigg|\int_0^t (\theta_N (|\Tu_\eps^N(s)|_{H^1}) - 1) (\partial_x^3 \Tu_\eps(s), \phi)_{L^2} ds \bigg| \\
                & \leq |\phi|_{H^2} \int_0^t |\Tu_\eps^N(s) - \Tu(s)|_{H^1} ds +  \int_0^t \bigg|(\theta_N (|\Tu_\eps^N(s)|_{H^1}) - 1) (\partial_x^3 \Tu_\eps(s), \phi)_{L^2} \bigg| ds
                 =: \beta_{N,1}(t) + \beta_{N,2}(t)
            \end{aligned}
        \end{equation*}
        and prove that $\beta_N \to 0$ in $L^1(\TOmega \times [0,T])$.
        
        We first observe that $\beta_{N, 1} \to 0$ in $L^1(\TOmega \times [0,T])$ since we have
        \begin{equation*}
            \begin{aligned}
                \TE \int_0^T \beta_{N,1}(t)dt \leq T |\phi|_{H^2} \TE \int_0^T |\Tu_\eps^N (s)- \Tu_\eps(s)|_{H^1}ds \to 0.
            \end{aligned}
        \end{equation*}
     
        {
        To treat the remaining term $\beta_{N,2}$, we first show $\beta_{N,2}(t) \to 0$, $d\TP\otimes dt$-a.e. By integrating by parts, we have
        \begin{equation} \label{betaN2}
        	\beta_{N,2}(t) \leq |\partial_x \phi|_{L^2} \sup_{t \in [0,T]} |\partial_x^2 \Tu_\eps(t)|_{L^2} \int_0^T \bigl| \theta_N(|\Tu_\eps^N(s)|_{H^1}) - 1 \bigr| ds.
        \end{equation}
        Then due to \eqref{thetaconv}, we conclude that $\beta_{N,2}(t) \to 0$, $d\TP \otimes dt$-a.e.
        Moreover, recalling that $|\theta_N| \leq 1$, and applying Young's inequality to \eqref{betaN2}, we further obtain
        \begin{equation*}
        	\beta_{N,2} \leq 4T^2 |\partial_x \phi|^2_{L^2} \sup_{t \in [0,T]} |\partial_x^2 \Tu_\eps(t)|_{L^2}^2 + 1 \in L^1(\TOmega \times [0,T]).
        \end{equation*}
        Hence, another application of Dominated Convergence Theorem concludes that $\beta_{N,2} \to 0 \text{ in } L^1(\TOmega \times [0,T])$.
        } 
       This completes the proof for our claim that $\beta_N \to 0$ in $L^1(\TOmega \times [0,T])$.
       
     Finally we assemble all the convergence results above and conclude that,  
        \begin{equation*}
            \begin{aligned}
            \chi_\Gamma \Tu_\eps(t) & + \int_0^t \chi_\Gamma \bigg( \theta_N(|\Tu_\eps(s)|_{H^1}) \Tu_\eps(s) \partial_x \Tu_\eps(s) + \partial_x^3 \Tu_\eps(s) + \gamma \Tu_\eps(s) + \eps \partial_x^4 \Tu_\eps(s) \bigg) ds \\
            & = \chi_\Gamma \Tu_{0} + \int_0^t \chi_\Gamma G(\Tu_\eps(s-)) dW(s) +  \int_0^t \int_{E_0} \chi_\Gamma K(\Tu_\eps(s-), \xi) d\widehat{\widetilde{\pi}}(s,\xi)
            \end{aligned}
        \end{equation*}
        holds almost surely in $H^{-2}(\T)$ for every $t\in[0,T]$. Moreover, since $\Tu_\eps$ takes c\`{a}dl\`ag paths in $H^{-2}(\T)$, $\TP$-a.s., thanks to Lemma 5.13 of \cite{CTT18euler}, we can show that $\Tu_\eps$ solves the regularized equation \eqref{eqn: regularized} with initial condition $\Tu_{0}$ for all $t \in [0,T]$. 
        This finishes the proof of existence of a global martingale solution to the regularized equation \eqref{eqn: regularized}.
    \end{proof}
Passage of $\epsilon\to0$ is handled similarly and we skip the details.
\subsection{Global pathwise solution} \label{section: pathwise uniqueness}
In this section, we aim to prove the pathwise uniqueness for the equation \eqref{eqn: skdv}. With the pathwise uniqueness, we will be able to conclude the existence of a unique pathwise solution to the equation \eqref{eqn: skdv}, by an application of a result by Gy\"{o}ngy and Krylov \cite{GK96}, which is an infinite dimensional version of the result by Watanabe and Yamada \cite{YW71}. 
\begin{definition}[Pathwise uniqueness]\label{def:unique}
    We say that the \emph{pathwise uniqueness} holds for the stochastic KdV equation \eqref{eqn: skdv} if for every pair of global martingale solutions $u$ and $v$ to $\eqref{eqn: skdv}$ with respect to the same sotchastic basis $\lp \Omega, \mathcal{F}, (\mathcal{F}_t)_{t \geq 0}, \P, W, \pi \rp$, one has
    \begin{equation}
        \P\bigg[ \mathbf{1}_{\{u(0) = v(0)\}} (u(t) - v(t)) = 0 \ \forall t \in [0,T] \bigg] = 1.
    \end{equation}
\end{definition}

\begin{theorem} \label{thm: pathwise uniqueness}
    There exists a unique pathwise solution to the stochastic KdV equation \eqref{eqn: skdv} in the sense of Definition \ref{def:unique}. 
\end{theorem}
    \begin{proof}
        Let $u$ and $v$ be two global martingale solutions to $\eqref{eqn: skdv}$ with respect to the same sotchastic basis $\lp \Omega, \mathcal{F}, (\mathcal{F}_t)_{t \geq 0}, \P, W, \pi \rp$.
        Define $\Omega_0 := \{u_0 = v_0\}$ with $\P(\Omega_0) = 1$, and $w(t) := \mathbf{1}_{\Omega_0}(u(t) - v(t))$. We will show that 
        \[
        \E \sup_{t \in [0,T]} |w(t)|_{L^2}^2 = 0.   
        \]
        Define 
        \[\tau_n := \inf \bigg\{ t > 0: \int_0^t \mathbf{1}_{\Omega_0}|v(s)|^2_{H^2} ds \geq n \bigg\} \wedge T.
        \]
        Then $\tau_n$ is a stopping time with $\tau_n \nearrow T \ \P$-a.s. By Monotone Convergence Theorem, it suffices to show
        \[
        \E \sup_{t \in [0,\tau_n]} |w(t)|_{L^2}^2 = 0.   
        \]
        Observe that $w$ solves the following equation
        \begin{equation} \label{eqn: pathwise uniqueness}
            \begin{cases}
            dw(t) + \left(\partial_x^3 w(t) + \gamma w (t)+ u(t)\partial_x u(t) - v(t)\partial_x v(t)\right) dt \\
            \hspace{1cm} = (G(u(t-)) - G(v(t-))) dW(t) + \int_{E_0} (K(u(t-), \xi) - K(v(t-),\xi)) d\widehat{\pi}(\xi, t), \\
            w(0) = 0.
            \end{cases}
        \end{equation}
        Apply It\^{o}'s formula with $\mathcal{I}_0(w) := |w|_{L^2}^2$ to \eqref{eqn: pathwise uniqueness} up to some time $t \leq \tau_n$, we obtain:
        \begin{equation}  \label{uniqueness: ito formula}
            \begin{aligned}
                & |w(t)|_{L^2}^2 + 2\gamma \int_0^t |w(s)|^2_{L^2} ds + 2 \int_0^{t}  \left(w(s), u(s)\partial_x u(s)-  v(s)\partial_x v(s) \right)_{L^2} ds \\
                 & = 2\int_0^{t} (w(s), (G(u(s)) - G(v(s)))dW(s))_{L^2} ds + 2\int_0^{t} \int_{E_0} (w(s), K(u(s), \xi) - K(v(s),\xi))_{L^2} d\widehat{\pi}(\xi,s)  \\
                 & + \int_0^{t}  \|G(u(s)) - G(v(s))\|_{HS(U_0, L^2)}^2 ds  + \int_0^{t} \int_{E_0} |K(u(s),\xi) - K(v(s),\xi)|_{L^2}^2 d\pi(\xi,t) \\
                & =: B_1(t) + B_2(t) + B_3(t) + B_4(t).
            \end{aligned}
        \end{equation}
        We will take $\E\sup_{t\in[0,\tau_n]}$, and estimate $B_j; j = 1,2,3,4$ term by term. We treat the terms $B_1$ and $B_2$ in a similar way to \eqref{eqn: ape wiener term} and \eqref{eqn: ape CPRM term}. We apply BDG inequality, the Lipschitz condition \eqref{assumption: lipschitz}, and Young's inequality, to obtain
        \begin{equation} \label{uniqueness: B1, B2}
            \begin{aligned}
                \E\sup_{t \in [0,\tau_n]} (| B_1(t)| + |B_2(t)|) \leq \frac{1}{2} \E \sup_{t \in [0,\tau_n]}|w(t)|_{L^2}^2 + C \E \int_0^{\tau_n} |w(s)|_{L^2}^2 ds.
            \end{aligned}
        \end{equation}
        For $B_3$ and $B_4$ terms, using  It\^{o} Isometry and the Lipschitz condition \eqref{assumption: lipschitz} on $G$ and $K$, we obtain 
        \begin{equation} \label{uniqueness: B3, B4}
            \begin{aligned}
                \E\sup_{t \in [0,\tau_n]}\lp |B_3(t)| + |B_4(t)| \rp \leq C \E\int_0^{\tau_n} |w(s)|_{L^2}^2 ds.
            \end{aligned}
        \end{equation}
        Next, we turn to the nonlinear term on the left side of \eqref{uniqueness: ito formula}. Integrating by parts, and applying the Sobolev embedding $H^2(\T) \hookrightarrow L^\infty(\T)$ give us
        \begin{equation} \label{uniqueness: deterministic}
            \begin{split}
                2 \int_0^{\tau_n}  \left(w(s), u(s)\partial_x u(s) -  v(s)\partial_x v(s) \right)_{L^2} ds &= \int_0^{\tau_n} (w(s), w(s) \partial_x v(s) )_{L^2} \\
                &\leq C \int_0^{\tau_n} |w(s)|_{L^2}^2 |v(s)|_{H^2} ds.
            \end{split}
        \end{equation}
        Substituting \eqref{uniqueness: B1, B2}, \eqref{uniqueness: B3, B4} and \eqref{uniqueness: deterministic} back to \eqref{uniqueness: ito formula}, we have
        \begin{equation} \label{uniqueness: GW condition 1}
        	\E\sup_{t \in [0,\tau_n]}|w(t)|_{L^2}^2 \leq C \E \int_0^{\tau_n} \biggl( (|v(s)|^2_{H^2} + C) \sup_{r \in [0,s]}|w(r)|_{L^2}^2 \biggr) ds.
        \end{equation}
       Note that by definition of $\tau_n$, 
       \begin{equation}  \label{uniqueness: GW condition 2}
       	\int_0^t |v(s)|_{H^2}^2 ds \leq n, \quad \text{$\P$-a.s. for all $t \leq \tau_n$}. 
       \end{equation}
        Hence, 
        \begin{equation} 
        \begin{aligned}  \label{uniqueness: GW condition 3}
            \E \bigg[\int_0^{\tau_n} (|v(s)|^2_{H^2} + C) \sup_{r \in [0,s)}|w(r)|_{L^2}^2 ds \bigg] \leq (n + CT) \E\biggl[\sup_{t \in [0,T]}|w(t)|_{L^2}^2 \biggr] < \infty
        \end{aligned}
        \end{equation}
        With \eqref{uniqueness: GW condition 1}, \eqref{uniqueness: GW condition 2} and \eqref{uniqueness: GW condition 3}, we  apply the Stochastic Gronwall Lemma \cite[Lemma 5.3]{NaMo09} to conclude
        \[
            \E\sup_{t \in [0,\tau_n]}|w(t)|_{L^2}^2 = 0,
        \]
        completing the proof for the pathwise uniqueness.
    \end{proof}

    Pathwise uniqueness of solutions can then be translated to uniqueness in law as is proved in Theorems 2 and 11  in \cite{Ondrejt04}. Thus, we obtain the following corollary.
\begin{corollary}[Uniqueness in Law]\label{uniquelaw}
	For any $T>0$ and every pair of global martingale solutions $(u^{(i)}, \Omega^i, \F_i, (\F^i_t)_{t \geq 0}, \P^i, W^i, \pi^i)$ to \eqref{eqn: damped skdv}  with the initial distributions $\mathcal{L}_{\P^1}(u_0^{(1)}) = \mathcal{L}_{\P^2}(u_0^{(2)})$ we have
	\begin{equation}
		\mathcal{L}_{\P^1}(u^{(1)}) = 	\mathcal{L}_{\P^2}(u^{(2)}) \text{ on  { $L^2(0,T; H^1(\T)) \cap \mathcal{D}([0,T]; H^{-2}(\T)) \cap \mathcal{D}([0,T]; H^2_w(\T))$},}
	\end{equation}
	where $\mathcal{L}(\cdot)$ denotes the law.
\end{corollary} 
\section{Existence of invariant measure} \label{section: invariant measure}
In this section, we will consider the stochastic damped KdV equation \eqref{eqn: skdv} obtained by setting $\mathscr{K}:=0$ in \eqref{eqn: skdv}.  We restate the equation here,
\begin{align} \label{eqn: damped skdv}
	du(t) + (\partial_x^3 u(t) + u \partial_x u (t)+ \gamma u(t))dt = G(u(t-)) dW(t) +  \int_{E_0} K(u(t-), \xi) d\widehat{\pi}(t,\xi).
\end{align}

Recall for any $T>0$, that Theorem \ref{Thm: summary existence}, proved in Section \ref{sec:existence}, gives the existence of a unique pathwise solution $u_T$ to \eqref{eqn: skdv}, and thus to \eqref{eqn: damped skdv}, in the sense of Definition \ref{definition: pathwise}.
Note that, for any given stochastic basis $(\Omega, \mathcal{F}, (\mathcal{F}_t)_{t\geq 0}, \P, W, \pi)$ we can extend this solution to the entire time interval $[0,\infty)$ by defining
\[u(t) := u_T(t), \;\; \text{ if } 0\leq t\leq T.\]
We emphasize that this definition of $u$ is appropriate since $u_T$ is pathwise unique which means that, for any $0<T_{1}< T_{2} $, the corresponding solutions $u_{T_{1}}$ and $u_{T_{2}}$ on the time intervals $[0, T_1]$ and $[0, T_2]$ respectively, agree on $[T_1, T_2]$. 
For any appropriate initial value $u_0$, we denote this global-in-time unique pathwise solution of \eqref{eqn: damped skdv} by $u(t) := u(t; u_0)$ for $t\in[0,\infty)$.

To establish the existence of an invariant measure for the semigroup $\{P_t\}_{t \geq 0}$ defined in \eqref{Pt} associated with \eqref{eqn: damped skdv}, we will apply the result of Maslowski-Seidler \cite[Proposition 3.1]{MaSe99} which generalizes the classical Krylov-Bogoliubov Theorem to spaces endowed with weak topology.
    Before stating this result, we introduce some notation. Recall that $H^2_w(\T)$ denotes the space $H^2(\T)$ equipped with the weak topology. We will denote by $SC(H^2_w(\T))$ the collections of functions that are sequentially continuous with respect to the weak topology $H^2_w(\T)$. Furthermore, by using the subindex $b$ in the notation $SC_b(H^2_w(\T))$ we add the requirement of boundedness of the functions. 
    We recall the following inclusions:
    \[
    C_b(H^2_w(\T)) \subset SC_b(H^2_w(\T)) \subset C_b(H^2(\T)). 
    \]
    Now we state the result by Maslowski-Seidler which gives us sufficient conditions for the semigroup $\{P_t\}_{t\geq 0}$ for existence of an associated invariant measure.
    
    \begin{theorem}[Proposition 3.1 in \cite{MaSe99}]\label{thm:MS}
    	Let $V$ be a Hilbert space. Let $\{P_t\}_{t \geq 0}$ be a Markov transition semigroup on $V$, and $\{P_t^*\}_{t \geq 0}$ be its adjoint semigroup defined by
    	\begin{equation}
    		\int_V \phi(x) (P_t^*\mu)(dx) =\int_V P_t\phi(x)\mu(dx),
    		\qquad \phi\in C_b(V_w),\ \mu\in Pr(V).
    	\end{equation}
    	 Assume that the following two conditions are satisfied:
    	\begin{enumerate}[(i)]
    		\item The semigroup $\{P_t\}_{t \geq 0}$ is sequential weak Feller in $V$, i.e., for any $t\geq0$, $P_t$ maps $C_b(V_w)$ into $SC_b(V_w)$, where $V_w$ stands for the space $V$ equipped with the weak topology.
    		\item There is a Borel probability measure $\nu$ on $V$ and $T_0 \geq 0$ such that for any $\varepsilon > 0$ there exists $R = R(\varepsilon)$ that satisfies
    		\begin{equation} \label{MS:tightness general}
    			\sup_{T \geq T_0} \frac{1}{T} \int_0^T P_t^*\nu( \{ |x|_V > R  \}  ) dt < \varepsilon.
    		\end{equation}
    	\end{enumerate}
    	Then there exists an invariant measure for $\{P_t\}_{t \geq 0}$.
    \end{theorem}
Our goal now is to prove that $\{P_t\}_{t \geq 0}$ defined in \eqref{Pt} satisfies the two conditions stated in Theorem \ref{thm:MS} with $V=H^2(\T)$.
    We will first verify the sequential weak Feller property of $\{P_t\}_{t \geq 0}$ in $H^2(\T)$. 
\begin{lemma} \label{lemma: weak feller}
	        The transition semigroup  $\{P_t\}_{t \geq 0}$ defined in \eqref{Pt} is sequentially weak Feller in $H^2(\T)$. In other words, for any $\phi \in C_b(H^2_w(\T))$  we have that $P_t\phi \in SC_b(H^2_w(\T))$.
	        \begin{proof}
	        		Fix any $\phi \in C_b(H^2_w(\T))$. In order to prove that $P_t\phi \in SC_b(H^2_w(\T))$, we must show that {for any $\{u_{n0}\}_{n=0}^\infty \subset H^2(\T)$ and $u_0\in H^2(\T)$ satisfying} $ u_{n0} \rightharpoonup u_0$ weakly in $H^2(\T)$, we have that $P_t\phi(u_{n0}) \to P_t\phi(u_0)$ as $n \to \infty$. \\ 
		            For any $t\geq 0$ we define $u_n(t) := u(t; u_{n0})$ and $u(t) := u(t; u_0)$ where we recall that $u$ denotes the unique pathwise solution to \eqref{eqn: damped skdv} corresponding to the initial data $u_{n0}, u_0$, respectively. \\
		            We first observe that there exists some large constant, $\mathcal{M} > 0$ we have
		            \[
		               \sup_{n \geq 0} \big[|u_{n0}|_{H^2}^2 + |u_{n0}|_{L^2}^8 \big]< \mathcal{M}. 
		            \]
		             We also recall Theorems \ref{ape:thm1}, \ref{ape:thm3} and \ref{ape:thm_fractional} which, for any $T>0$, give us the following bounds for $u_n$:
		                \begin{align*}
			                    \E \sup_{t \in [0,T]} |u_n(t)|^8_{L^2} & \leq C, \qquad
			                    \E \sup_{t \in [0,T]} |u_n(t)|^2_{H^2}  \leq C, \\
			                    \E \|u_n\|^2_{W^{1/4, 2}(0,T; H^{-2})} & \leq C,
			                \end{align*}
		            where $C$ depends on $\mathcal{M}, T$ but is independent of $n$. Therefore, following the arguments laid out in Proposition \ref{tightness:L2H2weak}, we can show that $(u_n)_{n \geq 0}$ is tight in the space $\mathcal{Y}_u$:
		        \begin{equation*}
		        	\mathcal{Y}_u := L^2_{loc}(0,\infty;H^1(\T)) \cap \mathcal{D}_{loc}([0,\infty); H^{-2}(\T)) \cap  \mathcal{D}_{loc}([0,\infty); H^2_w(\T)).
		        \end{equation*}
Applying Jakubowski's version of the classical Skorohod Representation Theorem to the sequence $(u_n, W, \pi)_{n \geq 0} \subset \mathcal{Y}_u \times C([0,\infty), U) \times \mathcal{N}_{[0,\infty)\times E}^{\#}$, as done in Theorem \ref{thm:skorohod}, we obtain a probability space $(\widetilde{\Omega}, \widetilde{\F}, \widetilde{\P})$ and 
		            { $\mathcal{Y}_u$}-valued random variables $\lp \widetilde{u}, \widetilde{W}, \widetilde{\pi} \rp$ and $\lp  \widetilde{u}_{n_k}, \widetilde{W}_{n_k}, \widetilde{\pi}_{n_k} \rp_{k \in \N}$. We then appropriately define the filtrations $( \widetilde{\F}^{n_k}_t)_{t\geq 0}$ and $( \widetilde{\F}_t)_{t\geq 0}$ as done in \eqref{filtration}. Then we have the following hold true:
		            \begin{enumerate}[(i)]
		            	\item  $\lp \widetilde{u}_{n_k}, \widetilde{W}_{n_k}, \widetilde{\pi}_{n_k} \rp$ has the same law as $(u_{n_k}, W, \pi)$ for each $k \in \N$, and
		            	\begin{equation} \label{as conv: weak feller}
		            		\lim_{k \to \infty}\lp\widetilde{u}_{n_k}, \widetilde{W}_{n_k}, \widetilde{\pi}_{n_k} \rp = \lp \widetilde{u}, \widetilde{W}, \widetilde{\pi} \rp, \ \widetilde{\P}\text{-a.s.}
		            	\end{equation}
		            	\item For every $n_k$, $\widetilde{u}_{n_k}$ is a solution to the Problem \eqref{eqn: damped skdv} with respect to the stochastic basis $\lp \widetilde{\Omega}, \widetilde{\F}, \widetilde{\P}, (\TF^{n_k}_t)_{t \geq 0}, \widetilde{W}_{n_k}, \widetilde{\pi}_{n_k} \rp$ and $\widetilde{u}_{n_k}(0) =u_{n_k0} $, $\TP$-a.s.
		            	\item $\widetilde{u}$ is a solution to the Problem \eqref{eqn: damped skdv} with respect to the stochastic basis $\lp \widetilde{\Omega}, \widetilde{\F}, \widetilde{\P}, (\TF_t)_{t \geq 0}, \widetilde{W}, \widetilde{\pi} \rp$ and $\widetilde{u}(0) =u_{0} $, $\TP$-a.s.
		            \end{enumerate}
		            Note that, a consequence of \eqref{as conv: weak feller}, is that for any $t\geq0$ we have
		            \begin{equation*}
		            	\Tu_{n_k}(t) \rightharpoonup \Tu(t) \text{ weakly in $H^2(\T)$, $\TP$-a.s.}
		            \end{equation*}
		             Moreover, due to the equivalence of laws of $\Tu_{n_k}$ and $ u_{n_k}$ in the space {$\mathcal{Y}_u$}, {and the Dominated Convergence Theorem}, one has for any $t\geq0$ that
		            \begin{equation} \label{weak feller: conv for u_n_k}
		            	P_t\phi(u_{n_k0}) = \E\phi(u_{n_k}(t)) = \TE \phi(\Tu_{n_k}(t)) \to \TE \phi (\Tu(t)), \text{ as $k \to \infty$}
		            	\end{equation}
		           Thus, to finish the proof of this lemma we need to show that
		           \begin{align}\label{weak feller: equal for u0}
		           	\TE \phi (\Tu(t)) = \E\phi(u(t)) = P_t\phi(u_0).
		           \end{align}
		         
		           Since, $\tilde u$ and $u$ both solve \eqref{eqn: damped skdv} emanating from the same initial condition $u_0$, it follows directly from uniqueness in law obtained in Corollary \ref{uniquelaw}:
		            \begin{equation} \label{weak feller: unique in law}
		            \mathcal{L}_{\P}(u) = \mathcal{L}_{\TP}(\Tu) \text{ on the space $\mathcal{Y}_u$}.
		            \end{equation}
		            Hence, we have established \eqref{weak feller: equal for u0}.
		We then apply pathwise uniqueness and a sub-subsequence argument to conclude the convergence for the whole sequence
		            \begin{equation*}
		            	P_t\phi(u_{n0}) \to P_t\phi(u_0) , \text{ as $n \to \infty$},
		            \end{equation*}
		            which completes the proof of Lemma \ref{lemma: weak feller}.
		        \end{proof}
	    \end{lemma}
Our next goal is to prove that $\{P_t\}_{t \geq 0}$ defined in \eqref{Pt} satisfies the second assumption (ii) in Theorem \ref{thm:MS}. 
    \begin{lemma} \label{lemma: tightness condition for invaraint measure}
        Let $\{P_t\}_{t \geq 0}$ be the transition semigroup associated to the equation \eqref{eqn: damped skdv}. Then for any $\varepsilon > 0$, there exists some $R = R(\varepsilon) > 0$ such that
        \begin{equation} \label{invariant: tight sequence requirement}
            \sup_{T \geq 1} \frac{1}{T} \int_0^T \P( |u(t;0)|_{H^2} > R )dt \leq \sup_{T \geq 1}   \frac{\frac{1}{T} \int_0^T \E|u(t;0)|_{H^2}^2 dt}{R^2} < \varepsilon.
        \end{equation}
    \end{lemma} 
To prove Lemma \ref{lemma: tightness condition for invaraint measure}, we mus show that $\displaystyle\int_0^t \E|u(s; u_0)|^2_{H^2} ds$ is $O(t)$ as $t\to\infty$. This will be done in the following two results i.e. Theorems \ref{thm: time-uniform L^2} and \ref{thm: time-uniform H^2} where we establish the desired linear-in-time estimates first in $L^2(\T)$ and then in $H^2(\T)$ respectively. The proof of Lemma \ref{lemma: tightness condition for invaraint measure} is delayed.  
    \begin{theorem} \label{thm: time-uniform L^2}
        Consider any $u_0 \in H^2(\T)$ and let $u(t) := u(t;u_0)$ be the unique pathwise solution to \eqref{eqn: damped skdv} with initial condition $u_0$. Then for any $2 \leq p \leq 8$, we have
            \begin{align}  \label{invariant: L^2}
            	\E |u(t)|_{L^2}^p + \eta{(p)} \int_0^t \E|u(s)|_{L^2}^{p} ds & \leq |u_0|_{L^2}^p +  \tilde{C}_p \kappa_2 \cdot t, 
            \end{align}
        where 
         \begin{equation} \label{invariant measure: gamma requirement 1}
        	{ \eta(p)} := \gamma p - \big( p(p-1)\kappa_1 + 3 \tilde{C}_p \kappa_2 \big),\quad \tilde{C}_p := (p^2-p)2^{p-3} > 0.
        \end{equation}
Consequently, if $\gamma$ is sufficiently large so that $\eta$ defined in \eqref{invariant measure: gamma requirement 1} is positive, then we have
        \begin{align}  \label{invariant: L^2 linear}
        	\E |u(t)|_{L^2}^p + \eta{(p)} \int_0^t \E|u(s)|_{L^2}^{p} ds & \leq C(t+1),
        \end{align}
        where $C = C(\kappa_1, \kappa_2, \gamma, \tilde{C}_p, u_0) > 0$.
            \end{theorem}
        \begin{proof}
            Apply It\^{o}'s formula to equation \eqref{eqn: damped skdv} with respect to the functional $\tilde{\mathcal{I}}_0(v) = |v|_{L^2}^p$ as in the proof of Theorem \ref{ape:thm1}, and take expectation, we obtain 
            \begin{equation} \label{invariant: ito L2}
                \begin{aligned}
                    & \E |u(t)|_{L^2}^p + \gamma p \int_0^t \E|u(s)|_{L^2}^{p} ds =  |u(0)|_{L^2}^p  + \frac{p}{2} \E \int_0^t |u(s)|_{L^2}^{p-2} \|G(u(s))\|_{HS(U_0, L^2)}^2 ds \\
                    & + \frac{p(p-2)}{2}\E \int_0^t |u(s)|_{L^2}^{p-4}\|G^*(u(s)) u(s)\|^2_{U_0}ds \\
                    & + \E \int_0^t \int_{E_0} \lp |u(s) + K(u(s), \xi)|_{L^2}^p - |u(s)|_{L^2}^p  - p |u(s)|_{L^2}^{p-2}(u(s), K(u(s), \xi))  \rp d\nu(\xi) ds \\
                    & =: |u(0)|_{L^2}^p + M_1 + M_2 + M_3,
                \end{aligned}
            \end{equation}
            where $G^*(u)$ is the adjoint operator of $G(u)$. We note that two mean-zero martingale terms vanish when we take the expectation. We will give bounds for each $M_j; j = 1,2,3$. For the term $M_1$ we observe, due to the growth condition assumption \eqref{assumption: growth of G} on $G$ and Holder inequality, that
            \begin{equation} \label{invariant: M1}
                \begin{aligned}
                \frac{p}{2} \E \int_0^t |u(s)|_{L^2}^{p-2} \|G(u(s))\|_{HS(U_0, L^2)}^2 ds 
                 \leq \frac{p}{2}\kappa_1 \E \int_0^t |u(s)|_{L^2}^{p-2}(|u(s)|_{L^2}^2 + 1) ds 
                \leq p \kappa_1 \int_0^t \E|u(s)|_{L^2}^p ds
                \end{aligned}
            \end{equation}
            Similarly, we obtain for $M_2$ that
            \begin{equation}
                \begin{aligned}\label{invariant: M2}
                M_2 &  \leq p(p-2)\kappa_1 \int_0^t \E|u(s)|_{L^2}^p ds
                \end{aligned}
            \end{equation}
            To estimate $M_3$, we utilize the inequality \eqref{ineq}, the growth condition \eqref{assumption: further growth of K}, and Holder's inequality as follows
            \begin{equation} \label{invariant: M3}
                \begin{aligned} 
                M_3 & = \E \int_0^t \int_{E_0} \biggl( |u(s) + K(u(s), \xi)|_{L^2}^p - |u(s)|_{L^2}^p  - (u(s), K(u(s), \xi))  \biggr) d\nu(\xi) ds \\
                & \leq \tilde{C}_p \E \int_0^t \int_{E_0}  \biggl( |u(s)|_{L^2}^{p-2}|K(u(s), \xi)|_{L^2}^2 + |K(u(s), \xi)|_{L^p}^p  \biggl) d\nu ds \\
                & \leq \tilde{C}_p \kappa_2 \E \int_0^t  \biggl( |u(s)|_{L^2}^{p-2}(1+|u(s)|_{L^2}^2) + (1+|u(s)|_{L^2}^p)  \biggl) ds \\
                & \leq \tilde{C}_p \kappa_2 \int_0^t \bigg(3\E|u(s)|_{L^2}^p + 1\bigg) ds = 3 \tilde{C}_p \kappa_2\int_0^t  \E |u(s)|_{L^2}^p ds + \tilde{C}_p \kappa_2 \cdot t
                \end{aligned}
            \end{equation}
            Therefore, substuting \eqref{invariant: M1}, \eqref{invariant: M2} and \eqref{invariant: M3} back to \eqref{invariant: ito L2} yields
            \begin{equation}
                \begin{aligned}
                \E |u(t)|_{L^2}^p + \gamma p \int_0^t \E|u(s)|_{L^2}^{p} ds \leq \E|u_0|_{L^2}^p + \lp p(p-1)\kappa_1 + 3 \tilde{C}_p \kappa_2 \rp \int_0^t  \E|u(s)|_{L^2}^p ds +  \tilde{C}_p \kappa_2 \cdot t.
                \end{aligned} 
            \end{equation}
            Rearranging the terms gives \eqref{invariant: L^2}. This finishes the proof of Theorem \ref{thm: time-uniform L^2}.
        \end{proof}
        We will upgrade the bounds in Theorem \ref{thm: time-uniform L^2} and obtain linear-in-time bounds in more regular spaces in the following Theorem.
    \begin{theorem} \label{thm: time-uniform H^2}
       Let $u_0 \in H^2(\T)$ be a deterministic function and $u(t) := u(t;u_0)$ be the unique pathwise solution to \eqref{eqn: damped skdv} with initial condition $u_0$. Suppose $\gamma $ is chosen large to ensure that $\eta(8)>0$ where $\eta$ defined in \eqref{invariant measure: gamma requirement 1}. Then we have 
        \begin{equation} \label{invariant: H2}
        	\E|\partial_x^2 u(t)|_{L^2}^2 + \eta'' \int_0^t \E|\partial_x^2 u(s)|^2_{L^2} ds \leq C (t+1),
        \end{equation}
        where 
        	$\eta''$ is a positive constant independent of $p$, given by
        \begin{equation}
        	\eta'' := \frac{197}{50}\gamma - \kappa_1 - \kappa_2 > 0,
        \end{equation}
        and $C = C(\kappa_1, \kappa_2, \gamma, u_0) > 0$.
           \end{theorem}
        \begin{proof}
            We begin with recalling the definition of the functional $\mathcal{I}_2(v)$,
           	\begin{equation*}
           		\mathcal{I}_2(v) = \int_\T \biggl(  (\partial_x^2 v(x))^2 ds - \frac{5}{3} v(x)(\partial_x v(x))^2 + \frac{5}{36}(v(x))^4 \biggr) dx.
           	\end{equation*}
            Applying It\^{o}'s formula to \eqref{eqn: damped skdv} with the functional  $\mathcal{I}_2(u)$, and taking expectation both sides, we obtain
            \begin{equation} \label{invariant: ito H2}
                \begin{aligned}
                   & \E|\partial_x^2 u(t)|_{L^2}^2 + \frac{5}{36}  \E|u(t)|_{L^4}^4  + 2\gamma \E \int_0^t |\partial_x^2 u(s)|_{L^2}^2 ds + \frac{5}{9} \E\int_0^t |u(s)|_{L^4}^4 ds \\
                  & = \E  \mathcal{I}_2(u_0) +  \frac{5}{3} \E \int_\T u(t,x) (\partial_x u(t,x))^2 dx + 5\gamma \E \int_0^t \int_\T u(s,x)  (\partial_x u(s,x))^2 dx ds \\
                   & +\frac{1}{2} \E \int_0^t \text{Tr}(\mathcal{I}_2''(u(s)) G(u(s)) G^*(u(s))) ds \\
                    & + \E \int_0^t\int_{E_0} \biggl( \mathcal{I}_2(u(s)+ K(u(s),\xi)) - \mathcal{I}_2(u(s)) - (\mathcal{I}_2'(u(s)), K(u(s), \xi))_{L^2} \biggr) d\pi(\xi, s)\\
                    & =: \E \mathcal{I}_2(u_0) + D_1 + D_2 + D_3 + D_4
                \end{aligned}
            \end{equation}
            We will give estimates to $D_j, j = 1,2,3, 4$. We first deal with $D_1$ term. Using Sobolev embedding $H^{1/4}(\T) \hookrightarrow L^4(\T)$, Sobolev interpolation $H^{5/4}(\T) = [H^2(\T), L^2(\T)]_{5/8}$, Young's inequality and Holder's inequality, we get
            \begin{equation*}
                \begin{aligned}
                    D_1 & = \E \int_\T u(t,x) (\partial_x u(t,x))^2 dx \leq \E \biggl[ |u(t)|_{L^2}|\partial_x u(t)|_{L^4}^2 \biggr] \leq \E \biggl[ |u(t)|_{L^2}|u(t)|_{H^{5/4}}^2 \biggr]\\
                    & \leq C\E \biggl[ |u(t)|_{L^2} \bigl( |u(t)|_{L^2}^2 + |\partial_x^2 u(t)|_{L^2}^{1/4} |u(t)|_{L^2}^{7/4} \bigr) \biggr] \leq \frac{1}{2}\E|\partial_x^2 u(t)|_{L^2}^2 + C\E|u(t)|_{L^2}^8.
                \end{aligned}
            \end{equation*}
            Similarly, by adjusting the constant when invoking Young's inequality, we obtain
            \begin{equation*}
                \begin{aligned}
                    D_2 = \E \int_0^t \int_\T u(s,x) (\partial_x u(s,x))^2 dx ds \leq \frac{\gamma}{100}\int_0^t \E|\partial_x^2 u(s)|_{L^2}^2 ds + C\int_0^t \E|u(s)|_{L^2}^8  ds
                \end{aligned}
            \end{equation*}
            Combining above estimates for $D_1$ and $D_2$, thanks to Theorem \ref{invariant: L^2}, we have
            \begin{equation} \label{d1+d2}
                \begin{aligned}
               D_1 + D_2 & \leq \frac{1}{2}\E|\partial_x^2 u(t)|_{L^2}^2 + \frac{\gamma}{100}\int_0^t \E|\partial_x^2 u(s)|_{L^2}^2 ds + C\lp \E|u(t)|_{L^2}^8 + \int_0^t \E|u(s)|_{L^2}^8 ds\rp \\
               & \leq \frac{1}{2}\E|\partial_x^2 u(t)|_{L^2}^2 + \frac{\gamma}{100}\int_0^t \E|\partial_x^2 u(s)|_{L^2}^2 ds + C (t+1) 
                \end{aligned}
            \end{equation}
            For $D_3$ and $D_4$, one can follow the arguments in estimating \eqref{r3} and \eqref{R4} in the proof of Theorem \ref{ape:thm3},  and apply Theorem \ref{invariant: L^2}, to conclude that 
            \begin{equation} \label{d3-d4}
                \begin{aligned}
                    D_3 & \leq{ \lp \frac{\kappa_1}{2} + \frac{\gamma}{100}\rp}\int_0^t \E |\partial_x^2 u(s)|_{L^2}^2 ds + C\int_0^t( \E|u(s)|_{L^2}^8 + 1) ds \leq  \lp \frac{\kappa_1}{2} + \frac{\gamma}{100}\rp\int_0^t \E |\partial_x^2 u(s)|_{L^2}^2 ds + C(t+1)\\
                    D_4 & \leq{  \lp \frac{\kappa_2}{2} + \frac{\gamma}{100}\rp}\int_0^t \E |\partial_x^2 u(s)|_{L^2}^2 ds + C\int_0^t( \E|u(s)|_{L^2}^8 + 1) ds \leq \lp \frac{\kappa_2}{2} + \frac{\gamma}{100}\rp\int_0^t \E |\partial_x^2 u(s)|_{L^2}^2 ds + C(t+1)
                \end{aligned}
            \end{equation}
            Thus, plugging \eqref{d1+d2} and \eqref{d3-d4} back to \eqref{invariant: ito H2}, we arrive at
            \begin{equation}
                \begin{aligned}\label{lineartH^2}
                    \frac{1}{2} \E|\partial_x^2 u(t)|_{L^2}^2 + \lp \frac{197}{100}\gamma - \frac{\kappa_1}{2} - \frac{\kappa_2}{2} \rp\int_0^t \E|\partial_x^2 u(s)|^2_{L^2} ds \leq C(t+1).
                \end{aligned}
            \end{equation}
            This completes our proof for \eqref{invariant: H2}.
        \end{proof}
 
As explained earlier, thanks to the bounds obtained in Theorem \ref{thm: time-uniform H^2}, the proof of Lemma \ref{lemma: tightness condition for invaraint measure} is straightforward.
  \begin{proof}[Proof of Lemma \ref{lemma: tightness condition for invaraint measure}] 
    Now let $\gamma$ be large enough so that $\eta$ defined in \eqref{invariant measure: gamma requirement 1} is positive for $p = 8$. Then by Theorem \ref{thm: time-uniform L^2}, Theorem \ref{thm: time-uniform H^2} and the norm equivalence \eqref{Hs norm equivalent}, we have that
    \begin{equation} \label{invariant: H2 linear}
        \int_0^t \E|u(s; 0)|_{H^2}^2 ds  \leq C(t+1),
    \end{equation}
    where $C = C(\kappa_1, \kappa_2, \gamma)$. For each $\veps > 0$, we could choose $R = R(\veps)$ sufficiently large such that
    \begin{equation} \label{invariant: R large}
        \frac{2C}{R^2} < \veps.
    \end{equation}
    It follows from Chebyshev's inequality, \eqref{invariant: H2 linear} and \eqref{invariant: R large} that
    \begin{equation}
        \sup_{T \geq 1}\frac{1}{T}\int_0^T \P( |u(t;0)|_{H^2} > R ) \leq \sup_{T \geq 1} \frac{\frac{1}{T} \int_0^T \E|u(t;0)|_{H^2}^2}{R^2} < \frac{2C}{R^2} < \veps,
    \end{equation}
    which completes the proof for Lemma \ref{lemma: tightness condition for invaraint measure}. 
        \end{proof}
We can now verify that the two conditions stated in Theorem \ref{thm:MS} are statisfied due to the results obtained in Lemma \ref{lemma: weak feller} and Lemma \ref{lemma: tightness condition for invaraint measure}. Thus we conclude the proof of our second main result Theorem \ref{Thm: summary invariant} i.e. that of the existence of an invariant measure for the transition semigroup $\{P_t\}_{t \geq 0}$ for the dynamics of the stochastic damped KdV equations \eqref{eqn: damped skdv}.

\appendix

\section{Proof of Convergence Lemma \ref{cor:L^1 convergence of K}}\label{section:appendix}
In this Appendix, we provide a proof of Lemma \ref{cor:L^1 convergence of K} in an abstract setting. We begin the section by recalling the framework for stochastic integration with respect to compensated Poisson random measures $\widehat{\pi}$ (see, e.g., \cite{CTT18compare, ikeda2014stochastic}).

We assume that the filtered probability space $(\Omega, \mathcal{F}, (\mathcal{F}_t)_{t \in [0,T]}, \P)$ satisfies the usual conditions. We denote by $\mathcal{P}_{[0,T]}$ the predictable $\sigma$-algebra on the product space $\Omega \times [0,T]$. \\
{Now we recall our earlier notation and assumptions: Let $U$ be a real, separable Hilbert space and let $E := U \setminus \{0\}$. We can endow a metric $d$ on $E$ so that it is complete and separable as described in  \cite[Example A.8]{CTT18euler}. Under this particular metric, a subset $B \subset E$ is bounded if and only if $B$ is separated from $0$. We denote by $\mathcal{E}$ the Borel $\sigma$-algebra of ${E}$. We assume that for some measure $\nu$ (e.g. the L\'evy measure), the space $(E,\mathcal{E},\nu)$ is measurable. We do not assume $\nu(E)$ to be finite, but assume that $\nu$ is finite on bounded subsets of $E$. By partitioning $E$ as follows,
\begin{equation} \label{appendix:sigma finite}
	E = \bigcup_{m\in \N} \biggl\{ x \in E: m^{-1} < |x|_U < m  \biggr\} =: \bigcup_{m \in \N} A_m,
\end{equation}
and noting that $\nu(A_m) < \infty$ for each $m \in \N$, we see that $(E, \mathcal{E}, \nu)$ is $\sigma$-finite. 
We further define $$E_0 := U \setminus \{ x \in U: 0 < |x|_U <1 \},$$ so that we have $\nu(E \setminus E_0) < \infty$. } \\
Let $\Xi$ be a stationary Poisson point process on $E$ with a countable domain $\mathfrak{D}(\Xi) \subset (0, \infty)$ and intensity measure given by $dt \otimes d\nu$. The process $\Xi$ naturally induces a Poisson random measure. In this Section, we will  consider Poisson random measures on $[0,T]\times E$ induced by stationary point processes with intensity $dt \otimes d\nu$.

Next we discuss the space for the integrands for stochastic integrals with respect to Poisson random measures induced by such $\Xi$. Let $V$ be a real, separable Hilbert space. For any $q \in [1, \infty]$, we define the space of predictable, $L^q$-integrable processes as
$$\mathbf{F}_{\nu,T}^q(V) := L^q(\Omega \times [0,T] \times E, \mathcal{P}_{[0,T]} \otimes \mathcal{E}, d\P \otimes dt \otimes d\nu; V).$$
For recall that for any $g\in \mathbf{F}_{\nu,T}^1(V)$, we have
$$	I^{{\pi}}_t(g):=\int_{(0,t]} \int_E g(s, \xi) \, d\pi(s, \xi)=\sum_{s \in (0,t] \cap \mathfrak{D}(\Xi)} g(s, \Xi(s)), $$
whereas for the following stochastic integral with respect to the compensated measure defined as $\widehat{\pi}:=\pi-\nu$ is a $V$-valued square-integrable martingale for integrands $g\in \mathbf{F}_{\nu,T}^2(V)$:
\begin{equation}
	I^{\widehat{\pi}}_t(g) := \int_0^t \int_{E_0} g(s,\xi) d\widehat{\pi}(s, \xi).
\end{equation}

\begin{theorem} \label{lemma:convergence appendix}
	For every $N\in\mathbb{N}$, let $K_N, K \in \mathbf{F}^2_{\nu, T}(V)$, and let $\pi_N, \pi$ be Poisson random measures on $[0,T] \times E$ each with intensity $dt\otimes d\nu$. Assume that the following convergences hold $\P$-a.s.,
	\begin{align*}
		K_N & \to K \text{ in $L^2([0,T] \times E_0; dt \otimes d\nu; V)$}, \\
		\pi_N & \to \pi \text{ in $\mathcal{N}_{[0,T] \times E}^{\#}$}.
	\end{align*} 
	Then, up to a subsequence, the sequence of $ L^2(0, T; V)$-valued processes $ \left( I_t^{\widehat{\pi}_N}(K_N) \right)_{t \in [0,T]}$ converges almost surely to the $ L^2(0, T; V)$-valued process $ \left(  I_t^{\widehat{\pi}}(K) \right)_{t \in [0,T]} $.
\end{theorem} 
The proof of this theorem relies on the following two lemmas.
\begin{lemma} \label{lemma:appendix 1}
For each $N\in\mathbb{N}$, consider the sequence of $g_N\in\mathbf{F}^2_{\nu, T}(V)$ satisfying
	\begin{equation}\label{gnto0}
		\lim_{N \to \infty }g_N= 0 \text{ in $L^2([0,T] \times E; dt\otimes d\nu; V)$, \qquad$\P$-a.s.}
	\end{equation}
	Then, the following convergence is true:
	 $$\left( I_t^{\widehat{\pi}_N}(g_N)  \right)_{t \in [0,T]} \to 0\quad\text{in probability in $ L^2(0, T; V)$}.$$ 
	The same conclusion is true if $\widehat{\pi}_N$ is replaced by $\widehat{\pi}$.
	\begin{proof}
		The proof of this lemma follows the techniques of \cite[Corollary 5.10]{CTT18euler} (See also \cite[Lemma 4.15]{NTT21}).  Recall that $\pi_N$ has intensity $dt \otimes d\nu$ for each $N \in \N$. Then for any $\veps, \delta > 0$, we write
		\begin{equation*} \label{eqn:appendix g conv}
			\begin{aligned}
				\P\biggl[  \int_0^T |I_t^{\widehat{\pi}_N}(g_N)|_{V}^2 dt > T\veps^2 \biggr] & \leq \P \biggl[ \sup_{t\in[0,T]}  |I_t^{\widehat{\pi}_N}(g_N)|_{V} > \veps \biggr] =: \P(\mathcal{M}_N) \\
				& = \P( \mathcal{M}_N \cap \mathcal{J}_N) + \P( \mathcal{M}_N \cap \mathcal{J}_N^c) \leq \P(\mathcal{J}_N )+ \P( \mathcal{M}_N \cap \mathcal{J}_N^c), 
			\end{aligned}
		\end{equation*}
		where 
		\begin{equation*}
			\begin{aligned}
				\mathcal{J}_N := \biggl\{  \int_0^T \int_{E_0} |g_N(s,\xi)|_V^2 d\nu(\xi) ds  > \delta \veps^2 \biggr\}.
			\end{aligned}
		\end{equation*}
We first handle the second term $ \P( \mathcal{M}_N \cap \mathcal{J}_N^c) $. Due to the maximal inequality for square-integrable martingales given in \cite[Theorem 3.8]{da2014stochastic} (see also \cite[Proposition 4.16]{da2014stochastic}), we know that
		\begin{equation*}
			\P ( \mathcal{M}_N \cap \mathcal{J}_N^c ) \leq \delta.
		\end{equation*}
		Next, we treat the other term $\P(\mathcal{J}_N)$. Thanks to the almost sure convergence (and thus in probability) assumption \eqref{gnto0} we can choose an $N\in\mathbb{N}$ large enough to ensure that
		\begin{equation*}
			\P(\mathcal{J}_N) \leq \delta.
		\end{equation*} 
		Collecting the bounds above for $\P(\mathcal{J}_N)$ and $\P ( \mathcal{M}_N \cap \mathcal{J}_N^c)$, and noting that $\delta > 0$ is arbitrary, we complete the proof that $\left( I_t^{\widehat{\pi}_N}(g_N)  \right)_{t \in [0,T]} \to 0$  in probability in $L^2(0, T; V)$. The same argument still works when we replace $\widehat{\pi}_N$ by $\widehat{\pi}$.
	\end{proof}
\end{lemma}
The next result concerns with the approximation of $K\in\mathbf{F}^2_{\nu, T}(V)$ with bounded, continuous functions with bounded support in $[0,T)\times E$ which are used to define the weak$-\#$ convergence of $\pi_N \to \pi$ in $\mathcal{N}_{[0,T] \times E}^\#$, $\P$-a.s. 
	Before stating our next result we recall that  (see e.g. \cite{ikeda2014stochastic, CTT18compare}):
	\begin{enumerate}
		\item $\mathbf{F}_{\nu,T}^1(V) \cap \mathbf{F}_{\nu,T}^2(V)$ is dense in $\mathbf{F}_{\nu,T}^2(V)$.
		\item The stochastic integration with respect to the random measure $\pi$ is only defined for integrands in $\mathbf{F}_{\nu,T}^1(V)$, while stochastic integration with respect to the compensated Poisson random measure $\widehat{\pi}$ is only defined for integrands in $\mathbf{F}_{\nu,T}^2(V)$. Hence, the formula 
		\begin{equation*}  \label{appendix:equality hat pi = pi - dnu ds}
			\begin{aligned} 
				\int_{(0,t]} \int_E f(s, \xi) \, d\widehat{\pi}(s, \xi) &= \int_{(0,t]} \int_E f(s, \xi) \, d\pi(s, \xi) - \int_0^t \int_E f(s, \xi) \, d\nu(\xi) \, ds \\
				&= \sum_{s \in (0,t] \cap \mathfrak{D}(\Xi)} f(s, \Xi(s)) - \int_0^t \int_E f(s, \xi) \, d\nu(\xi) \, ds 
			\end{aligned}
		\end{equation*}
		only holds for integrands $f \in \mathbf{F}_{\nu,T}^1(V) \cap \mathbf{F}_{\nu,T}^2(V)$. 
	\end{enumerate}
Thus, in the following proof we ensure that the approximation integrands are constructed in appropriate spaces.

\begin{lemma} \label{lemma:appendix 2}
	Let $K \in \mathbf{F}_{\nu,T}^2(V)$. Then for any $\delta > 0$, there exists a process $F: \Omega \times [0,T] \times E \to V$ belonging to $ \mathbf{F}_{\nu,T}^1(V) \cap \mathbf{F}_{\nu,T}^2(V)$ such that
	\begin{enumerate}[(i)]
		\item {
		 $F$ takes the following form:
			\begin{equation*}
				F(\omega, t, \xi) := \sum_{j=1}^n F_j(\omega, t, \xi) = \sum_{j=1}^n \mathbbm{1}_{D_j}(\omega) h_{s_j, r_j}(t)f_{B_j}(\xi)  v_j,
			\end{equation*}
	for some $D_j \in \mathcal{F}_{s_j}, B_j \in \mathcal{E}, v_j \in V$.	Here for each $j = 1, \cdots, n$
			 \begin{equation} \label{h, f requirement}
			 	\begin{aligned}
			 		& \text{$F_j\in \mathbf{F}_{\nu,T}^1(V) \cap \mathbf{F}_{\nu,T}^2(V)$,}\\
			 		& \text{$h_{s_j, r_j}: [0,T] \to [0,1]$  is continuous and is supported on $[s_j, T]$ for $s_j \geq 0$,} \\
			 		& \text{$f_{B_j}: E \to [0,1]$  is continuous and has bounded support in $E$.} 
			 	\end{aligned}
			 \end{equation}
			Consequently, for almost any $\omega \in \Omega$, $F(\omega, \cdot, \cdot)$ is continuous, bounded and has bounded support in $[0,T] \times E$.
			 }
		\item The following bound holds true: 
		\begin{equation}\label{eqn:K - F approximation}
			\E \int_0^T \int_{E_0} |K(t,\xi) - F(t,\xi)|_V^2 d\nu(\xi) dt < \delta.
		\end{equation}
	\end{enumerate}
\end{lemma}

\begin{proof}
	\textbf{Step 1}: 
		Let $Leb_T$ denote the Lebesgue measure on $[0,T]$. Since $\P\times {Leb}_T\times \nu$ is $\sigma$-finite measure on $\Omega\times [0,T]\times E$, the set of simple functions is dense in $\mathbf{F}_{\nu,T}^2(V)$. Thus, in this step
	we will first assume $K$ to be a simple function and later in Step 5, we will establish the desired result for an arbitrary $K$. 
	In that spirit, assume	that $K$ takes the following form:
	\begin{equation} \label{eqn:simple K}
		K(\omega, t, \xi) := \mathbbm{1}_D(\omega) \mathbbm{1}_{[s, r)}(t) \mathbbm{1}_B(\xi) v, \quad \text{$D \in \mathcal{F}_s$, $B \in \mathcal{E}$, and $v \in V$.}
	\end{equation}
Note for unbounded sets $B$ that we can write $B = \bigcup_{m\in \N} (B \cap \tilde{A}_m)$ as a disjoint union, where $\tilde{A}_m := A_{m+1} \setminus A_m$ for $A_m$ defined in the partition \eqref{appendix:sigma finite}. Then we can decompose the indicator $\mathbbm{1}_B$ as follows:
	 \begin{equation}\label{B}
	 	\mathbbm{1}_B = \sum_{m \in \N} \mathbbm{1}_{B \cup \tilde{A}_m}.
	 \end{equation} 
	 In this case, $K$ can be approximated by the following sequence of simple function,
	\begin{equation}\label{tildeKm}
		\tilde{K}_m(\omega, t, \xi) = \mathbbm{1}_A(\omega) \mathbbm{1}_{[s, r)}(t) \mathbbm{1}_{B \cap \tilde{A}_m}(\xi) v.
	\end{equation}
		  Hence, in the following analysis, for simplicity, we assume that $B$ in  \eqref{eqn:simple K} is bounded. 

	 To obtain a desired approximation for $K$, we will find smooth approximations for the indicator functions  $\mathbbm{1}_{[s, r)}$ and $\mathbbm{1}_B$. \\
	\textbf{Step 2}: We start with $\mathbbm{1}_B$. Since $B$ is bounded, $\nu(B) < \infty$, thanks to \cite[Chapter II, Theorem 1.2]{parthasarathy2005probability}, for any $\veps > 0$, we might choose $B_1, B_2 \in \mathcal{E}$ such that 
	\begin{enumerate}[(1)]
		\item $B_1 \subset B \subset B_2$, 
		\item $B_1$ is closed and $B_2$ is open,
		\item $\nu(B_2 \setminus B_1) < \veps$.
	\end{enumerate}
	We can define the function 
	\begin{equation*} \label{eqn:B mollifier}
		f_B(\xi) := \frac{d(\xi, B_2^c)}{d(\xi, B_1) + d(\xi, B_2^c)}. 
	\end{equation*}
We can then verify that $f_B: E \to [0,1]$ is continuous, $f_B = 1$ on $B_1$ and $f_B = 0$ outside of $B_2$. \\
	\textbf{Step 3}: Next, we approximate the time indicator $\mathbbm{1}_{(s, r]}$. For $\veps > 0$, we define a continuous function $h_{s,r}: [0, T] \to [0, 1]$ 
	such that $h(t) = 0$ for $t \in [0, s]$, increases linearly to $1$ on $(s, s+\veps]$, equals $1$ on $(s+\veps, r]$, and decreases linearly to $0$ on $(r, (r+\veps) \wedge T]$.  \\
	\textbf{Step 4}: We verify that 
	\begin{equation*}  \label{indicator F}
		F(\omega, t, \xi) :=  \mathbbm{1}_A(\omega) h_{s,r}(t) f_B(\xi) v
	\end{equation*}
	satisfies the requirement. We observe that $F \in \mathbf{F}_{\nu,T}^1(V) \cap \mathbf{F}_{\nu,T}^2(V)$, and $F(\omega, \cdot, \cdot)$ is a continuous, bounded functions with bounded support. Next, we proceed to check \eqref{eqn:K - F approximation}. We recall that for $K$ and $F$ both have supports in $B_2$, $dt\otimes d\P$-a.e.. Moreover, 
	\begin{equation*}
		K(t, \xi) \neq F(t, \xi) \text{ only if $t \in (s, s+\veps] \cup (r, (r+\veps) \wedge T]$ or $\xi \in B_2 \setminus B_1$, $\P$-a.s.}
	\end{equation*}
	It follows that,
	\begin{equation*}
		\begin{aligned}
			& \E \int_0^T \int_{E_0} |K(t,\xi) - F(t,\xi)|_V^2 d\nu(\xi) dt = \E \int_0^T \int_{ E_0\cap B_2} |K(t,\xi) - F(t,\xi)|_V^2 d\nu(\xi) dt \\
			& \leq \E\int_{(s, s+\veps] \cup (r, (r+\veps) \wedge T]} \int_{ E_0\cap B_2} 4|v|_V^2 d\nu(\xi)dt + \E\int_{0}^T \int_{ (E_0\cap B_2) \setminus B_1} 4|v|_V^2 d\nu(\xi)dt \\
			&\leq (8\nu({ E_0\cap B_2}) + 4 T)\veps |v|_V^2.
		\end{aligned}
	\end{equation*}
	Since $\nu(B_2 \setminus B_1) < \veps$ by construction, and $\nu(B) < \infty$, we obtain that
	\begin{equation*}
		\nu(E_0 \cap B_2) \leq \nu(B_2) \leq \nu(B) + \nu(B_2 \setminus B_1) < \infty.
	\end{equation*} 
	By choosing $\displaystyle \veps < \frac{\delta}{(8\nu({ E_0\cap B_2}) + 4T)|v|_V^2}$, we conclude that  \eqref{eqn:K - F approximation} holds, thus completing the proof for the case $K$ is simple.
	
	\textbf{Step 5}: Next, we establish the result for arbitrary $K \in \mathbf{F}_{\nu, T}^2(V)$. Recalling that simple functions are dense in $\mathbf{F}_{\nu, T}^2(V)$, for any $\delta$ there exists $n\in\mathbb{N}$ such that
	\begin{equation*}
		\E \int_0^T \int_{E_0} \bigg|K(t,\xi) - \sum_{j=1}^n K_j (t,\xi) \bigg|_V^2 d\nu(\xi) dt < \frac{\delta}{2},
	\end{equation*}
	where 
	\begin{equation} \label{Kj}
		K_j(\omega, t, \xi) := \mathbbm{1}_{D_j}(\omega) \mathbbm{1}_{[s_j, r_j)}(t) \mathbbm{1}_{B_j}(\xi) v_j, \quad \text{for some $D_j \in \mathcal{F}_{s_j}$, $B_j \in \mathcal{E}$, and $v_j \in V$}.
	\end{equation}
	Here the measurable sets $B_j$ may be unbounded. However, in that case we can use the characterization \eqref{B} and approximate each $K_j$ with $\tilde K_m$ given in \eqref{tildeKm}. To avoid excessive notation, for simplicity we thus assume that each $B_j$ in \eqref{Kj} is bounded.
	
	Since each $K_j$ is simple, following the argument given in Steps 1--4 we conclude that there exist $F_j \in \mathbf{F}_{\nu,T}^1(V) \cap \mathbf{F}_{\nu,T}^2(V)$ such that 
	\begin{equation*}
		F_j(\omega,t,\xi) =  \mathbbm{1}_{D_j}(\omega) h_{s_j, r_j}(t) f_{B_j}(\xi) v_j, 
	\end{equation*}
	where $h_{s_j, r_j}$, $f_{B_j}$ satisfy  \eqref{h, f requirement}. Moreover, 
	\begin{equation*}
		\E \int_0^T \int_{E_0} |K_j(t,\xi) - F_j(t,\xi)|_V^2 d\nu(\xi) dt < \frac{\delta}{2n}.
	\end{equation*}
	Finally, define
	\begin{equation*}
		F(\omega, t, \xi) := \sum_{j=1}^n F_j(\omega, t, \xi).
	\end{equation*}
	Then it follows that $F \in \mathbf{F}_{\nu,T}^1(V) \cap \mathbf{F}_{\nu,T}^2(V)$ and by the triangle inequality we have
	\begin{equation*}
		\begin{aligned}
		\E \int_0^T \int_{E_0} |K(t,\xi) - F(t,\xi)|_V^2 d\nu(\xi) &dt  \leq \E \int_0^T \int_{E_0} \bigg|K(t,\xi) -  \sum_{j=1}^n K_j(t,\xi)\bigg|_V^2 d\nu(\xi) dt \\
		& + \sum_{j=1}^n 	\E \int_0^T \int_{E_0} |K_j(t,\xi) - F_j(t,\xi)|_V^2 d\nu(\xi) dt  < \frac{\delta}{2} + n\cdot \frac{\delta}{2n} = \delta.
	\end{aligned}
	\end{equation*}
\end{proof}

Next, we finally give a proof of Theorem \ref{lemma:convergence appendix}. 
\begin{proof}[Proof of Theorem \ref{lemma:convergence appendix}]

	 	Consider $K \in \mathbf{F}_{\nu, T}^2(V)$.	Then for any $\delta, \veps > 0$, we know, due to Lemma \ref{lemma:appendix 2}, that there exists $F \in \mathbf{F}_{\nu, T}^1(V) \cap \mathbf{F}_{\nu, T}^2(V)$ which is almost surely continuous, bounded and has bounded support such that
	 	\begin{equation} \label{eqn:appendix approximation to K}
	 		\E \int_0^T \int_{E_0} |K(t,\xi) - F(t,\xi)|_V^2 d\nu(\xi) dt < \delta \veps^2.
	 	\end{equation}
	 Moreover, the function $F$ takes the following form: 
	 \begin{equation} \label{simple F}
	 	F(\omega, t, \xi) := \sum_{j=1}^m F_j(\omega, t, \xi) :=  \sum_{j=1}^m  \mathbbm{1}_{D_j}(\omega) h_{s_j, r_j}(t)f_{B_j}(\xi)  v_j, \quad  D_j \in \mathcal{F}_{s_j}, v_j \in V,
	 \end{equation}
	 where for each $j=1,..,m$, the function $F_j\in\mathbf{F}_{\nu,T}^1(V) \cap \mathbf{F}_{\nu,T}^2(V)$, the functions $h_{s_j, r_j}:[0,T] \to [0,1]$ and $\tilde{f}_{B_j}: E \to [0,1]$ are both continuous, bounded and have bounded supports. 
	By rewriting
	\begin{equation*} \label{eqn:appendix split events}
		\begin{aligned}
			I_t^{\widehat{\pi}_N}(K_N) - I_t^{\widehat{\pi}}(K) & = I_t^{\widehat{\pi}_N}(K_N - K) + I_t^{\widehat{\pi}_N}(K - F)
			+ \biggl( I_t^{\widehat{\pi}_N}(F)  - I_t^{\widehat{\pi}}(F) \biggr) + I_t^{\widehat{\pi}}(F - K),
		\end{aligned}
	\end{equation*}
	we can split the event
	\begin{equation*}
		\begin{aligned}
			\P \biggl[ \int_0^T|_t^{\widehat{\pi}_N}(K_N) - I_t^{\widehat{\pi}}(K)|_V^2 dt > 4\veps  \biggr] & \leq \P \biggl[ \int_0^T | I_t^{\widehat{\pi}_N}(K_N - K)|_V^2 dt > \veps  \biggr] + \P \biggl[ \int_0^T | I_t^{\widehat{\pi}_N}(K - F)|_V^2 dt > \veps  \biggr] \\
			&  + \P \biggl[ \int_0^T| I_t^{\widehat{\pi}_N}(F)  - I_t^{\widehat{\pi}}(F) |_V^2 dt > \veps  \biggr]+ 
			\P \biggl[ \int_0^T| I_t^{\widehat{\pi}}(F - K) |_V^2 dt > \veps  \biggr] \\
			& =: P^N_1 + P^N_2 + P^N_3 + P_4.
		\end{aligned}
	\end{equation*}
	We will obtain estimates for $P^N_j; j = 1,2,3,4$ term-by-term starting with the term $P^N_1$. Since $K_N \to K$ in $L^2([0,T] \times E_0; dt \otimes d\nu; V)$, $\P$-a.s., we invoke Lemma \ref{lemma:appendix 1}, and conclude that there exists an $N$ such that for any $n \geq N$ we have
	\begin{equation*}
		P^n_1 < \delta. 
	\end{equation*}
	Next, due to Chebyshev's inequality, the It\^{o} isometry and \eqref{eqn:appendix approximation to K}, we obtain for the term $P_2^N$ that
	\begin{equation*}
		\begin{aligned}
			P^N_2 \leq \frac{\E \int_0^T \int_{E_0} |K(t,\xi) - F(t,\xi)|_V^2 d\nu(\xi) dt}{\veps^2} < \delta, \qquad \text{ for every }N\in\mathbb{N}.
		\end{aligned}
	\end{equation*}
The term $P_4$ is treated similarly and we have,
	\begin{equation*}
		\begin{aligned}
			P_4 < \delta.
		\end{aligned}
	\end{equation*}
	Finally, we turn our attention to the term $P^N_3$. {The fact that $F_j \in \mathbf{F}_{\nu,T}^1(V) \cap \mathbf{F}_{\nu,T}^2(V)$ allows us to write
		\begin{equation*}
			I_t^{\widehat{\pi}_N}(F) = I_t^{\pi_N}(F)  + \int_0^t \int_{E_0} F(s,\xi) d\nu(\xi) ds,\quad I_t^{\widehat{\pi}}(F) = I_t^{\pi}(F)  + \int_0^t \int_{E_0} F(s,\xi) d\nu(\xi) ds.
		\end{equation*}
	}Recalling that $\pi_N \to \pi$ almost surely in $\mathcal{N}_{[0,\infty) \times E}^\#$ endowed with the weak-$\#$ topology (see \eqref{weak hash conv} for definition), and that $\pi_N, \pi$ have intensity $dt \otimes d\nu$, we can conclude that $I_t^{\pi_N}(F) \to I_t^{\pi}(F)$ as $N \to \infty$, $dt\otimes d\P$-a.s. in $V$. Hence,
	\begin{equation}\label{hashconvergence}
		I_t^{\widehat{\pi}_N}(F) \to I_t^{\widehat{\pi}}(F), \quad \text{as $N \to \infty$, $dt\otimes d\P$-a.s. in $V$.}
	\end{equation}
	For $F$ given in $\eqref{simple F}$,  and the linearity of the integral operators $I_t^{\widehat{\pi}_N}$ and $I_t^{\widehat{\pi}}$, we have
	\begin{equation*} 
		I_t^{\widehat{\pi}_N}(F) = \sum_{j=1}^m I_t^{\widehat{\pi}_N}(F_j), \quad I_t^{\widehat{\pi}}(F) = \sum_{j=1}^m I_t^{\widehat{\pi}}(F_j).
	\end{equation*}
	Therefore, to prove bounds for the term $P_3^N$, it suffices to show that for each $j = 1, \cdots, m$ we have, 
	\begin{equation} \label{F_j conv}
		\int_0^T |I_t^{\widehat{\pi}_N}(F_j)|_V^2 dt \to \int_0^T | I_t^{\widehat{\pi}}(F_j)|_V^2 dt, \text{ as $N \to \infty$, $\P$-a.s.}
	\end{equation}
	We note that the function $h_{s_j, r_j}(t)\tilde{f}_{B_j}(\xi)$ is continuous on $[0,T] \times E$. Hence, using the same ideas as in \eqref{hashconvergence} and invoking the definition of $\pi_N \to \pi$ in $\mathcal{N}_{[0,\infty) \times E}^\#$ endowed with the weak-$\#$ topology we obtain that 
	\begin{multline*}
\mathbbm{1}_{D_j}	v_j	\int_0^T \int_{E_0} h_{s_j, r_j}(t) \tilde{f}_{B_j}(\xi) d\widehat{\pi}_N(t,\xi) \to  \mathbbm{1}_{D_j} v_j \int_0^T \int_{E_0} h_{s_j, r_j}(t) \tilde{f}_{B_j}(\xi) d\widehat{\pi}(t,\xi) \\
\text{ as $N \to \infty$, $\P$-a.s. in $V$.}
	\end{multline*}
	Hence for almost every $\omega\in\Omega$ there exists some $M = M(\omega)> 0 $ such that
	\begin{equation*}
		\sup_{N \in \N}|I_t^{\widehat{\pi}_N}(F_j)|_V \leq  |v_j|_V \sup_{N \in \N} \bigg| \int_0^T \int_{E_0} h_{s_j, r_j}(t) \tilde{f}_{B_j}(\xi) d\widehat{\pi}_N(t,\xi) \bigg| \leq M |v_j|_V,\quad \text{ $\P$-a.s.}
	\end{equation*}
Hence \eqref{F_j conv} follows from the previous two equations and an application of the Dominated Convergence Theorem. We thus arrive at
	\begin{equation*}
		\int_0^T |I_t^{\widehat{\pi}_N}(F)|_V^2 dt \to \int_0^T | I_t^{\widehat{\pi}}(F)|_V^2 dt, \text{ as $N \to \infty$, $\P$-a.s.}
	\end{equation*}
	This tells us that for sufficiently large $N$, we have
	\begin{equation*}
		P^N_3 < \delta.
	\end{equation*}
Collecting the estimates for $P^N_j; j = 1,2,3,4$ ends the proof.
	\end{proof}

\section*{Acknowledgment}
K. Tawri gratefully acknowledges the support by the National Science Foundation grant DMS-2553666 (formerly DMS-2407197). X. Yang gratefully acknowledges the support by the Research Fund of Indiana University.

\bibliographystyle{plain}
\bibliography{Reference.bib}

@article{DGT11,
  title={Local martingale and pathwise solutions for an abstract fluids model},
  author={Debussche, Arnaud and Glatt-Holtz, Nathan and Temam, Roger},
  journal={Physica D: Nonlinear Phenomena},
  volume={240},
  number={14-15},
  pages={1123--1144},
  year={2011},
  publisher={Elsevier}
}

@book{parthasarathy2005probability,
  title={Probability measures on metric spaces},
  author={Parthasarathy, Kalyanapuram Rangachari},
  volume={352},
  year={2005},
  publisher={American Mathematical Soc.}
}

@article{jakubowski1998almost,
  author={Jakubowski, Adam},
  title={The almost sure {Skorokhod} representation for subsequences in nonmetric spaces},
  journal={Theory of Probability \& Its Applications},
  volume={42},
  number={1},
  pages={167--174},
  year={1998},
  publisher={SIAM}
}

@book{daley2008introduction,
  title={An introduction to the theory of point processes: volume {II}: general theory and structure},
  author={Daley, Daryl J and Vere-Jones, David},
  year={2008},
  publisher={Springer}
}

@article{KPV96,
  title={A bilinear estimate with applications to the {KdV} equation},
  author={Kenig, Carlos and Ponce, Gustavo and Vega, Luis},
  journal={Journal of the American Mathematical Society},
  volume={9},
  number={2},
  pages={573--603},
  year={1996}
}

@article{chang1986,
doi = {10.1088/0741-3335/28/4/005},
url = {https://doi.org/10.1088/0741-3335/28/4/005},
year = {1986},
month = {apr},
publisher = {},
volume = {28},
number = {4},
pages = {675},
author = {Hong-Young Chang and Chuong Lien and S Sukarto and S Raychaudhuri and J Hill and E K Tsikis and K E Lonngren},
title = {Propagation of ion-acoustic solitons in a non-quiescent plasma},
journal = {Plasma Physics and Controlled Fusion}
}

@article{SRC1998,
  title={{Korteweg--de Vries} solitons under additive stochastic perturbations},
  author={Scalerandi, Marco and Romano, A and Condat, CA},
  journal={Physical Review E},
  volume={58},
  number={4},
  pages={4166},
  year={1998},
  publisher={APS}
}

@article{Herman1990,
doi = {10.1088/0305-4470/23/7/014},
url = {https://doi.org/10.1088/0305-4470/23/7/014},
year = {1990},
month = {apr},
publisher = {},
volume = {23},
number = {7},
pages = {1063},
author = {R L Herman},
title = {The stochastic, damped {KdV} equation},
journal = {Journal of Physics A: Mathematical and General}
}

@article{Bourgain93,
  title={{Fourier transform restriction phenomena for certain lattice subsets and applications to nonlinear evolution equations: Part II: The KdV-equation}},
  author={Bourgain, Jean},
  journal={Geometric \& Functional Analysis GAFA},
  volume={3},
  number={3},
  pages={209--262},
  year={1993},
  publisher={Springer}
}

@article{CKSTT03,
  title={{Sharp global well-posedness for {KdV} and modified {KdV} on $\mathbb{R}$ and $\mathbb{T}$}},
  author={Colliander, James and Keel, Markus and Staffilani, Gigliola and Takaoka, Hideo and Tao, Terence},
  journal={Journal of the American Mathematical Society},
  volume={16},
  number={3},
  pages={705--749},
  year={2003}
}

@article{Ondrejt04,
  title={Uniqueness for stochastic evolution equations in {Banach} spaces},
  author={Martin Ondrej{\'a}t},
  journal={Dissertationes Mathematicae},
  year={2004},
  volume={426},
  pages={1-63},
  url={https://api.semanticscholar.org/CorpusID:120245797}
}

@article{motyl2013,
  title={{Stochastic {Navier--Stokes} equations driven by {L{\'e}vy} noise in unbounded {3D} domains}},
  author={Motyl, El{\.z}bieta},
  journal={Potential Analysis},
  volume={38},
  number={3},
  pages={863--912},
  year={2013},
  publisher={Springer}
}

@article{BMO17,
author = {Zdzisław Brzeźniak and Elżbieta Motyl and Martin Ondrejat},
title = {{Invariant measure for the stochastic {Navier–-Stokes} equations in unbounded {2D} domains}},
volume = {45},
journal = {The Annals of Probability},
number = {5},
publisher = {Institute of Mathematical Statistics},
pages = {3145 -- 3201},
year = {2017},
doi = {10.1214/16-AOP1133},
URL = {https://doi.org/10.1214/16-AOP1133}
}

@article{RM23,
  title={Nonlinear {SPDE} driven by {L\'{e}vy} noise: Well-posedness, optimal control and invariant measure},
  author={Kavin R and Ananta K. Majee},
  journal={arXiv preprint arXiv:2306.04303},
  year={2023}
}

@article{Kato79,
  title={{On the {Korteweg--de Vries} equation}},
  author={Kato, Tosio},
  journal={Manuscripta mathematica},
  volume={28},
  number={1},
  pages={89--99},
  year={1979},
  publisher={Springer}
}

@article{BS75,
  title={{The initial-value problem for the {Korteweg-de Vries} equation}},
  author={Bona, Jerry L and Smith, Ronald},
  journal={Philosophical Transactions of the Royal Society of London. Series A, Mathematical and Physical Sciences},
  volume={278},
  number={1287},
  pages={555--601},
  year={1975},
  publisher={The Royal Society London}
}

@article{ST76,
  title={Remarks on the {Korteweg--de Vries} equation},
  author={Saut, Jean-Claude and Temam, Roger},
  journal={Israel Journal of Mathematics},
  volume={24},
  pages={78--87},
  year={1976},
  publisher={Springer}
}

@article{BS76,
author = {Jerry Bona and Ridgway Scott},
title = {Solutions of the {Korteweg--de Vries} equation in fractional order {Sobolev} spaces},
volume = {43},
journal = {Duke Mathematical Journal},
number = {1},
publisher = {Duke University Press},
pages = {87 -- 99},
year = {1976},
doi = {10.1215/S0012-7094-76-04309-X},
URL = {https://doi.org/10.1215/S0012-7094-76-04309-X}
}

@article{Ghi88,
  title={Weakly damped forced {Korteweg--de Vries} equations behave as a finite dimensional dynamical system in the long time},
  author={Ghidaglia, Jean-Michel},
  journal={Journal of Differential Equations},
  volume={74},
  number={2},
  pages={369--390},
  year={1988},
  publisher={Elsevier}
}

@article{MR97,
author = {I. Moise and R. Rosa},
title = {{On the regularity of the global attractor of a weakly damped, forced {Korteweg--de Vries} equation}},
volume = {2},
journal = {Advances in Differential Equations},
number = {2},
publisher = {Khayyam Publishing, Inc.},
pages = {257 -- 296},
year = {1997},
doi = {10.57262/ade/1366809216},
URL = {https://doi.org/10.57262/ade/1366809216}
}

@article{DeBussche98,
  title={On the stochastic {Korteweg--de Vries} equation},
  author={de Bouard, Anne and Debussche, Arnaud},
  journal={journal of functional analysis},
  volume={154},
  number={1},
  pages={215--251},
  year={1998},
  publisher={Elsevier}
}

@article{dBDT99,
  title={White noise driven {Korteweg--de Vries} equation},
  author={de Bouard, Anne and Debussche, Arnaud and Tsutsumi, Yoshio},
  journal={Journal of Functional Analysis},
  volume={169},
  number={2},
  pages={532--558},
  year={1999},
  publisher={Elsevier}
}

@article{dBDT04,
  title={Periodic Solutions of the {Korteweg--de Vries} Equation Driven by White Noise},
  author={de Bouard, Anne and Debussche, Arnaud and Tsutsumi, Yoshio},
  journal={SIAM journal on mathematical analysis},
  volume={36},
  number={3},
  pages={815--855},
  year={2005},
  publisher={SIAM}
}

@article{GK96,
	title={Existence of strong solutions for {It{\^o}'s} stochastic equations via approximations},
	author={Gy{\"o}ngy, Istv{\'a}n and Krylov, Nicolai},
	journal={Probability theory and related fields},
	volume={105},
	number={2},
	pages={143--158},
	year={1996},
	publisher={Springer}
}

@article{YW71,
	title={On the uniqueness of solutions of stochastic differential equations},
	author={Yamada, Toshio and Watanabe, Shinzo},
	journal={Journal of Mathematics of Kyoto University},
	volume={11},
	number={1},
	pages={155--167},
	year={1971},
	publisher={Duke University Press}
}

@book{peszat2007stochastic,
  title={Stochastic partial differential equations with {L{\'e}vy} noise: An evolution equation approach},
  author={Peszat, Szymon and Zabczyk, Jerzy},
  volume={113},
  year={2007},
  publisher={Cambridge University Press}
}

@book{da2014stochastic,
  title={Stochastic equations in infinite dimensions},
  author={Da Prato, Giuseppe and Zabczyk, Jerzy},
  volume={152},
  year={2014},
  publisher={Cambridge university press}
}

@article{GMR17,
  title={On unique ergodicity in nonlinear stochastic partial differential equations},
  author={Glatt-Holtz, Nathan and Mattingly, Jonathan C and Richards, Geordie},
  journal={Journal of Statistical Physics},
  volume={166},
  pages={618--649},
  year={2017},
  publisher={Springer}
}

@article{GMR21,
  title={On the long-time statistical behavior of smooth solutions of the weakly damped, stochastically-driven {KdV} equation},
  author={Glatt-Holtz, Nathan and Martinez, Vincent R and Richards, Geordie H},
  journal={arXiv preprint arXiv:2103.12942},
  year={2021}
}

@article{EKZ18,
  title={Existence of invariant measures for the stochastic damped {KdV} equation},
  author={Ekren, Ibrahim and Kukavica, Igor and Ziane, Mohammed},
  journal={Indiana University Mathematics Journal},
  pages={1221--1254},
  year={2018},
  publisher={JSTOR}
}

@article{Tem69,
  title={Sur un probl\'{e}me non lin\'{e}aire},
  author={Temam, Roger},
  volume={48},
  pages={159--172},
  year={1969},
  journal={J. Math. Pures Appl}
}

@book{Tem97,
	title     = {Infinite-Dimensional Dynamical Systems in Mechanics and Physics},
	author    = {Temam, Roger},
	year      = {1997},
	publisher = {Springer New York, NY},
	series    = {Applied Mathematical Sciences},
	volume    = {68},
	edition   = {2},
	doi       = {10.1007/978-1-4612-0645-3},
	pages     = {XXII, 650}
}

@book{ikeda2014stochastic,
  title={Stochastic differential equations and diffusion processes},
  author={Ikeda, Nobuyuki and Watanabe, Shinzo},
  year={2014},
  publisher={Elsevier}
}

@article {CTT18euler,
    AUTHOR = {Cyr, Justin and Tang, Sisi and Temam, Roger},
     TITLE = {The {E}uler equations of an inviscid incompressible fluid
              driven by a {L}\'{e}vy noise},
   JOURNAL = {Nonlinear Anal. Real World Appl.},
  FJOURNAL = {Nonlinear Analysis. Real World Applications. An International
              Multidisciplinary Journal},
    VOLUME = {44},
      YEAR = {2018},
     PAGES = {173--222},
}

@article{Aldous78,
	author = {Aldous, David},
	title = {{Stopping Times and Tightness}},
	volume = {6},
	journal = {The Annals of Probability},
	number = {2},
	publisher = {Institute of Mathematical Statistics},
	pages = {335 -- 340},
	year = {1978},
	doi = {10.1214/aop/1176995579},
	URL = {https://doi.org/10.1214/aop/1176995579}
}

@article{flandoli1995martingale,
  title={Martingale and stationary solutions for stochastic {Navier--Stokes} equations},
  author={Flandoli, Franco and Gatarek, Dariusz},
  journal={Probability Theory and Related Fields},
  volume={102},
  pages={367--391},
  year={1995},
  publisher={Springer}
}

@article{CTT18compare,
  title={A comparison of two settings for stochastic integration with respect to {L{\'e}vy} processes in infinite dimensions},
  author={Cyr, Justin and Tang, Sisi and Temam, Roger},
  journal={Trends in Applications of Mathematics to Mechanics},
  pages={289--373},
  year={2018},
  publisher={Springer}
}

@article{CNTT20,
  title={Review of local and global existence results for stochastic {PDEs} with {L{\'e}vy} noise},
  author={Cyr, Justin and Nguyen, Phuong and Tang, Sisi and Temam, Roger},
  journal={Discrete \& Continuous Dynamical Systems-A},
  volume={40},
  number={10},
  pages={5639},
  year={2020},
  publisher={American Institute of Mathematical Sciences}
}

@book{billingsley2013convergence,
  title={Convergence of probability measures},
  author={Billingsley, Patrick},
  year={2013},
  publisher={John Wiley \& Sons}
}

@article{NTT21,
  title={Nonlinear stochastic parabolic partial differential equations with a monotone operator of the {Ladyzenskaya--Smagorinsky} type, driven by a {L{\'e}vy} noise},
  author={Nguyen, Phuong and Tawri, Krutika and Temam, Roger},
  journal={Journal of Functional Analysis},
  volume={281},
  number={8},
  pages={109157},
  year={2021},
  publisher={Elsevier}
}

@article{NaMo09,
author = {Nathan Glatt-Holtz and Mohammed Ziane},
title = {Strong pathwise solutions of the stochastic {Navier--Stokes}system},
volume = {14},
journal = {Advances in Differential Equations},
number = {5/6},
publisher = {Khayyam Publishing, Inc.},
pages = {567 -- 600},
year = {2009},
doi = {10.57262/ade/1355867260},
URL = {https://doi.org/10.57262/ade/1355867260}
}

@article{MaSe99,
author = {Maslowski, Bohdan and Seidler, Jan},
journal = {Atti della Accademia Nazionale dei Lincei. Classe di Scienze Fisiche, Matematiche e Naturali. Rendiconti Lincei. Matematica e Applicazioni},
language = {eng},
month = {6},
number = {2},
pages = {69-78},
publisher = {Accademia Nazionale dei Lincei},
title = {On sequentially weakly {Feller} solutions to {SPDE}’s},
url = {http://eudml.org/doc/252296},
volume = {10},
year = {1999},
}

\end{document}